\documentclass[preprint,3p,times]{elsarticle}
\usepackage[toc,page,title,titletoc,header]{appendix}
\usepackage{amsmath}
\allowdisplaybreaks[4]
\usepackage{amssymb}
\usepackage{amsthm}
\usepackage{enumerate}
\usepackage{enumitem}
\usepackage{mathrsfs}
\usepackage{graphicx}
\usepackage{epsfig}
\usepackage{tikz,pgfplots,tikz-3dplot}
\usepackage{epstopdf}
\usetikzlibrary{fit}
\biboptions{sort&compress}

\newtheorem{thm}{Theorem}[section]

\newtheorem{rem}{Remark}
\newtheorem{exa}{Example}

\newcommand{\bR}{{\mathbb R}}

\newtheorem{myDef}{Definition}[section]

\newcommand{\bI}{\mathbb{I}}

\newcommand{\nn}{\nonumber}

\newcommand{\mS}{\mathcal{S}}

\def\epsilon{\varepsilon} 
\newcommand{\mat}[1]{\boldsymbol{#1}}

\begin{document}
\begin{frontmatter}
\title{
{Error analysis for temporal second-order finite element approximations of axisymmetric mean curvature flow of genus-1 surfaces}
}

\author[1]{Meng Li}
\address[1]{School of Mathematics and Statistics, Zhengzhou University,
Zhengzhou 450001, China}
\ead{limeng@zzu.edu.cn}
\author[1]{Lining Wang}
\author[1]{Yiming Wang}

\begin{abstract}
Existing studies on the convergence of numerical methods for curvature flows primarily focus on first-order temporal schemes. In this paper, we establish a novel error analysis for parametric finite element approximations of genus-1 axisymmetric mean curvature flow, formulated using two classical second-order time-stepping methods: the Crank-Nicolson method and the BDF2 method.
Our results establish optimal error bounds in both the \(L^2\)-norm and \(H^1\)-norm, along with a superconvergence result in the \(H^1\)-norm for each fully discrete approximation. Finally, we perform convergence experiments to validate the theoretical findings and present numerical simulations for various genus-1 surfaces. Through a series of comparative experiments, we also demonstrate that the methods proposed in this paper exhibit significant mesh advantages.
\end{abstract} 


\end{frontmatter}

\section{Introduction}\setcounter{equation}{0}
Mean curvature flow is one of the most fundamental and widely studied geometric evolution equations, governing the motion of surfaces driven by their mean curvature. 
This flow naturally arises in various physical and geometric contexts, including the evolution of soap films, surface smoothing in image processing, and phase transition modeling in materials science. 
Let $\{\mS(t)\}_{t\in[0,T]}\subset\bR^3$ be a family of smooth, oriented, and closed hypersurfaces. The motion by mean curvature flow is given by
\begin{equation}\label{eq:model1}
	\mathcal{V}_{\mS} = \kappa_m, 
\end{equation}
where $\mathcal{V}_{\mS}$ denotes the velocity of the surface $\mS(t)$ in the direction of the unit normal $\mat n_{_\mS}$, and $\kappa_m$ is the mean curvature of $\mS(t)$, defined as the sum of its principal curvatures. For a comprehensive introduction to mean curvature flow and its key results, we refer to \cite{mantegazza2011lecture}.
Over the past few decades, a wide range of numerical methods have been developed for approximating mean curvature flow. The use of parametric finite element methods (FEMs) for two-dimensional surfaces traces back to the pioneering work of Dziuk \cite{Dziuk90}. 
Since then, various alternative approaches have been proposed, including those in \cite{Barrett08JCP, m2017approximations} and the references therein.

Despite significant advancements in numerical methods for mean curvature flow and related flows, convergence analysis remains highly challenging. 
The convergence of certain semidiscrete and fully discrete parametric FEMs for the mean curvature flow and Willmore flow of curves has been established by Dziuk \cite{dziuk1994convergence}, Deckelnick and Dziuk \cite{deckelnick1995approximation,deckelnick2009error}, Bartels \cite{bartels2013simple}, Li \cite{li2020convergence,bai2024convergence}, and Ye and Cui \cite{ye2021convergence}, among others.
Furthermore, the convergence of numerical schemes for the mean curvature flow and Willmore flow of closed surfaces has been investigated by Dziuk and Elliott \cite{dziuk2007finite}, Kov{'a}cs et al. \cite{kovacs2017convergence,Kovacs2019,kovacs2021convergent}, Li \cite{li2021convergence}, Barrett et al. \cite{barrett2021finite}, Deckelnick and N{\"u}rnberg \cite{deckelnick2021error}, Elliott et al. \cite{elliott2022numerical}, Hu and Li \cite{Hu22evolving}, Bai and Li \cite{bai2023erratum}, and Li \cite{li2025error}, among others.
We note that existing convergence analyses for the fully discrete schemes are largely limited to first-order time-stepping methods. 
Very recently, the first second-order error analysis for curve shortening flow and curve diffusion was presented in \cite{deckelnick2025second}. 
In this work, we focus on constructing and analyzing the convergence of second-order parametric FEMs, including the Crank-Nicolson (CN) method and the two-step backward differentiation formula (BDF2) method, for a special type of three-dimensional mean curvature flow. This represents a significant innovation and contribution to the field. It is worth emphasizing that our work appeared about one month after the recent study \cite{deckelnick2025second}; however, the numerical methods and the problems addressed in the two works are fundamentally different.

\begin{figure}[!htp]
	\centering
	\includegraphics[width=0.9\textwidth]{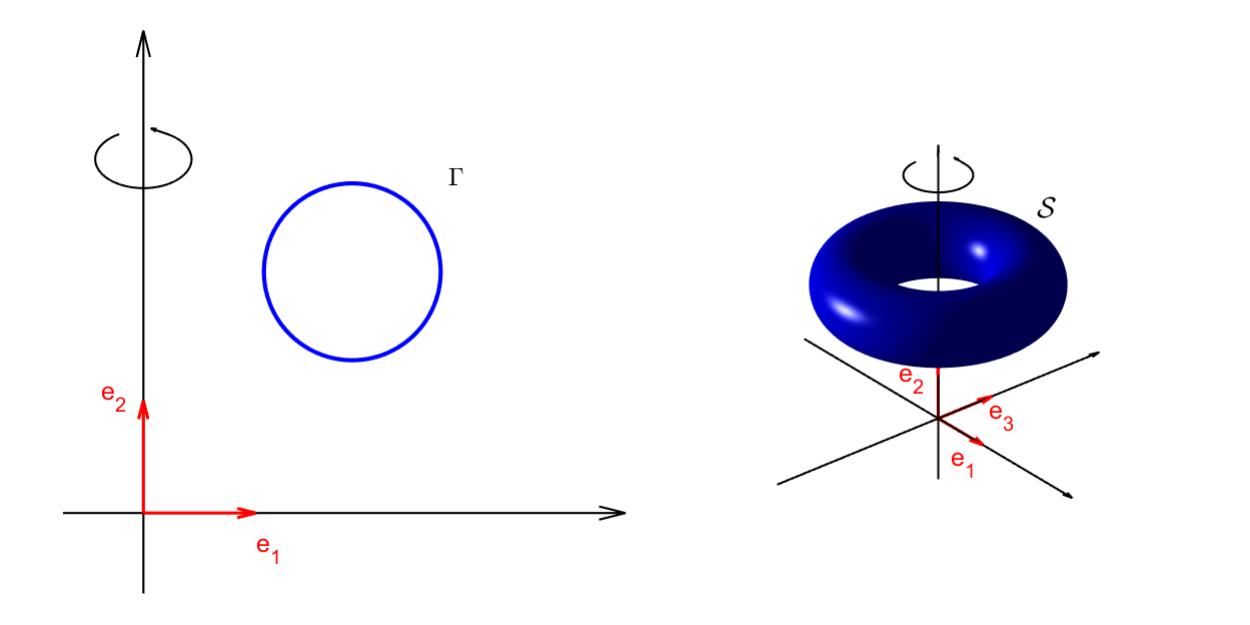}
	\caption{Sketch of $\Gamma$ and $\mS$, as well as the unit vectors $\mat e_1$, $\mat e_2$ and $\mat e_3$.}
	\label{figure:torus}
\end{figure}

In many practical scenarios, evolving three-dimensional surfaces often exhibit rotational symmetry (see Fig. \ref{figure:torus}). 
This symmetry property enables a significant simplification of geometric flow by reducing the problem to a one-dimensional setting, as demonstrated in several studies \cite{Bernoff1998, Deckelnick03, Zhao19, Barrett19, Barrett2019variational, Barrett2021stable, BGNZ21volume}. Such a reduction not only drastically reduces computational complexity but also eliminates the need for sophisticated mesh control techniques, as the focus shifts to the one-dimensional generating curve of the axisymmetric surfaces. In this work, due to theoretical limitations, we only consider the mean curvature flow with genus-1 axisymmetric structure. Specifically, we denote $\bI=\mathbb R/\mathbb Z$ with $\partial \bI=\emptyset$. Let $\mat x(t): \overline{\bI}\rightarrow \mathbb R_{>0}\times \mathbb R$
 parameterize $\Gamma(t)$, which is a generating curve of a torus surface $\mS(t)$ that is axisymmetric with respect to the $x_2$-axis. The mean curvature flow  \eqref{eq:model1} with genus-1 axisymmetry can be formulated as 
 \begin{align}\label{eqn:model2} \mat x_t\cdot\mat{\nu}=\varkappa-\frac{\mat{\nu}\cdot\mat e_1}{\mat x\cdot\mat e_1},\quad \varkappa=\frac{1}{|\mat x_\rho|}\left(\frac{\mat x_\rho}{|\mat x_\rho|}\right)_\rho\cdot\mat{\nu}\qquad \text{in}\quad \bI\times (0, T],
 \end{align}
where $\mat{\nu}$ denotes the outer unit normal vector of the curve $\Gamma(t)$. It is obvious that the system 
\begin{align}\label{eqn:model3}
	\mat x_t=\frac{1}{|\mat x_\rho|}\left(\frac{\mat x_\rho}{|\mat x_\rho|}\right)_\rho-\frac{\mat{\nu}\cdot\mat e_1}{\mat x\cdot\mat e_1}\mat{\nu}
	=\frac{\mat x_{\rho\rho}-\left(\mat x_{\rho\rho}\cdot\mat{\tau}\right)\mat{\tau}}{|\mat x_\rho|^2}-\frac{\mat{\nu}\cdot\mat e_1}{\mat x\cdot\mat e_1}\mat{\nu}
\end{align}
satisfies \eqref{eqn:model2}, where $\mat{\tau}$ is the tangential vector of $\Gamma(t)$. Furthermore, since $\mat x_t\cdot\mat{\tau}=0$, the system \eqref{eqn:model3} is degenerate in the tangential direction, and 
the mesh quality of the corresponding discretization may deteriorate over time. 
One way to address this problem is to design a scheme based on the 
DeTurck’s trick \cite{deckelnick1995approximation,m2017approximations}. 
Specifically, we introduce an additional tangential motion to remove the degeneracy. The updated system is 
\begin{align}\label{eqn:model4}
	\mat x_t-\frac{\mat x_{\rho\rho}}{|\mat x_\rho|^2}+\frac{\mat{\nu}\cdot\mat e_1}{\mat x\cdot\mat e_1}\mat{\nu}=\mat 0, 
\end{align}
which is equivalent to the following system: 
\begin{align}\label{eqn:model}
	\mat x\cdot \mat e_1|\mat x_\rho|^2\mat x_t-\left(\mat x\cdot\mat e_1\mat x_\rho\right)_\rho+\left|\mat x_\rho\right|^2\mat e_1 =\mat 0. 
\end{align}
In a previous study by Barrett et al. \cite{barrett2021finite}, the authors established detailed convergence analysis for the temporal first-order numerical approximation of the axisymmetric system \eqref{eqn:model}.
 Building upon their work, we in this  study develop two types of temporal second-order approximations for the system \eqref{eqn:model} and rigorously analyze the convergence of the fully discretized schemes. 
 We obtain optimal convergence results in both the $L^2$-norm and the 
 $H^1$-norm, as well as establish a superconvergence result in the sense of $H^1$-norm. 
 Recently, several second-order time-stepping methods have been proposed for solving curvature flow problems \cite{jiang2024second,Jiang24,li2024efficient,guo2025structure}; however, none of these works have provided a convergence analysis.

The outline of the paper is as follows. In Section \ref{sec2}, we build the temporal second-order CN method and BDF2 method, and also present the main convergence results of this paper. 
In Section \ref{sec3}, we provide a detailed proof of the error estimate for the CN method. In Section \ref{sec4}, we supply the error estimate for the BDF2 method. 
In Section \ref{sec5}, we conduct numerical experiments to validate the robustness and accuracy of the proposed numerical schemes, as well as to explore interesting phenomena in differential geometry. 
To check the mesh quality of the CN method and the BDF2 method, we also conduct a series of comparative experiments in this section.
Finally, in Section \ref{sec6}, we summarize our findings and draw conclusions.
In the following sections, we denote the \( L^2 \)-inner product on \( \bI \) by \( (\cdot, \cdot) \). For \( l \in \mathbb{N}_0 \) and \( p \in [1, \infty] \), we use \( W^{l, p}(\bI) \) to represent the Sobolev space equipped with the norm \( \|\cdot\|_{W^{l, p}} \) and the seminorm \( |\cdot|_{W^{l, p}} \). In the special case of \( p = 2 \), we adopt the simplified notation \( H^l(\bI) := W^{l, 2}(\bI) \), with the corresponding norm and seminorm given by \( \|\cdot\|_{H^l} := \|\cdot\|_{W^{l, 2}} \) and \( |\cdot|_{H^l} := |\cdot|_{W^{l, 2}} \), respectively. For convenience, we also define vector-valued Sobolev spaces as \( \mat W^{l, p}(\bI) := \left[W^{l, p}(\bI)\right]^2 \) and \( \mat H^l(\bI) := \left[H^l(\bI)\right]^2 \). During the theoretical analysis, \( C > 0 \) is a bounded constant independent of the time step \( \Delta t \) and the spatial step \( h \), and it may have different values in different places.

\section{Temporal second-order schemes and main results}
\label{sec2}
\setcounter{equation}{0}

A weak formulation for \eqref{eqn:model} is given as follows: for $\mat x(0)\in \mat H^1(\bI)$, to find $\mat x(t)\in \mat H^1(\bI)$, such that 
\begin{align}\label{eqn:wf}
	\left(\mat x\cdot \mat e_1\mat x_t, \mat \eta\left|\mat x_\rho\right|^2\right)
	+\left(\mat x\cdot\mat e_1\mat x_\rho, \mat \eta_\rho\right)
	+\left(\mat \eta\cdot\mat e_1, \left|\mat x_\rho\right|^2\right) = 0,\qquad \forall \mat\eta\in \mat H^1(\bI). 
\end{align}

Let $\bI=\bigcup_{j=1}^J\bI_j$, $J\geq 3$, with $\bI_j=\left[q_{j-1}, q_j\right]$, $q_j=jh=j/J$ for $j=0, \ldots, J$, and we identity $0=q_0=q_J=1$. 
Then, we define the finite element spaces 
\begin{align*}
	V^h:=\left\{\chi\in C(\bar{\bI})\cap H^1(\bI): \chi\bigg|_{\mathbb I_j} \text{is affine,} j = 1, \ldots, J\right\},\qquad \mat V^h:=V^h\times V^h. 
\end{align*}
For temporal discretization, we define \( t_m = m \Delta t \) for \( m = 0, \ldots, M \), where \(\Delta t = T / M\) is the uniform time step size. 

To derive the following CN method, for a sequence of vector functions $\mat f^m:=\mat f(t_m)$, we denote
\begin{align*}
	D_t\mat f^{m+\frac12}:=\frac{\mat f^{m+1}-\mat f^m}{\Delta t},\qquad \widehat{\mat f}^{m+\frac12}:=\frac{\mat f^{m+1}+\mat f^m}{2},\qquad \widetilde{\mat f}^{m+\frac12}:=\frac{3\mat f^{m}-\mat f^{m-1}}{2}. 
\end{align*}
Then, we define the following CN method of the weak formulation \eqref{eqn:wf}. 
\begin{myDef}\label{def:cn}
	(\textbf{CN method}) For given $\mat X^0, \mat X^1\in \mat V^h$, find $\mat X^{m+1}$, $m=1, \ldots, M-1$, such that 
	\begin{align}\label{eqn:cn}
		\left(\widetilde{\mat X}^{m+\frac12} \cdot \mat e_1D_t\mat X^{m+\frac12}, \mat \eta^h\left|\widetilde{\mat X}^{m+\frac12}_\rho\right|^2\right)
		+\left(\widetilde{\mat X}^{m+\frac12} \cdot \mat e_1\widehat{\mat X}^{m+\frac12}_\rho, \mat \eta^h_\rho\right)
		+\left(\mat \eta^h\cdot\mat e_1, \left|\widetilde{\mat X}^{m+\frac12}_\rho\right|^2\right) = 0,\qquad \forall \mat\eta^h\in \mat V^h. 
	\end{align}
\end{myDef}

Next, to derive the following BDF2 method, we denote 
\begin{align*}
	\mathbb D_t\mat f^{m+1}:=\frac{3\mat f^{m+1}-4\mat f^m+\mat f^{m-1}}{2\Delta t},\qquad 
	\overline{\mat f}^{m+1} := 2\mat f^m-\mat f^{m-1}. 
\end{align*}
The following definition gives the BDF2 method of the weak formulation \eqref{eqn:wf}. 
\begin{myDef}\label{def:bdf}
	(\textbf{BDF2 method}) For given $\mat X^0, \mat X^1\in \mat V^h$, find $\mat X^{m+1}$, $m=1, \ldots, M-1$, such that 
	\begin{align}\label{eqn:bdf}
		\left(\overline{\mat X}^{m+1} \cdot \mat e_1\mathbb D_t\mat X^{m+1}, \mat \eta^h\left|\overline{\mat X}^{m+1}_\rho\right|^2\right)
		+\left(\overline{\mat X}^{m+1} \cdot \mat e_1\mat X^{m+1}_\rho, \mat \eta^h_\rho\right)
		+\left(\mat \eta^h\cdot\mat e_1, \left|\overline{\mat X}^{m+1}_\rho\right|^2\right) = 0,\qquad \forall \mat\eta^h\in \mat V^h. 
	\end{align}
\end{myDef}

We define the standard interpolation operator $\Pi^h: C(\bar{\bI})\rightarrow V^h$, such that 
\begin{align}\label{eqn:inter_error}
	\left\|f-\Pi^hf\right\|_{W^{k, p}}\leq Ch^{l-k}\left|f\right|_{W^{l, p}},\qquad \forall f\in W^{l, p}(\bI),
\end{align}
where $k\in\{0, 1\}$, $l\in \{1, 2\}$ and $p\in [2, \infty]$. 
For $f\in L^1(\bI)$, we further define 
\begin{align}\label{eqn:ph}
	\left(P^hf\right)\bigg|_{\bI_j}:=\frac{1}{h}\int_{\bI_j}fd\rho,\qquad j=1,\ldots, J, 
\end{align}
which satisfies that for $p\in [2, \infty]$,  
\begin{align}\label{eqn:ph_error}
	\left\|f-P^hf\right\|_{W^{0, p}}\leq Ch\left|f\right|_{W^{1, p}} ,\qquad \forall f\in W^{1, p}(\bI). 
\end{align}

\begin{rem}

We notice that in order to solve the CN method and the BDF2 method, the values of \(\mat{X}^0\) and \(\mat{X}^1\) must be known in advance. We can compute \(\mat{X}^0 = \Pi^h \mat{x}^0\), and \(\mat{X}^1\) can be obtained using the BDF1 method in \cite{barrett2021finite}. Although BDF1 is a first-order method, it is only used for the first step. Since theoretical analysis does not rely on Gronwall’s inequality at this stage, there is no error accumulation or order reduction, ensuring that \(\mat{X}^1\) still achieves second-order accuracy. Consequently, this does not affect the convergence results of subsequent time steps. Of course, other methods could also be used to ensure second-order accuracy at the first step.
For simplicity, we only consider the convergence order of the CN and BDF2 methods in our theoretical analysis.  
\end{rem}

The following theorem presents the main convergence results of this paper. 
\begin{thm}\label{thm:main}
	Suppose that \eqref{eqn:model} has a solution $\mat x(\rho, t): \bI\times [0, T]\rightarrow \mathbb R^2$, satisfying that 
	\begin{align}\label{eqn:assum2}
		\mat x\in C\left(\left[0, T\right]; \left[W^{2, \infty}(\bI)\right]^2\right),~ 
		\mat x_t\in C\left(\left[0, T\right]; \left[H^2(\bI)\right]^2\right),~
		\mat x_{tt}\in C\left(\left[0, T\right]; \left[H^2(\bI)\right]^2\right),~ \mat x_{ttt}\in C\left(\left[0, T\right]; \left[L^2(\bI)\right]^2\right), 
	\end{align}
	as well as 
	\begin{align}\label{eqn:assum1x}
		\left|\mat x_\rho\right|>0,\qquad \mat x\cdot\mat e_1>0\qquad \text{in}\quad \bar{\bI}\times [0, T]. 
	\end{align}
	Then there exist $\Delta t_0$, $h_0$, $\gamma_1$ and $\gamma_2$, such that when $0<h\leq h_0$, $0<\Delta t\leq \Delta t_0$, $\Delta t\leq \gamma_1\sqrt[4]{h}$ and $h\leq \gamma_2\sqrt{\Delta t}$, the CN method \eqref{eqn:cn} and the BDF2 method \eqref{eqn:bdf} have a unique solution $\mat X^{m+1}$, $m=1, \ldots, M-1$, such that 
	\begin{align}\label{eqn:err1}
		\max_{m=0,\ldots, M}\left\|\mat x^m-\mat X^m\right\|_{L^2}\leq C\left(\Delta t^2+h^2\right),\qquad 
		\max_{m=0,\ldots, M}\left|\mat x^m-\mat X^m\right|_{H^1}\leq C\left(\Delta t^2+h\right). 
	\end{align}
	In addition, we have the following superconvergence result:
	\begin{align}\label{eqn:err2}
		\max_{m=0,\ldots, M}\left\|\mat \Pi^hx^m-\mat X^m\right\|_{H^1}\leq C\left(\Delta t^2+h^2\right). 
	\end{align}
\end{thm}
\begin{proof}
	From \eqref{eqn:assum2} and \eqref{eqn:assum1x}, we observe that there exist positive constants $c_0$, $c_1$ and $C_0$, such that 
	\begin{align}\label{eqn:assum1}
		\left\|\mat x(\cdot, t)\right\|_{H^1}\leq C_0 
		\quad \text{in}\quad [0, T];\quad 
		c_0\leq \left|\mat x_\rho\right|\leq C_0,\quad 
		\mat x\cdot \mat e_1\geq c_1\quad \text{in} \quad \mathbb I\times [0, T]. 
	\end{align}
	We split
	\begin{align}\label{eqn:error_split}
		\mat x^m-\mat X^m = \bigg[\mat x^m-\Pi^h\mat x^m\bigg]+\bigg[\Pi^h\mat x^m-\mat X^m\bigg]=:\mat d^m+\mat E^m, \qquad m\geq 0. 
	\end{align}
	For convenience, we simply denote $\mat x_\Pi^m:=\Pi^h\mat x^m$. From \eqref{eqn:inter_error}, it suffices to prove \eqref{eqn:err2}.
	Following a similar approach as in \cite{barrett2021finite}, we can obtain the following result:
	\begin{align}\label{eqn:ini}
		\left\|\mat E^1\right\|_{H^1}\leq C\left(\Delta t^2+h^2\right). 
	\end{align}
	The subsequent proof will be given in the following two sections. 
\end{proof}

\begin{rem}
	The conditions $\Delta t\leq \gamma_1\sqrt[4]{h}$ and $h\leq \gamma_2\sqrt{\Delta t}$ are relatively stringent, yet it is indispensable for our proof. To remove the time-space ratio restriction, we can employ the time-space error splitting technique that is given in our recent work \cite{li2025error}. 
\end{rem}

For the theoretical analysis, we will frequently use the well-known Sobolev embedding inequality:
\begin{align}\label{eqn:inequ1}
	\left\| f \right\|_{L^\infty} \leq C\left\| f \right\|_{H^1}, \quad \forall f \in H^1(\bI).
\end{align}
In addition, for $k\in\{0, 1\}$, $l\in \{1, 2\}$ and $p\in [2, \infty]$, there holds 
\begin{align}\label{eqn:inequ2}
	h^{\frac{1}{p}-\frac1r}\left\|\omega\right\|_{W^{0, r}}+h\left|\omega\right|_{W^{1, p}}\leq C\left\|\omega\right\|_{W^{0, p}},\qquad \forall \omega\in V^h,\quad r\in [p, \infty].
\end{align}

\section{Error estimates for the CN method}\label{sec3}
\setcounter{equation}{0}
In this section, we aim to demonstrate the convergence results  \eqref{eqn:err1}-\eqref{eqn:err2} for the CN method. 
From \eqref{eqn:wf} and \eqref{eqn:cn}, we have the following error equation:
\begin{align}\label{eqn:cn_error}
&	\left(\widetilde{\mat X}^{m+\frac12} \cdot \mat e_1D_t\mat E^{m+\frac12}, \mat \eta^h\left|\widetilde{\mat X}^{m+\frac12}_\rho\right|^2\right)
	+\left(\widetilde{\mat X}^{m+\frac12} \cdot \mat e_1\widehat{\mat E}^{m+\frac12}_\rho, \mat \eta^h_\rho\right)\nn\\
& = \left[	\left(\widetilde{\mat X}^{m+\frac12} \cdot \mat e_1D_t\mat x_\Pi^{m+\frac12}, \mat \eta^h\left|\widetilde{\mat X}^{m+\frac12}_\rho\right|^2\right) -	\left(\mat x^{m+\frac12}\cdot \mat e_1\mat x_t^{m+\frac12}, \mat \eta^h\left|\mat x_\rho^{m+\frac12}\right|^2\right)\right]+\Bigg[\left(\widetilde{\mat X}^{m+\frac12} \cdot \mat e_1\widehat{\mat x}^{m+\frac12}_{\Pi, \rho}, \mat \eta^h_\rho\right)-\left(\mat x^{m+\frac12}\cdot\mat e_1\mat x^{m+\frac12}_\rho, \mat \eta^h_\rho\right)\Bigg]\nn\\
&~~~~+\left(\mat \eta^h\cdot\mat e_1, \left|\widetilde{\mat X}^{m+\frac12}_\rho\right|^2-\left|\mat x^{m+\frac12}_\rho\right|^2\right)=:\sum_{i=1}^3\mathbb T_i\left(\mat \eta^h\right).
\end{align}
Taking $\mat \eta^h =\Delta tD_t\mat E^{m+\frac12} $ in \eqref{eqn:cn_error}, and summing over $m=1, \ldots, n$, we have 
\begin{align}\label{eqn:cn_con_pf1}
	\Delta t\sum_{m=1}^n\left(\widetilde{\mat X}^{m+\frac12} \cdot \mat e_1\left|D_t\mat E^{m+\frac12}\right|^2, \left|\widetilde{\mat X}^{m+\frac12}_\rho\right|^2\right)
	+\sum_{m=1}^n\left(\widetilde{\mat X}^{m+\frac12} \cdot \mat e_1\widehat{\mat E}^{m+\frac12}_\rho, \mat E_\rho^{m+1}-\mat E_\rho^{m}\right)=\Delta t\sum_{m=1}^n\sum_{i=1}^3\mathbb T_i\left(D_t\mat E^{m+\frac12}\right).
\end{align}
We split 
\begin{align}\label{eqn:cn_con_pf2}
&	\sum_{m=1}^n\left(\widetilde{\mat X}^{m+\frac12} \cdot \mat e_1\widehat{\mat E}^{m+\frac12}_\rho, \mat E_\rho^{m+1}-\mat E_\rho^{m}\right)
	=\frac{1}{2}	\sum_{m=1}^n\left(\widetilde{\mat X}^{m+\frac12} \cdot \mat e_1, \left|\mat E_\rho^{m+1}\right|^2-\left|\mat E_\rho^{m}\right|^2\right)\nn\\
	&=\frac{3\Delta t}{4}	\sum_{m=1}^n\left( \mat X^m \cdot \mat e_1, D_t\left|\mat E_\rho^{m+\frac12}\right|^2\right)-\frac{\Delta t}{4}	\sum_{m=1}^n\left( \mat X^{m-1} \cdot \mat e_1, D_t\left|\mat E_\rho^{m+\frac12}\right|^2\right).
\end{align}
For the two terms on the right-hand side of \eqref{eqn:cn_con_pf2}, there hold 
\begin{align}
&\frac{3\Delta t}{4}	\sum_{m=1}^n\left( \mat X^m \cdot \mat e_1, D_t\left|\mat E_\rho^{m+\frac12}\right|^2\right)
=\frac{3}{4}\left( \mat X^n \cdot \mat e_1, \left|\mat E_\rho^{n+1}\right|^2\right)
-\frac{3}{4}\sum_{m=2}^n\left(\left[\mat X^{m}-\mat X^{m-1}\right]\cdot \mat e_1, \left|\mat E_\rho^{m}\right|^2\right)
-\frac34\left( \mat X^1 \cdot \mat e_1, \left|\mat E_\rho^{1}\right|^2\right),\label{eqn:cn_con_pf3}\\
&\frac{\Delta t}{4}	\sum_{m=1}^n\left( \mat X^{m-1} \cdot \mat e_1, D_t\left|\mat E_\rho^{m+\frac12}\right|^2\right)
=\frac{1}{4}\left( \mat X^{n-1} \cdot \mat e_1, \left|\mat E_\rho^{n+1}\right|^2\right)
-\frac{1}{4}\sum_{m=2}^n\left( \left[\mat X^{m-1}-\mat X^{m-2}\right] \cdot \mat e_1, \left|\mat E_\rho^{m}\right|^2\right)
-\frac14\left( \mat X^0 \cdot \mat e_1, \left|\mat E_\rho^{1}\right|^2\right).\label{eqn:cn_con_pf4}
\end{align}
Substituting \eqref{eqn:cn_con_pf3} and \eqref{eqn:cn_con_pf4} into \eqref{eqn:cn_con_pf2} gives that 
\begin{align}\label{eqn:cn_con_pf5}
	\sum_{m=1}^n\left(\widetilde{\mat X}^{m+\frac12} \cdot \mat e_1\widehat{\mat E}^{m+\frac12}_\rho, \mat E_\rho^{m+1}-\mat E_\rho^{m}\right)
	=\frac{1}{2}\left( \widetilde{\mat X}^{n+\frac12}  \cdot \mat e_1, \left|\mat E_\rho^{n+1}\right|^2\right)
	-\frac{\Delta t}{2}\sum_{m=2}^n\left( D_t\widetilde{\mat X}^{m+\frac12} \cdot \mat e_1, \left|\mat E_\rho^{m}\right|^2\right)
	-\frac12\left( 	\widetilde{\mat X}^{\frac32} \cdot \mat e_1, \left|\mat E_\rho^{1}\right|^2\right).
\end{align}
Substituting \eqref{eqn:cn_con_pf5} into \eqref{eqn:cn_con_pf1}, we obtain 
\begin{align}\label{eqn:cn_con_pf6}
&	\Delta t\sum_{m=1}^n\left(\widetilde{\mat X}^{m+\frac12} \cdot \mat e_1\left|D_t\mat E^{m+\frac12}\right|^2, \left|\widetilde{\mat X}^{m+\frac12}_\rho\right|^2\right)
	+\frac{1}{2}\left( \widetilde{\mat X}^{n+\frac12}  \cdot \mat e_1, \left|\mat E_\rho^{n+1}\right|^2\right)\nn\\
	&=
	\frac{\Delta t}{2}\sum_{m=2}^n\left( D_t\widetilde{\mat X}^{m+\frac12} \cdot \mat e_1, \left|\mat E_\rho^{m}\right|^2\right)+
	\Delta t\sum_{m=1}^n\sum_{i=1}^3\mathbb T_i\left(D_t\mat E^{m+\frac12}\right)+\frac12\left( 	\widetilde{\mat X}^{\frac32} \cdot \mat e_1, \left|\mat E_\rho^{1}\right|^2\right).
\end{align}

By the mathematical induction, using Taylor's formula, \eqref{eqn:assum2} and \eqref{eqn:inequ1}, we have 
\begin{align}\label{eqn:cn_con_pf7}
&	\left\|\mat x^{m+\frac12}\cdot \mat e_1-\widetilde{\mat X}^{m+\frac12}\cdot \mat e_1\right\|_{L^\infty}\leq \left\|\mat x^{m+\frac12}\cdot \mat e_1-\widetilde{\mat x}^{m+\frac12}\cdot \mat e_1\right\|_{L^\infty}+	\left\|\widetilde{\mat x}^{m+\frac12}\cdot \mat e_1-\widetilde{\mat X}^{m+\frac12}\cdot \mat e_1\right\|_{L^\infty}\nn\\
&
	\leq C\Delta t^2+ C\left\|\widetilde{\mat x}^{m+\frac12}-\widetilde{\mat X}^{m+\frac12}\right\|_{H^1}
	\leq C\left(\Delta t^2+h\right).
\end{align}
From \eqref{eqn:cn_con_pf7} and the assumption \eqref{eqn:assum1}, we have 
\begin{align}\label{eqn:cn_con_pf8}
	\widetilde{\mat X}^{m+\frac12}\cdot \mat e_1\geq \mat x^{m+\frac12}\cdot \mat e_1-C\left(\Delta t^2+h\right)
	\geq C_1-C\left(\Delta t^2+h\right)\geq \frac{C_1}{2},
\end{align}
provided that $\Delta t>0$ and $h>0$ are selected sufficiently small. 
In addition, by using the inverse inequality, mathematical induction, \eqref{eqn:inter_error} and Taylor's formula, we have
\begin{align}\label{eqn:cn_con_pf9}
&	\left|\widetilde{\mat X}^{m+\frac12}_\rho-{\mat x}_\rho^{m+\frac12}\right|
	\leq 	\left\|\widetilde{\mat X}_\rho^{m+\frac12}-\widetilde{\mat x}_\rho^{m+\frac12}\right\|_{L^\infty}
	+\left\|{\mat x}_\rho^{m+\frac12}-\widetilde{\mat x}_\rho^{m+\frac12}\right\|_{L^\infty}\nn\\
&\leq 
	\left\|\widetilde{\mat E}_\rho^{m+\frac12}\right\|_{L^\infty}
	+\left\|\widetilde{\mat d}_\rho^{m+\frac12}\right\|_{L^\infty}
	+C\Delta t^2
	\leq Ch^{-\frac12}\left\|\widetilde{\mat E}_\rho^{m+\frac12}\right\|_{L^2}+C\left(\Delta t^2+h\right)\nn\\
	&\leq Ch^{-\frac12}\left(\Delta t^2+h^2\right)+C\left(\Delta t^2+h\right)\leq 
	C\left(h^{-\frac12}\Delta t^2+h+\Delta t^2\right).
\end{align}
Under the conditions of 
$\Delta t\leq \gamma_1\sqrt[4]{h}$ with suitably selected positive constants $\gamma_1$, and thanks to \eqref{eqn:assum1} and \eqref{eqn:cn_con_pf9}, we derive 
\begin{align}\label{eqn:cn_con_pf10}
\frac{c_0}{2}\leq \left|{\mat x}^{m+\frac12}_\rho\right|-C\left(h^{-\frac12}\Delta t^2+h+\Delta t^2\right)	\leq \left|\widetilde{\mat X}^{m+\frac12}_\rho\right|\leq \left|{\mat x}^{m+\frac12}_\rho\right|+C\left(h^{-\frac12}\Delta t^2+h+\Delta t^2\right)\leq 2C_0.
\end{align}
In addition, we also have 
\begin{align}\label{eqn:cn_con_pf10x}
\left\|\widetilde{\mat X}^{m+\frac12}\right\|_{L^\infty}\leq \left\|\mat x^{m+\frac12}\right\|_{L^\infty}+ \left\|\mat x^{m+\frac12}-\widetilde{\mat x}^{m+\frac12}\right\|_{L^\infty}+\left\|\widetilde{\mat d}^{m+\frac12}\right\|_{L^\infty}+\left\|\widetilde{\mat E}^{m+\frac12}\right\|_{L^\infty}\leq \left\|\mat x^{m+\frac12}\right\|_{L^\infty}+ C(\Delta t^2+h^2)\leq 2C_0.
\end{align}
By applying \eqref{eqn:cn_con_pf8} and \eqref{eqn:cn_con_pf10} to the left-hand side of \eqref{eqn:cn_con_pf6}, we obtain
\begin{align}\label{eqn:cn_con_pf11}
		\Delta t\sum_{m=1}^n\left(\widetilde{\mat X}^{m+\frac12} \cdot \mat e_1\left|D_t\mat E^{m+\frac12}\right|^2, \left|\widetilde{\mat X}^{m+\frac12}_\rho\right|^2\right)
	+\frac{1}{2}\left( \widetilde{\mat X}^{n+\frac12}  \cdot \mat e_1, \left|\mat E_\rho^{n+1}\right|^2\right)
	\geq \frac{C_1c_0^2}{8}	\Delta t\sum_{m=1}^n\left\|D_t\mat E^{m+\frac12}\right\|_{L^2}^2
	+\frac{C_1}{4}\left\|\mat E_\rho^{n+1}\right\|_{L^2}^2.
\end{align}
For the first term on the right-hand side of \eqref{eqn:cn_con_pf6}, using \eqref{eqn:inequ1} and the mathematical induction, we get  
\begin{align}\label{eqn:cn_con_pf12}
	& \frac{\Delta t}{2}\sum_{m=2}^n\left( D_t\widetilde{\mat X}^{m+\frac12} \cdot \mat e_1, \left|\mat E_\rho^{m}\right|^2\right)
	=\frac{\Delta t}{2}\sum_{m=2}^n\left( D_t\widetilde{\mat x}_\Pi^{m+\frac12} \cdot \mat e_1, \left|\mat E_\rho^{m}\right|^2\right)
	-\frac{\Delta t}{2}\sum_{m=2}^n\left( D_t\widetilde{\mat E}^{m+\frac12} \cdot \mat e_1, \left|\mat E_\rho^{m}\right|^2\right)\nn\\
	&\leq \frac{\Delta t}{2}\sum_{m=2}^n
	\left(\left\|D_t\widetilde{\mat x}_\Pi^{m+\frac12} \right\|_{L^\infty}+\left\|D_t\widetilde{\mat E}^{m+\frac12}\right\|_{L^\infty}\right)\left\|\mat E_\rho^{m}\right\|_{L^2}^2\leq C\Delta t\sum_{m=2}^n
	\left(\left\|D_t\widetilde{\mat x}_\Pi^{m+\frac12} \right\|_{L^\infty}+\left\|D_t\widetilde{\mat E}^{m+\frac12}\right\|_{H^1}\right)\left\|\mat E_\rho^{m}\right\|_{L^2}^2\nn\\
	&\leq C\Delta t\sum_{m=2}^n
	\left[1+\Delta t^{-1}\left(\Delta t^2+h^2\right)\right]\left\|\mat E_\rho^{m}\right\|_{L^2}^2\leq C\Delta t\sum_{m=2}^n\left\|\mat E_\rho^{m}\right\|_{L^2}^2,
	\end{align}
	where we assume that $h\leq \gamma_2\sqrt{\Delta t}$ with suitably selected positive constant $\gamma_2$. 
Next, we estimate the second term on the right-hand side of \eqref{eqn:cn_con_pf6}. We split 
\begin{align}\label{eqn:cn_con_pf13}
\mathbb T_1\left(D_t\mat E^{m+\frac12}\right)
&=\left(\widetilde{\mat X}^{m+\frac12} \cdot \mat e_1D_t\mat x_\Pi^{m+\frac12}, D_t\mat E^{m+\frac12}\left|\widetilde{\mat X}^{m+\frac12}_\rho\right|^2\right) -	\left(\mat x^{m+\frac12}\cdot \mat e_1\mat x_t^{m+\frac12}, D_t\mat E^{m+\frac12}\left|\mat x_\rho^{m+\frac12}\right|^2\right)\nn\\
&=\left(\left[\widetilde{\mat X}^{m+\frac12}-\mat x^{m+\frac12}\right] \cdot \mat e_1D_t\mat x_\Pi^{m+\frac12}, D_t\mat E^{m+\frac12}\left|\widetilde{\mat X}^{m+\frac12}_\rho\right|^2\right)
+\left(\mat x^{m+\frac12}\cdot \mat e_1
\left[D_t\mat x_\Pi^{m+\frac12}-\mat x_t^{m+\frac12}\right], D_t\mat E^{m+\frac12}\left|\widetilde{\mat X}^{m+\frac12}_\rho\right|^2\right)\nn\\
&~~~~+\left(\mat x^{m+\frac12}\cdot \mat e_1
\mat x_t^{m+\frac12}, D_t\mat E^{m+\frac12}
\left[\left|\widetilde{\mat X}^{m+\frac12}_\rho\right|^2-\left|\mat x_\rho^{m+\frac12}\right|^2\right]\right)\nn\\
&=:\sum_{i=1}^3\mathbb T_{1, i}\left(D_t\mat E^{m+\frac12}\right). 
\end{align}
Using \eqref{eqn:cn_con_pf10}, \eqref{eqn:inequ1} and \eqref{eqn:assum2}, we have 
\begin{align}\label{eqn:cn_con_pf14}
&\mathbb T_{1, 1}\left(D_t\mat E^{m+\frac12}\right)
=\left(\left[\widetilde{\mat X}^{m+\frac12}-\mat x^{m+\frac12}\right] \cdot \mat e_1D_t\mat x_\Pi^{m+\frac12}, D_t\mat E^{m+\frac12}\left|\widetilde{\mat X}^{m+\frac12}_\rho\right|^2\right)\nn\\
&\leq \left\|\widetilde{\mat X}^{m+\frac12}-\mat x^{m+\frac12}\right\|_{L^2}\left\|D_t\mat x_\Pi^{m+\frac12}\right\|_{L^\infty}
\left\|D_t\mat E^{m+\frac12}\right\|_{L^2}
\left\|\widetilde{\mat X}^{m+\frac12}_\rho\right\|_{L^\infty}^2\nn\\
&\leq \epsilon\left\|D_t\mat E^{m+\frac12}\right\|_{L^2}^2+C_\epsilon\left\|\widetilde{\mat X}^{m+\frac12}-\mat x^{m+\frac12}\right\|_{L^2}^2.
\end{align}
We observe that 
\begin{align}\label{eqn:cn_con_pf15}
\widetilde{\mat X}^{m+\frac12}-\mat x^{m+\frac12}
=\bigg[\widetilde{\mat X}^{m+\frac12}-\widetilde{\mat x}^{m+\frac12}\bigg]+\bigg[\widetilde{\mat x}^{m+\frac12}-\mat x^{m+\frac12}\bigg]
=-\frac32\left(\mat E^m+\mat d^m\right)+\frac12\left(\mat E^{m-1}+\mat d^{m-1}\right)+\bigg[\widetilde{\mat x}^{m+\frac12}-\mat x^{m+\frac12}\bigg]. 
\end{align}
Substituting \eqref{eqn:cn_con_pf15} into \eqref{eqn:cn_con_pf14}, from \eqref{eqn:inter_error} and by using Taylor's formula, we can obtain 
\begin{align}\label{eqn:cn_con_pf16}
	\Delta t\sum_{m=1}^n \mathbb T_{1, 1}\left(D_t\mat E^{m+\frac12}\right)\leq 
	\epsilon\,\Delta t\sum_{m=1}^n\left\|D_t\mat E^{m+\frac12}\right\|_{L^2}^2
	+C_\epsilon\Delta t\sum_{m=0}^n\left\|\mat E^m\right\|_{L^2}^2
	+C_\epsilon\left(\Delta t^4+h^4\right). 
\end{align}
In addition, from \eqref{eqn:assum2} and \eqref{eqn:cn_con_pf10}, we have 
\begin{align}\label{eqn:cn_con_pf17}
	&	\Delta t\sum_{m=1}^n \mathbb T_{1, 2}\left(D_t\mat E^{m+\frac12}\right)\leq 
		\Delta t\sum_{m=1}^n \left\|\mat x^{m+\frac12}\right\|_{L^\infty}
		\left\|D_t\mat x_\Pi^{m+\frac12}-\mat x_t^{m+\frac12}\right\|_{L^2}\left\|D_t\mat E^{m+\frac12}\right\|_{L^2}\left\|\widetilde{\mat X}^{m+\frac12}_\rho\right\|_{L^\infty}^2\nn\\
		&\leq \epsilon\,\Delta t\sum_{m=1}^n\left\|D_t\mat E^{m+\frac12}\right\|_{L^2}^2
		+C_\epsilon\Delta t\sum_{m=1}^n\left\|D_t\mat x_\Pi^{m+\frac12}-\mat x_t^{m+\frac12}\right\|_{L^2}^2.
\end{align}
Thanks to 
\begin{align*}
	&D_t\mat x_\Pi^{m+\frac12}-\mat x_t^{m+\frac12} = \frac{1}{\Delta t}\int_{t_m}^{t_{m+1}}\bigg[\mat x_{\Pi, t}-\mat x_t^{m+\frac12}\bigg]dt
	 = \frac{1}{\Delta t}\int_{t_m}^{t_{m+1}}\bigg[\mat x_{\Pi, t}-\mat x_t\bigg]dt
	 + \frac{1}{\Delta t}\int_{t_m}^{t_{m+1}}\bigg[\mat x_{t}-\mat x_t^{m+\frac12}\bigg]dt\nn\\
	 &=\frac{1}{\Delta t}\int_{t_m}^{t_{m+1}}\bigg[\mat x_{\Pi, t}-\mat x_t\bigg]dt
	 + \frac{1}{\Delta t}\int_{t_m}^{t_{m+1}}\left[\left(t-t_{m+\frac12}\right)\mat x_{tt}^{m+\frac12}+\frac{\left(t-t_{m+\frac12}\right)^2}{2}\mat x_{ttt}
	 \left(\cdot, \xi\right)\right]dt\nn\\
	 &=\frac{1}{\Delta t}\int_{t_m}^{t_{m+1}}\bigg[\mat x_{\Pi, t}-\mat x_t\bigg]dt
	 + \frac{1}{2\Delta t}\int_{t_m}^{t_{m+1}}\left(t-t_{m+\frac12}\right)^2\mat x_{ttt}
	 \left(\cdot, \xi\right)dt,
\end{align*}
and using \eqref{eqn:inter_error} and \eqref{eqn:assum2}, we have 
\begin{align}\label{eqn:cn_con_pf18}
&	\left\|D_t\mat x_\Pi^{m+\frac12}-\mat x_t^{m+\frac12}\right\|_{L^2}
	\leq \frac{1}{\Delta t}\int_{t_m}^{t_{m+1}}\left\|\mat x_{\Pi, t}-\mat x_t\right\|_{L^2}dt
	+ \frac{1}{2\Delta t}\int_{t_m}^{t_{m+1}}\left(t-t_{m+\frac12}\right)^2dt\max_{t\in [t_m, t_{m+1}]}\left\|\mat x_{ttt}\right\|_{L^2}\nn\\
	&\leq Ch^2\max_{t\in [t_m, t_{m+1}]}\left|\mat x_{t}\right|_{H^2}
	+\frac{\Delta t^2}{24}\max_{t\in [t_m, t_{m+1}]}\left\|\mat x_{ttt}\right\|_{L^2}
	\leq C\left(\Delta t^2+h^2\right). 
\end{align}
Hence, from \eqref{eqn:cn_con_pf17} and \eqref{eqn:cn_con_pf18}, there holds 
\begin{align}\label{eqn:cn_con_pf18add}
	\Delta t\sum_{m=1}^n \mathbb T_{1, 2}\left(D_t\mat E^{m+\frac12}\right)\leq \epsilon\,\Delta t\sum_{m=1}^n\left\|D_t\mat E^{m+\frac12}\right\|_{L^2}^2
	+ C_\epsilon\left(\Delta t^4+h^4\right). 
\end{align}
We next estimate the term $\mathbb T_{1, 3}\left(D_t\mat E^{m+\frac12}\right)$. To this end, we first split 
\begin{align}\label{eqn:cn_con_pf19}
&	\left|\widetilde{\mat X}^{m+\frac12}_\rho\right|^2-\left|\mat x_\rho^{m+\frac12}\right|^2
	=\left(\widetilde{\mat X}^{m+\frac12}_\rho-\mat x_\rho^{m+\frac12}\right)^2
	+2\left(\widetilde{\mat X}^{m+\frac12}_\rho-\widetilde{\mat x}^{m+\frac12}_\rho\right)\cdot \mat x_\rho^{m+\frac12}+2\left(\widetilde{\mat x}^{m+\frac12}_\rho-\mat x_\rho^{m+\frac12}\right)\cdot \mat x_\rho^{m+\frac12}\nn\\
	&=
	\left(\mat x_\rho^{m+\frac12}-\widetilde{\mat x}^{m+\frac12}_\rho
	+\widetilde{\mat E}^{m+\frac12}_\rho+\widetilde{\mat d}^{m+\frac12}_\rho
	\right)^2
	-2\left(\widetilde{\mat E}^{m+\frac12}_\rho+\widetilde{\mat d}^{m+\frac12}_\rho\right)\cdot \mat x_\rho^{m+\frac12}
	+2\left(\widetilde{\mat x}^{m+\frac12}_\rho-\mat x_\rho^{m+\frac12}\right)\cdot \mat x_\rho^{m+\frac12}\nn\\
	&=:-2\widetilde{\mat d}^{m+\frac12}_\rho\cdot \mat x_\rho^{m+\frac12}
	+u^{m+\frac12}.
\end{align}
Using \eqref{eqn:assum1}, \eqref{eqn:cn_con_pf10}, \eqref{eqn:inequ1}, \eqref{eqn:assum2}, \eqref{eqn:inter_error} and Taylor's formula, we have 
\begin{align}\label{eqn:cn_con_pf20}
	\left\|u^{m+\frac12}\right\|_{L^2}\leq &
	\left(\left\|\mat x_\rho^{m+\frac12}-\widetilde{\mat x}^{m+\frac12}_\rho\right\|_{L^\infty}\left\|\mat x_\rho^{m+\frac12}-\widetilde{\mat x}^{m+\frac12}_\rho\right\|_{L^2}
	+\left\|\widetilde{\mat E}^{m+\frac12}_\rho\right\|_{L^\infty}\left\|\widetilde{\mat E}^{m+\frac12}_\rho\right\|_{L^2}+\left\|\widetilde{\mat d}^{m+\frac12}_\rho\right\|_{L^\infty}\left\|\widetilde{\mat d}^{m+\frac12}_\rho\right\|_{L^2}\right)\nn\\
	&+2\left(\left\|\mat x_\rho^{m+\frac12}-\widetilde{\mat x}^{m+\frac12}_\rho\right\|_{L^\infty}\left\|\widetilde{\mat E}^{m+\frac12}_\rho\right\|_{L^2} 
	+\left\|\mat x_\rho^{m+\frac12}-\widetilde{\mat x}^{m+\frac12}_\rho\right\|_{L^\infty}\left\|\widetilde{\mat d}^{m+\frac12}_\rho\right\|_{L^2} 
	+\left\|\widetilde{\mat d}^{m+\frac12}_\rho\right\|_{L^\infty} \left\|\widetilde{\mat E}^{m+\frac12}_\rho\right\|_{L^2} 
	\right)\nn\\
	&+2\left\|\widetilde{\mat E}^{m+\frac12}_\rho\right\|_{L^2}\left\|\mat x_\rho^{m+\frac12}\right\|_{L^\infty}+2\left\|\widetilde{\mat x}^{m+\frac12}_\rho-\mat x_\rho^{m+\frac12}\right\|_{L^2}\left\|\mat x_\rho^{m+\frac12}\right\|_{L^\infty}\nn\\
	\leq & C\left(\left\|\widetilde{\mat E}^{m+\frac12}_\rho\right\|_{L^2}+\Delta t^2+h^2\right). 
\end{align}
From \eqref{eqn:cn_con_pf19} and \eqref{eqn:cn_con_pf20}, using integration by parts, there holds 
\begin{align}\label{eqn:cn_con_pf21}
	&\mathbb T_{1, 3}\left(D_t\mat E^{m+\frac12}\right) = -2\left(\mat x^{m+\frac12}\cdot \mat e_1
	\mat x_t^{m+\frac12}, D_t\mat E^{m+\frac12}
	\widetilde{\mat d}^{m+\frac12}_\rho\cdot \mat x_\rho^{m+\frac12}
	\right)+\left(\mat x^{m+\frac12}\cdot \mat e_1
	\mat x_t^{m+\frac12}, D_t\mat E^{m+\frac12}
	u^{m+\frac12}\right)\nn\\
	& = 2\left(\mat x_\rho^{m+\frac12}\cdot \mat e_1
	\mat x_t^{m+\frac12}, D_t\mat E^{m+\frac12}
	\widetilde{\mat d}^{m+\frac12}\cdot \mat x_\rho^{m+\frac12}
	\right)+
	2\left(\mat x^{m+\frac12}\cdot \mat e_1
	\mat x_{t\rho}^{m+\frac12}, D_t\mat E^{m+\frac12}
	\widetilde{\mat d}^{m+\frac12}\cdot \mat x_\rho^{m+\frac12}
	\right)\nn\\
& ~~~~+	2\left(\mat x^{m+\frac12}\cdot \mat e_1
\mat x_t^{m+\frac12}, D_t\mat E_\rho^{m+\frac12}
\widetilde{\mat d}^{m+\frac12}\cdot \mat x_\rho^{m+\frac12}
\right)
	+2\left(\mat x^{m+\frac12}\cdot \mat e_1
	\mat x_t^{m+\frac12}, D_t\mat E^{m+\frac12}
	\widetilde{\mat d}^{m+\frac12}\cdot \mat x_{\rho\rho}^{m+\frac12}
	\right)\nn\\
&~~~~	+
	\left(\mat x^{m+\frac12}\cdot \mat e_1
	\mat x_t^{m+\frac12}, D_t\mat E^{m+\frac12}
	u^{m+\frac12}\right)\nn\\
	&=: \left(D_t\mat E^{m+\frac12}, \mat v^{m+\frac12}\right)+	\left(D_t\mat E_\rho^{m+\frac12}, \mat w^{m+\frac12}
	\right),
\end{align}
where 
\begin{align*}
&\mat v^{m+\frac12}:=2\mat x_\rho^{m+\frac12}\cdot \mat e_1
\widetilde{\mat d}^{m+\frac12}\cdot \mat x_\rho^{m+\frac12}\mat x_t^{m+\frac12}
+2\mat x^{m+\frac12}\cdot \mat e_1	\widetilde{\mat d}^{m+\frac12}\cdot \mat x_\rho^{m+\frac12}
\mat x_{t\rho}^{m+\frac12}
+2\mat x^{m+\frac12}\cdot \mat e_1	\widetilde{\mat d}^{m+\frac12}\cdot \mat x_{\rho\rho}^{m+\frac12}
\mat x_t^{m+\frac12}+\mat x^{m+\frac12}\cdot \mat e_1u^{m+\frac12}
\mat x_t^{m+\frac12},\\
&\mat w^{m+\frac12}:=2\mat x^{m+\frac12}\cdot \mat e_1	\widetilde{\mat d}^{m+\frac12}\cdot \mat x_\rho^{m+\frac12}
\mat x_t^{m+\frac12}. 
\end{align*}
By using \eqref{eqn:inter_error}, \eqref{eqn:assum2} and \eqref{eqn:cn_con_pf20}, we easily obtain 
\begin{align}\label{eqn:cn_con_pf22}
	\left\|\mat v^{m+\frac12}\right\|_{L^2}\leq C\left(\left\|\widetilde{\mat E}^{m+\frac12}_\rho\right\|_{L^2}+\Delta t^2+h^2\right), 
\end{align}
which further implies that 
\begin{align}\label{eqn:cn_con_pf22x}
	\Delta t\sum_{m=1}^n\left(D_t\mat E^{m+\frac12}, \mat v^{m+\frac12}
	\right)\leq \epsilon \Delta t\sum_{m=1}^n\left\|D_t\mat E^{m+\frac12}\right\|_{L^2}^2+C_\epsilon\Delta t\sum_{m=1}^n\left\|\widetilde{\mat E}^{m+\frac12}_\rho\right\|_{L^2}^2+C_\epsilon\left(\Delta t^4+h^4\right).
\end{align}
Additionally, it is obvious that 
\begin{align}\label{eqn:cn_con_pf23}
\Delta t\sum_{m=1}^n\left(D_t\mat E_\rho^{m+\frac12}, \mat w^{m+\frac12}
\right)=\left(\mat E_\rho^{n+1}, \mat w^{n+\frac12}\right)
-\Delta t\sum_{m=2}^n\left(\mat E_\rho^{m}, \frac{\mat w^{m+\frac12}-\mat w^{m-\frac12}}{\Delta t}\right)
-\left(\mat E_\rho^1, \mat w^{\frac32}\right).
\end{align}
Obviously, from \eqref{eqn:inter_error} and because of \eqref{eqn:assum2}, we have 
\begin{align}\label{eqn:cn_con_pf24}
	\left\|\mat w^{n+\frac12}\right\|_{L^2}\leq Ch^2,\qquad 	\left\|\mat w^{\frac32}\right\|_{L^2}\leq Ch^2. 
\end{align}
Thanks to \eqref{eqn:inter_error}, Taylor's formula and \eqref{eqn:assum2}, we have 
\begin{align}\label{eqn:cn_con_pf25}
	\left\|\frac{\mat w^{m+\frac12}-\mat w^{m-\frac12}}{\Delta t}\right\|_{L^2}
&	\leq 2\left\|\frac{\mat x^{m+\frac12}\cdot\mat e_1-\mat x^{m-\frac12}\cdot\mat e_1}{\Delta t}\right\|_{L^\infty}\left\|\mat x_\rho^{m+\frac12}\right\|_{L^\infty}\left\|\mat x_t^{m+\frac12}\right\|_{L^\infty}\left\|\widetilde{\mat d}^{m+\frac12}\right\|_{L^2}\nn\\
&~~~~+2\left\|\mat x^{m-\frac12}\cdot \mat e_1\right\|_{L^\infty}\left\|\mat x_\rho^{m+\frac12}\right\|_{L^\infty}\left\|\mat x_t^{m+\frac12}\right\|_{L^\infty}
\left\|\frac{\widetilde{\mat d}^{m+\frac12}-\widetilde{\mat d}^{m-\frac12}}{\Delta t}\right\|_{L^2}\nn\\
&~~~~+2\left\|\mat x^{m-\frac12}\cdot \mat e_1\right\|_{L^\infty}\left\|\frac{\mat x_\rho^{m+\frac12}-\mat x_\rho^{m-\frac12}}{\Delta t}\right\|_{L^\infty}\left\|\mat x_t^{m+\frac12}\right\|_{L^\infty}\left\|\widetilde{\mat d}^{m-\frac12}\right\|_{L^2}\nn\\
&~~~~+2\left\|\mat x^{m-\frac12}\cdot \mat e_1\right\|_{L^\infty}\left\|\mat x_\rho^{m-\frac12}\right\|_{L^\infty}\left\|\frac{\mat x_t^{m+\frac12}-\mat x_t^{m-\frac12}}{\Delta t}\right\|_{L^\infty}\left\|\widetilde{\mat d}^{m-\frac12}\right\|_{L^2}\nn\\
&\leq C\left(\left\|\widetilde{\mat d}^{m+\frac12}\right\|_{L^2}+\left\|\widetilde{\mat d}^{m-\frac12}\right\|_{L^2}+\left\|\frac{\widetilde{\mat d}^{m+\frac12}-\widetilde{\mat d}^{m-\frac12}}{\Delta t}\right\|_{L^2}\right)\leq Ch^2. 
\end{align}
Using \eqref{eqn:cn_con_pf24} and \eqref{eqn:cn_con_pf25} in \eqref{eqn:cn_con_pf23} gives that 
\begin{align}\label{eqn:cn_con_pf26}
	\Delta t\sum_{m=1}^n\left(D_t\mat E_\rho^{m+\frac12}, \mat w^{m+\frac12}
	\right)\leq \epsilon\left\|\mat E_\rho^{n+1}\right\|_{L^2}^2
	+C\Delta t\sum_{m=2}^n\left\|\mat E_\rho^{m}\right\|_{L^2}^2
	+C\left\|\mat E_\rho^{1}\right\|_{L^2}^2 +C_\epsilon h^4. 
\end{align}
Then, combining \eqref{eqn:cn_con_pf22x} and \eqref{eqn:cn_con_pf26}, we derive that 
\begin{align}\label{eqn:cn_con_pf27}
\Delta t\sum_{m=1}^n \mathbb T_{1, 3}\left(D_t\mat E^{m+\frac12}\right)\leq 
\epsilon\left\|\mat E_\rho^{n+1}\right\|_{L^2}^2+
\epsilon \Delta t\sum_{m=1}^n\left\|D_t\mat E^{m+\frac12}\right\|_{L^2}^2
+C_\epsilon\Delta t\sum_{m=1}^n\left\|\mat E_\rho^{m}\right\|_{L^2}^2
+C\left\|\mat E_\rho^{1}\right\|_{L^2}^2 +C_\epsilon \left(\Delta t^4+h^4\right). 
\end{align}
Combining \eqref{eqn:cn_con_pf16}, \eqref{eqn:cn_con_pf18add} and \eqref{eqn:cn_con_pf27} together, we arrive at 
\begin{align}\label{eqn:cn_con_pf28}
	\Delta t\sum_{m=1}^n \mathbb T_{1}\left(D_t\mat E^{m+\frac12}\right)\leq \epsilon\left\|\mat E_\rho^{n+1}\right\|_{L^2}^2+
	\epsilon \Delta t\sum_{m=1}^n\left\|D_t\mat E^{m+\frac12}\right\|_{L^2}^2
	+C_\epsilon\Delta t\sum_{m=1}^n\left\|\mat E^{m}\right\|_{H^1}^2
	+C\left\|\mat E_\rho^{1}\right\|_{L^2}^2 +C_\epsilon \left(\Delta t^4+h^4\right). 
\end{align}
Let us next investigate the terms involving $\mathbb T_2(D_t\mat E^{m+\frac12})$ in \eqref{eqn:cn_error}. Since 
\begin{align}\label{eqn:cn_con_pf29}
	\int_{\mathbb I_j}\left[f-\Pi^hf\right]_\rho\eta_\rho d\rho = 0,\qquad \eta\in V^h,\quad j=1,\ldots, J, 
\end{align}
we can write 
\begin{align}\label{eqn:cn_con_pf30}
\mathbb T_2\left(D_t\mat E^{m+\frac12}\right)&=	\left(\widetilde{\mat X}^{m+\frac12} \cdot \mat e_1\widehat{\mat x}^{m+\frac12}_{\Pi, \rho}, D_t\mat E^{m+\frac12}_\rho\right)-\left(\mat x^{m+\frac12}\cdot\mat e_1\mat x^{m+\frac12}_\rho, D_t\mat E^{m+\frac12}_\rho\right)\nn\\
&=-\left(\widetilde{\mat X}^{m+\frac12} \cdot \mat e_1\widehat{\mat d}^{m+\frac12}_{ \rho}, D_t\mat E^{m+\frac12}_\rho\right)
+\left(\widetilde{\mat X}^{m+\frac12} \cdot \mat e_1\left[\widehat{\mat x}^{m+\frac12}_{\rho}-\mat x^{m+\frac12}_\rho\right], D_t\mat E^{m+\frac12}_\rho\right)\nn\\
&~~~~
+\left(\left[\widetilde{\mat X}^{m+\frac12} \cdot \mat e_1-\mat x^{m+\frac12}\cdot\mat e_1\right]\mat x^{m+\frac12}_\rho, D_t\mat E^{m+\frac12}_\rho\right)\nn\\
&=-\left(\left[\widetilde{\mat X}^{m+\frac12} \cdot \mat e_1-P^h\left(\widetilde{\mat X}^{m+\frac12} \cdot \mat e_1\right)\right]\widehat{\mat d}^{m+\frac12}_{\rho}, D_t\mat E^{m+\frac12}_\rho\right)
+\left(\widetilde{\mat X}^{m+\frac12} \cdot \mat e_1\left[\widehat{\mat x}^{m+\frac12}_{\rho}-\mat x^{m+\frac12}_\rho\right], D_t\mat E^{m+\frac12}_\rho\right)\nn\\
&~~~~
+\left(\left[\widetilde{\mat X}^{m+\frac12} \cdot \mat e_1-\mat x^{m+\frac12}\cdot\mat e_1\right]\mat x^{m+\frac12}_\rho, D_t\mat E^{m+\frac12}_\rho\right)\nn\\
&=: \sum_{i=1}^3\mathbb T_{2, i}\left(D_t\mat E^{m+\frac12}\right). 
\end{align}
For the term involving $\mathbb T_{2, 1}\left(D_t\mat E^{m+\frac12}\right)$, we denote  
\begin{align}\label{eqn:cn_con_pf31}
	\Delta t\sum_{m=1}^n\mathbb T_{2, 1}\left(D_t\mat E^{m+\frac12}\right)
	=
	-\Delta t\sum_{m=1}^n\left(\left[\widetilde{\mat X}^{m+\frac12} \cdot \mat e_1-P^h\left(\widetilde{\mat X}^{m+\frac12} \cdot \mat e_1\right)\right]\widehat{\mat d}^{m+\frac12}_{\rho}, D_t\mat E^{m+\frac12}_\rho\right)
	=:-\Delta t\sum_{m=1}^n\left(\mat g^{m+\frac12}, D_t\mat E^{m+\frac12}_\rho\right). 
\end{align}
In addition, we have 
\begin{align}\label{eqn:cn_con_pf32}
	-\Delta t\sum_{m=1}^n\left(\mat g^{m+\frac12}, D_t\mat E^{m+\frac12}_\rho\right)
	=-\left(\mat g^{n+\frac12}, \mat E^{n+1}_\rho\right)+\Delta t \sum_{m=2}^n\left(\frac{\mat g^{m+\frac12}-\mat g^{m-\frac12}}{\Delta t}, \mat E^{m}_\rho\right)
	+\left(\mat g^{\frac32}, \mat E^{1}_\rho\right).
\end{align}
Obviously, from \eqref{eqn:ph_error}, \eqref{eqn:inter_error}, \eqref{eqn:cn_con_pf10} and \eqref{eqn:assum2}, there holds 
\begin{align}\label{eqn:cn_con_pf33}
	\left\|\mat g^{n+\frac12}\right\|_{L^2}\leq 
	\left\|\widetilde{\mat X}^{m+\frac12} \cdot \mat e_1-P^h\left(\widetilde{\mat X}^{m+\frac12} \cdot \mat e_1\right)\right\|_{L^\infty}
	\left\|\widehat{{\mat d}}^{m+\frac12}_{\rho}\right\|_{L^2}
	\leq 
	Ch^2\left\|\widetilde{\mat X}_\rho^{n+\frac12} \cdot \mat e_1\right\|_{L^\infty}\left|\widehat{\mat x}^{n+\frac12}\right|_{H^2}\leq Ch^2.
\end{align}
Similarly, we have 
\begin{align}\label{eqn:cn_con_pf34}
	\left\|\mat g^{\frac32}\right\|_{L^2}\leq Ch^2. 
\end{align}
Moreover, we write 
\begin{align}\label{eqn:cn_con_pf35}
	&\frac{\mat g^{m+\frac12}-\mat g^{m-\frac12}}{\Delta t}
	=\frac{\left[\widetilde{\mat X}^{m+\frac12} \cdot \mat e_1-P^h\left(\widetilde{\mat X}^{m+\frac12} \cdot \mat e_1\right)\right]\widehat{\mat d}^{m+\frac12}_{ \rho}-\left[\widetilde{\mat X}^{m-\frac12} \cdot \mat e_1-P^h\left(\widetilde{\mat X}^{m-\frac12} \cdot \mat e_1\right)\right]\widehat{\mat d}^{m-\frac12}_{ \rho}}{\Delta t}\nn\\
	& =\frac{\left[\widetilde{\mat X}^{m+\frac12} \cdot \mat e_1-P^h\left(\widetilde{\mat X}^{m+\frac12} \cdot \mat e_1\right)\right]-\left[\widetilde{\mat X}^{m-\frac12} \cdot \mat e_1-P^h\left(\widetilde{\mat X}^{m-\frac12} \cdot \mat e_1\right)\right]}{\Delta t}\widehat{\mat d}^{m+\frac12}_{ \rho}+\left[\widetilde{\mat X}^{m-\frac12} \cdot \mat e_1-P^h\left(\widetilde{\mat X}^{m-\frac12} \cdot \mat e_1\right)\right]\frac{\widehat{\mat d}^{m+\frac12}_{ \rho}-\widehat{\mat d}^{m-\frac12}_{\rho}}{\Delta t}\nn\\
	&=\left[\frac{\widetilde{\mat X}^{m+\frac12} \cdot \mat e_1-\widetilde{\mat X}^{m-\frac12} \cdot \mat e_1}{\Delta t}-P^h\frac{\widetilde{\mat X}^{m+\frac12} \cdot \mat e_1-\widetilde{\mat X}^{m-\frac12} \cdot \mat e_1}{\Delta t}\right]\widehat{\mat d}^{m+\frac12}_{\rho}+\left[\widetilde{\mat X}^{m-\frac12} \cdot \mat e_1-P^h\left(\widetilde{\mat X}^{m-\frac12} \cdot \mat e_1\right)\right]\frac{\widehat{\mat d}^{m+\frac12}_{\rho}-\widehat{\mat d}^{m-\frac12}_{\rho}}{\Delta t}. 
\end{align}
Then, using \eqref{eqn:inter_error} and \eqref{eqn:ph_error} in \eqref{eqn:cn_con_pf35},  thanks to \eqref{eqn:inter_error}, \eqref{eqn:assum2}, \eqref{eqn:cn_con_pf10} and Taylor's formula, we have 
\begin{align}\label{eqn:cn_con_pf36}
\left\|\frac{\mat g^{m+\frac12}-\mat g^{m-\frac12}}{\Delta t}\right\|_{L^2}
\leq 
&\left\|\frac{\widetilde{\mat X}^{m+\frac12} \cdot \mat e_1-\widetilde{\mat X}^{m-\frac12} \cdot \mat e_1}{\Delta t}-P^h\frac{\widetilde{\mat X}^{m+\frac12} \cdot \mat e_1-\widetilde{\mat X}^{m-\frac12} \cdot \mat e_1}{\Delta t}\right\|_{L^2}
\left\|\widehat{\mat d}^{m+\frac12}_{\rho}\right\|_{L^\infty}\nn\\
&+\left\|\widetilde{\mat X}^{m-\frac12} \cdot \mat e_1-P^h\left(\widetilde{\mat X}^{m-\frac12} \cdot \mat e_1\right)\right\|_{L^\infty}\left\|\frac{\widehat{\mat d}^{m+\frac12}_{\rho}-\widehat{\mat d}^{m-\frac12}_{\rho}}{\Delta t}\right\|_{L^2}\nn\\
\leq & Ch^2\left|\frac{\widetilde{\mat X}^{m+\frac12}-\widetilde{\mat X}^{m-\frac12} }{\Delta t}\right|_{H^1}\left|\widehat{\mat x}^{m+\frac12}\right|_{W^{2,\infty}}
+Ch^2\left|\widetilde{\mat X}^{m-\frac12}\right|_{W^{1, \infty}}\left|\frac{\widehat{\mat x}^{m+\frac12}-\widehat{\mat x}^{m-\frac12}}{\Delta t}\right|_{H^2}\nn\\
\leq & Ch^2\left|\frac{\widetilde{\mat X}^{m+\frac12}-\widetilde{\mat X}^{m-\frac12} }{\Delta t}\right|_{H^1}
+Ch^2
\leq Ch^2\left[\left|\frac{\widetilde{\mat x}_\Pi^{m+\frac12}-\widetilde{\mat x}_\Pi^{m-\frac12} }{\Delta t}\right|_{H^1}+\left|\frac{\widetilde{\mat E}^{m+\frac12}-\widetilde{\mat E}^{m-\frac12} }{\Delta t}\right|_{H^1}\right]+Ch^2\nn\\
\leq &  Ch^2\bigg[\Delta t^{-1}\left(\Delta t^2+h^2\right)+1\bigg]+Ch^2\leq Ch^2,
\end{align}
provided that $h\leq \gamma_3\sqrt{\Delta t}$ with suitably selected positive constant $\gamma_3$. 
Using \eqref{eqn:cn_con_pf33}, \eqref{eqn:cn_con_pf34} and \eqref{eqn:cn_con_pf36} in 
\eqref{eqn:cn_con_pf31}, we get 
\begin{align}\label{eqn:cn_con_pf37}
	\Delta t\sum_{m=1}^n\mathbb T_{2, 1}\left(D_t\mat E^{m+\frac12}\right)
	\leq \epsilon\left\|\mat E_\rho^{n+1}\right\|_{L^2}^2
	+C_\epsilon\Delta t\sum_{m=2}^n\left\|\mat E_\rho^{m}\right\|_{L^2}^2
	+C\left\|\mat E_\rho^{1}\right\|_{L^2}^2 +C_\epsilon h^4. 
\end{align}
For the term involving $\mathbb T_{2, 2}\left(D_t\mat E^{m+\frac12}\right)$, using integration by parts, Taylor's formula, \eqref{eqn:cn_con_pf10}, \eqref{eqn:cn_con_pf10x} and \eqref{eqn:assum1x}, we obtain 
\begin{align}\label{eqn:cn_con_pf38}
	&\Delta t\sum_{m=1}^n\mathbb T_{2, 2}\left(D_t\mat E^{m+\frac12}\right)=		\Delta t\sum_{m=1}^n\left(\widetilde{\mat X}^{m+\frac12} \cdot \mat e_1\left[\widehat{\mat x}^{m+\frac12}_{\rho}-\mat x^{m+\frac12}_\rho\right], D_t\mat E^{m+\frac12}_\rho\right)\nn\\
	&=-	\Delta t\sum_{m=1}^n\left(\widetilde{\mat X}_\rho^{m+\frac12} \cdot \mat e_1\left[\widehat{\mat x}^{m+\frac12}_{\rho}-\mat x^{m+\frac12}_\rho\right], D_t\mat E^{m+\frac12}\right)-	\Delta t\sum_{m=1}^n\left(\widetilde{\mat X}^{m+\frac12} \cdot \mat e_1\left[\widehat{\mat x}^{m+\frac12}_{\rho\rho}-\mat x^{m+\frac12}_{\rho\rho}\right], D_t\mat E^{m+\frac12}\right)\nn\\
	&\leq \Delta t\sum_{m=1}^n\left\|\widetilde{\mat X}_\rho^{m+\frac12} \cdot \mat e_1\right\|_{L^\infty}\left\|\widehat{\mat x}^{m+\frac12}_{\rho}-\mat x^{m+\frac12}_\rho\right\|_{L^2} \left\|D_t\mat E^{m+\frac12}\right\|_{L^2}
	+\Delta t\sum_{m=1}^n\left\|\widetilde{\mat X}^{m+\frac12} \cdot \mat e_1\right\|_{L^\infty}\left\|\widehat{\mat x}^{m+\frac12}_{\rho\rho}-\mat x^{m+\frac12}_{\rho\rho}\right\|_{L^2} \left\|D_t\mat E^{m+\frac12}\right\|_{L^2}\nn\\
	&\leq C\Delta t^3\sum_{m=1}^n\left\|D_t\mat E^{m+\frac12}\right\|_{L^2}
	\leq \epsilon\Delta t\sum_{m=1}^n\left\|D_t\mat E^{m+\frac12}\right\|_{L^2}^2 +C_\epsilon\Delta t^4. 
\end{align}
Furthermore, for the term involving $\mathbb T_{2, 3}\left(D_t\mat E^{m+\frac12}\right)$, by virtue of integration by parts, we have 
\begin{align}\label{eqn:cn_con_pf39}
	&\Delta t\sum_{m=1}^n\mathbb T_{2, 3}\left(D_t\mat E^{m+\frac12}\right)=		\Delta t\sum_{m=1}^n\left(\left[\widetilde{\mat X}^{m+\frac12} \cdot \mat e_1-\mat x^{m+\frac12}\cdot\mat e_1\right]\mat x^{m+\frac12}_\rho, D_t\mat E^{m+\frac12}_\rho\right)\nn\\
	& =	-\Delta t\sum_{m=1}^n\left(\left[\widetilde{\mat E}^{m+\frac12}+\widetilde{\mat d}^{m+\frac12}+\left(\mat x^{m+\frac12}-\widetilde{\mat x}^{m+\frac12}\right)\right] \cdot \mat e_1\mat x^{m+\frac12}_\rho, D_t\mat E^{m+\frac12}_\rho\right)\nn\\
	& =	\Delta t\sum_{m=1}^n\left(\mat\psi^{m+\frac12}, D_t\mat E^{m+\frac12}\right)
	-\Delta t\sum_{m=1}^n\left(\widetilde{\mat d}^{m+\frac12}\cdot \mat e_1\mat x^{m+\frac12}_\rho, D_t\mat E^{m+\frac12}_\rho\right),
\end{align}
where 
\begin{align*}
\mat \psi^{m+\frac12}:=	\left[\widetilde{\mat E}_\rho^{m+\frac12}+\left(\mat x_\rho^{m+\frac12}-\widetilde{\mat x}_\rho^{m+\frac12}\right)\right] \cdot \mat e_1\mat x^{m+\frac12}_\rho+	\left[\widetilde{\mat E}^{m+\frac12}+\left(\mat x^{m+\frac12}-\widetilde{\mat x}^{m+\frac12}\right)\right] \cdot \mat e_1\mat x^{m+\frac12}_{\rho\rho}.
\end{align*}
Obviously, using \eqref{eqn:assum2} and Taylor's formula, we have 
\begin{align}\label{eqn:cn_con_pf40}
&	\Delta t\sum_{m=1}^n\left(\mat\psi^{m+\frac12}, D_t\mat E^{m+\frac12}\right)\leq 
	\Delta t\sum_{m=1}^n\left\|\mat\psi^{m+\frac12}\right\|_{L^2} \left\|D_t\mat E^{m+\frac12}\right\|_{L^2}\nn\\
&	\leq C
	\Delta t\sum_{m=1}^n\left(\left\|\widetilde{\mat E}_\rho^{m+\frac12}\right\|_{L^2}+\left\|\widetilde{\mat E}^{m+\frac12}\right\|_{L^2}+\Delta t^2\right) \left\|D_t\mat E^{m+\frac12}\right\|_{L^2}\nn\\
	&\leq \epsilon	\Delta t\sum_{m=1}^n \left\|D_t\mat E^{m+\frac12}\right\|_{L^2}^2+C_\epsilon\Delta t\sum_{m=1}^n\left\|\mat E_\rho^{m}\right\|_{L^2}^2+C_\epsilon\Delta t\sum_{m=1}^n\left\|\mat E^{m}\right\|_{L^2}^2+C_\epsilon\Delta t^4. 
\end{align}
For the second term on the right-hand side of \eqref{eqn:cn_con_pf39}, we have 
\begin{align}\label{eqn:cn_con_pf41}
	-\Delta t\sum_{m=1}^n\left(\widetilde{\mat d}^{m+\frac12}\cdot \mat e_1\mat x^{m+\frac12}_\rho, D_t\mat E^{m+\frac12}_\rho\right)
	=&-\left(\widetilde{\mat d}^{n+\frac12}\cdot \mat e_1\mat x^{n+\frac12}_\rho, \mat E^{n+1}_\rho\right)
	+\Delta t\sum_{m=2}^n\left(\frac{\widetilde{\mat d}^{m+\frac12}\cdot \mat e_1\mat x^{m+\frac12}_\rho-\widetilde{\mat d}^{m-\frac12}\cdot \mat e_1\mat x^{m-\frac12}_\rho}{\Delta t}, \mat E_\rho^m\right)\nn\\
	&+\left(\widetilde{\mat d}^{\frac32}\cdot \mat e_1\mat x^{\frac32}_\rho, \mat E^{1}_\rho\right). 
\end{align}
Thanks to \eqref{eqn:inter_error}, \eqref{eqn:assum2} and Taylor's formula, we obtain 
\begin{align}\label{eqn:cn_con_pf42}
&	\left\|\frac{\widetilde{\mat d}^{m+\frac12}\cdot \mat e_1\mat x^{m+\frac12}_\rho-\widetilde{\mat d}^{m-\frac12}\cdot \mat e_1\mat x^{m-\frac12}_\rho}{\Delta t}\right\|_{L^2}
	=	\left\|\frac{\widetilde{\mat d}^{m+\frac12}\cdot \mat e_1-\widetilde{\mat d}^{m-\frac12}\cdot \mat e_1}{\Delta t}\mat x^{m+\frac12}_\rho
	+\widetilde{\mat d}^{m-\frac12}\cdot \mat e_1\frac{\mat x^{m+\frac12}_\rho-\mat x^{m-\frac12}_\rho}{\Delta t}\right\|_{L^2}\nn\\
	&\leq \left\|\frac{\widetilde{\mat d}^{m+\frac12}-\widetilde{\mat d}^{m-\frac12}}{\Delta t}\right\|_{L^2}\left\|\mat x^{m+\frac12}_\rho\right\|_{L^\infty}+\left\|\widetilde{\mat d}^{m-\frac12}\right\|_{L^2}\left\|\frac{\mat x^{m+\frac12}_\rho-\mat x^{m-\frac12}_\rho}{\Delta t}\right\|_{L^\infty}\nn\\
&	\leq Ch^2\left|\frac{\widetilde{\mat x}^{m+\frac12}-\widetilde{\mat x}^{m-\frac12}}{\Delta t}\right|_{H^2}\left|\mat x^{m+\frac12}\right|_{H^2}
+Ch^2\left|\widetilde{\mat x}^{m-\frac12}\right|_{H^2}\left|\frac{{\mat x}^{m+\frac12}-{\mat x}^{m-\frac12}}{\Delta t}\right|_{H^2}\leq Ch^2. 
\end{align}
Using \eqref{eqn:cn_con_pf42} in \eqref{eqn:cn_con_pf41}, there obviously holds that 
\begin{align}\label{eqn:cn_con_pf43}
	-\Delta t\sum_{m=1}^n\left(\widetilde{\mat d}^{m+\frac12}\cdot \mat e_1\mat x^{m+\frac12}_\rho, D_t\mat E^{m+\frac12}_\rho\right)\leq \epsilon\left\|\mat E_\rho^{n+1}\right\|_{L^2}^2
	+C_\epsilon\Delta t\sum_{m=2}^n\left\|\mat E_\rho^{m}\right\|_{L^2}^2
	+C\left\|\mat E_\rho^{1}\right\|_{L^2}^2 +C_\epsilon h^4. 
\end{align}
Taking \eqref{eqn:cn_con_pf40} and \eqref{eqn:cn_con_pf43} in \eqref{eqn:cn_con_pf39} gives that 
\begin{align}\label{eqn:cn_con_pf44}
	&\Delta t\sum_{m=1}^n\mathbb T_{2, 3}\left(D_t\mat E^{m+\frac12}\right)\leq 
	\epsilon\left\|\mat E_\rho^{n+1}\right\|_{L^2}^2 +\epsilon	\Delta t\sum_{m=1}^n \left\|D_t\mat E^{m+\frac12}\right\|_{L^2}^2+C_\epsilon\Delta t\sum_{m=1}^n\left\|\mat E^{m}\right\|_{H^1}^2
	+C\left\|\mat E_\rho^{1}\right\|_{L^2}^2 +C_\epsilon h^4. 
\end{align}
Combining \eqref{eqn:cn_con_pf37}, \eqref{eqn:cn_con_pf38} and \eqref{eqn:cn_con_pf44}, we derive 
\begin{align}\label{eqn:cn_con_pf45}
	\Delta t\sum_{m=1}^n \mathbb T_{2}\left(D_t\mat E^{m+\frac12}\right)\leq \epsilon\left\|\mat E_\rho^{n+1}\right\|_{L^2}^2+
	\epsilon \Delta t\sum_{m=1}^n\left\|D_t\mat E^{m+\frac12}\right\|_{L^2}^2
	+C_\epsilon\Delta t\sum_{m=1}^n\left\|\mat E^{m}\right\|_{H^1}^2
	+C\left\|\mat E_\rho^{1}\right\|_{L^2}^2 +C_\epsilon \left(\Delta t^4+h^4\right). 
\end{align}
In what follows, we estimate the term involving $\mathbb T_{3}\left(D_t\mat E^{m+\frac12}\right)$ in \eqref{eqn:cn_con_pf1}. From \eqref{eqn:cn_con_pf19}, we have 
\begin{align}\label{eqn:cn_con_pf46}
&\Delta t\sum_{m=1}^n \mathbb T_{3}\left(D_t\mat E^{m+\frac12}\right)=	\Delta t\sum_{m=1}^n\left(D_t\mat E^{m+\frac12}\cdot\mat e_1, -2\widetilde{\mat d}^{m+\frac12}_\rho\cdot \mat x_\rho^{m+\frac12}
+u^{m+\frac12}\right).
\end{align}
Following a similar derivation process as that for $\mathbb{T}_{1, 3}$, we can obtain 
\begin{align}\label{eqn:cn_con_pf47}
	\Delta t\sum_{m=1}^n \mathbb T_{3}\left(D_t\mat E^{m+\frac12}\right)\leq 
	\epsilon\left\|\mat E_\rho^{n+1}\right\|_{L^2}^2+
	\epsilon \Delta t\sum_{m=1}^n\left\|D_t\mat E^{m+\frac12}\right\|_{L^2}^2
	+C_\epsilon\Delta t\sum_{m=1}^n\left\|\mat E_\rho^{m}\right\|_{L^2}^2
	+C\left\|\mat E_\rho^{1}\right\|_{L^2}^2 +C_\epsilon \left(\Delta t^4+h^4\right). 
\end{align}
For the last term on the right-hand side of \eqref{eqn:cn_con_pf6}, from \eqref{eqn:cn_con_pf10x}, we can easily derive 
\begin{align}\label{eqn:cn_con_pf48}
	\frac12\left( 	\widetilde{\mat X}^{\frac32} \cdot \mat e_1, \left|\mat E_\rho^{1}\right|^2\right)\leq \frac12\left\|\widetilde{\mat X}^{\frac32}\right\|_{L^\infty}\left\|\mat E_\rho^{1}\right\|^2_{L^2}\leq C_0\left\|\mat E_\rho^{1}\right\|^2_{L^2}. 
\end{align}

Using \eqref{eqn:cn_con_pf11}, \eqref{eqn:cn_con_pf12}, \eqref{eqn:cn_con_pf28}, \eqref{eqn:cn_con_pf45}, \eqref{eqn:cn_con_pf47} and \eqref{eqn:cn_con_pf48} in  \eqref{eqn:cn_con_pf6}, by virtue of \eqref{eqn:ini}, and selecting sufficiently small $\epsilon$, we can conclude that 
\begin{align}\label{eqn:cn_con_pf49}
	\Delta t\sum_{m=1}^n\left\|D_t\mat E^{m+\frac12}\right\|_{L^2}^2
	+\left\|\mat E_\rho^{n+1}\right\|_{L^2}^2\leq C\Delta t\sum_{m=1}^n\left\|\mat E^{m}\right\|_{H^1}^2+
	C \left(\Delta t^4+h^4\right). 
\end{align}
Noting that the second term on the left-hand side of \eqref{eqn:cn_con_pf49} represents an \( H^1 \)-seminorm while the first term on the right-hand side represents an \( H^1 \)-norm, the Gronwall inequality cannot be directly applied.  
Since $\mat E^0=0$, we write 
\begin{align}\label{eqn:cn_con_pf50}
	\left\|\mat E^{n+1}\right\|_{L^2}^2
	=\sum_{m=0}^n\left(\left\|\mat E^{m+1}\right\|_{L^2}^2-\left\|\mat E^{m}\right\|_{L^2}^2\right)
	=2\Delta t\sum_{m=0}^n\left(D_t\mat E^{m+\frac12}, \widehat{\mat E}^{m+\frac12}\right)
	\leq \epsilon \Delta t\sum_{m=1}^n\left\|D_t\mat E^{m+\frac12}\right\|_{L^2}^2
	+C\Delta t \sum_{m=1}^{n+1}\left\|\mat E^{m}\right\|_{L^2}^2. 
\end{align}
Adding \eqref{eqn:cn_con_pf49} and \eqref{eqn:cn_con_pf50}, choosing a small $\epsilon$, we have 
\begin{align}\label{eqn:cn_con_pf51}
	\left\|\mat E^{n+1}\right\|_{H^1}^2\leq C\Delta t\sum_{m=1}^{n+1}\left\|\mat E^{m}\right\|_{H^1}^2 +C \left(\Delta t^4+h^4\right). 
\end{align}
Using the discrete Gronwall inequality in \eqref{eqn:cn_con_pf51}, we finally conclude that  
\begin{align}\label{eqn:cn_con_pf52}
		\left\|\mat E^{n+1}\right\|_{H^1}\leq C\left(\Delta t^2+h^2\right), 
\end{align}
provided that $\Delta t$ is sufficiently small. Therefore, we have completed the proof. 

\section{Error estimate for the BDF2 method}\label{sec4}
\setcounter{equation}{0}
In this section, we aim to demonstrate the convergence results \eqref{eqn:err1}-\eqref{eqn:err2} for the BDF2 method.
From \eqref{eqn:wf} and \eqref{eqn:cn}, we have the following error equation:
\begin{align}\label{eq:bdf_error}
	&	\left(\overline{\mat X}^{m+1} \cdot \mat e_1\mathbb D_t\mat E^{m+1}, \mat \eta^h\left|\overline{\mat X}^{m+1}_\rho\right|^2\right)
	+\left(\overline{\mat X}^{m+1} \cdot \mat e_1\mat E^{m+1}_\rho, \mat \eta^h_\rho\right)\nn\\
	& = \left[	\left(\overline{\mat X}^{m+1} \cdot \mat e_1\mathbb D_t\mat x_\Pi^{m+1}, \mat \eta^h\left|\overline{\mat X}^{m+1}_\rho\right|^2\right) -	\left(\mat x^{m+1}\cdot \mat e_1\mat x_t^{m+1}, \mat \eta^h\left|\mat x_\rho^{m+1}\right|^2\right)\right]
	+\left[\left(\overline{\mat X}^{m+1} \cdot \mat e_1\mat x^{m+1}_{\Pi, \rho}, \mat \eta^h_\rho\right)-\left(\mat x^{m+1}\cdot\mat e_1\mat x^{m+1}_\rho, \mat \eta^h_\rho\right)\right]\nn\\
	&~~~~+\left(\mat \eta^h\cdot\mat e_1, \left|\overline{\mat X}^{m+1}_\rho\right|^2-\left|\mat x^{m+1}_\rho\right|^2\right)=:\sum_{i=1}^3\mathscr T_i\left(\mat \eta^h\right).
\end{align}
Taking $\mat \eta^h =\Delta t\mathbb D_t\mat E^{m+1} $ in \eqref{eq:bdf_error}, and summing over $m=1, \ldots, n$, we have 
\begin{align}\label{eq:bdf_con_pf1}
	\Delta t\sum_{m=1}^n\left(\overline{\mat X}^{m+1} \cdot \mat e_1\left|\mathbb D_t\mat E^{m+1}\right|^2, \left|\overline{\mat X}^{m+1}_\rho\right|^2\right)
+\Delta t\sum_{m=1}^n\left(\overline{\mat X}^{m+1} \cdot \mat e_1\mat E^{m+1}_\rho, \mathbb D_t\mat E^{m+1}_\rho\right)
=\Delta t\sum_{m=1}^n\sum_{i=1}^3\mathscr T_i\left(\mathbb D_t\mat E^{m+1}\right).
\end{align}
By simple calculation, we can obtain 
\begin{align}\label{eq:bdf_con_pf2}
	\mat E_\rho^{m+1}\cdot \mathbb D_t\mat E_\rho^{m+1}
	&=\frac{1}{4\Delta t}\left[\left(\left|\mat E_\rho^{m+1}\right|^2-\left|\mat E_\rho^{m}\right|^2\right)+\left(\left|2\mat E_\rho^{m+1}-\mat E_\rho^m\right|^2-\left|2\mat E_\rho^{m}-\mat E_\rho^{m-1}\right|^2\right)+\left|\mat E_\rho^{m+1}-2\mat E_\rho^m+\mat E_\rho^{m-1}\right|^2\right]\nn\\
	&=:\frac14D_tF^{m+\frac12}+\frac{\left|\mat E_\rho^{m+1}-2\mat E_\rho^m+\mat E_\rho^{m-1}\right|^2}{4\Delta t},
\end{align}
where $F^m=\left|\mat E_\rho^{m}\right|^2+\left|2\mat E_\rho^{m}-\mat E_\rho^{m-1}\right|^2$. 
Similar as \eqref{eqn:cn_con_pf8}, we have 
\begin{align}\label{eq:bdf_con_pf3}
	\overline{\mat X}^{m+1} \cdot \mat e_1\geq \frac{C_1}{2}. 
\end{align}
From \eqref{eq:bdf_con_pf2} and \eqref{eq:bdf_con_pf3}, we obtain 
\begin{align}\label{eq:bdf_con_pf4}
	\Delta t\sum_{m=1}^n\left(\overline{\mat X}^{m+1} \cdot \mat e_1\mat E^{m+1}_\rho, \mathbb D_t\mat E^{m+1}_\rho\right)
&	=	\frac{\Delta t}{4}\sum_{m=1}^n\left(\overline{\mat X}^{m+1} \cdot \mat e_1, D_tF^{m+\frac12}\right)
	+\frac14\sum_{m=1}^n\left(\overline{\mat X}^{m+1} \cdot \mat e_1, \left|\mat E_\rho^{m+1}-2\mat E_\rho^m+\mat E_\rho^{m-1}\right|^2\right). 
\end{align}
For the first term on the right-hand side of \eqref{eq:bdf_con_pf4}, there holds 
\begin{align}\label{eq:bdf_con_pf5}
	\frac{\Delta t}{4}\sum_{m=1}^n\left(\overline{\mat X}^{m+1} \cdot \mat e_1, D_tF^{m+\frac12}\right)
	= \frac14\left(\overline{\mat X}^{n+1}\cdot\mat e_1, F^{n+1}\right)
	-\frac{\Delta t}{4}\sum_{m=2}^n\left(D_t\overline{\mat X}^{m+\frac12}\cdot\mat e_1, F^m\right)
	-\frac{1}{4}\left(\overline{\mat X}^2\cdot\mat e_1, F^1\right).
\end{align} 
Substituting \eqref{eq:bdf_con_pf4} and \eqref{eq:bdf_con_pf5} into \eqref{eq:bdf_con_pf1} gives that 
\begin{align}\label{eq:bdf_con_pf6}
	&	\Delta t\sum_{m=1}^n\left(\overline{\mat X}^{m+1} \cdot \mat e_1\left|\mathbb D_t\mat E^{m+1}\right|^2, \left|\overline{\mat X}^{m+1}_\rho\right|^2\right)
	+ \frac14\left(\overline{\mat X}^{n+1}\cdot\mat e_1, F^{n+1}\right)+\frac14\sum_{m=1}^n\left(\overline{\mat X}^{m+1} \cdot \mat e_1, \left|\mat E_\rho^{m+1}-2\mat E_\rho^m+\mat E_\rho^{m-1}\right|^2\right)\nn\\
	&=\frac{\Delta t}{4}\sum_{m=2}^n\left(D_t\overline{\mat X}^{m+\frac12}\cdot\mat e_1, F^m\right)+\Delta t\sum_{m=1}^n\sum_{i=1}^3\mathscr T_i\left(\mathbb D_t\mat E^{m+1}\right)+\frac{1}{4}\left(\overline{\mat X}^2\cdot\mat e_1, F^1\right).
\end{align}
Similar as \eqref{eqn:cn_con_pf10} and \eqref{eqn:cn_con_pf10x}, we have 
\begin{align}\label{eq:bdf_con_pf7}
	\frac{c_0}{2}\leq \left|\overline{\mat X}^{m+1}_\rho\right|\leq 2C_0,\qquad 
	\left\|\overline{\mat X}^{m+1}\right\|_{L^\infty}\leq 2C_0,\qquad m\geq 1. 
\end{align}
Using \eqref{eq:bdf_con_pf3} and \eqref{eq:bdf_con_pf7} on the left-hand side of \eqref{eq:bdf_con_pf6}, we have 
\begin{align}\label{eq:bdf_con_pf8}
&	\Delta t\sum_{m=1}^n\left(\overline{\mat X}^{m+1} \cdot \mat e_1\left|\mathbb D_t\mat E^{m+1}\right|^2, \left|\overline{\mat X}^{m+1}_\rho\right|^2\right)
	+ \frac14\left(\overline{\mat X}^{n+1}\cdot\mat e_1, F^{n+1}\right)+\frac14\sum_{m=1}^n\left(\overline{\mat X}^{m+1} \cdot \mat e_1, \left|\mat E_\rho^{m+1}-2\mat E_\rho^m+\mat E_\rho^{m-1}\right|^2\right)\nn\\
	& \geq \frac{C_1c_0^2}{8}\Delta t\sum_{m=1}^n\left\|\mathbb D_t\mat E^{m+1}\right\|_{L^2}^2+\frac{C_1}{8}\left(1, F^{n+1}\right)
	+\frac{C_1}{8}\sum_{m=1}^n\left\|\mat E_\rho^{m+1}-2\mat E_\rho^m+\mat E_\rho^{m-1}\right\|_{L^2}^2. 
\end{align}
For the first term on the right-hand side of \eqref{eq:bdf_con_pf6}, by using \eqref{eqn:inequ1} and the mathematical induction, we have 
\begin{align}\label{eq:bdf_con_pf9}
	&\frac{\Delta t}{4}\sum_{m=2}^n\left(D_t\overline{\mat X}^{m+\frac12}\cdot\mat e_1, F^m\right)=
		\frac{\Delta t}{4}\sum_{m=2}^n\left(D_t\overline{\mat x}_\Pi^{m+\frac12}\cdot\mat e_1, F^m\right)
		-
		\frac{\Delta t}{4}\sum_{m=2}^n\left(D_t\overline{\mat E}^{m+\frac12}\cdot\mat e_1, F^m\right)\nn\\
		&\leq \frac{\Delta t}{4}\sum_{m=2}^n\left(\left\|D_t\overline{\mat x}_\Pi^{m+\frac12}\right\|_{L^\infty}+\left\|D_t\overline{\mat E}^{m+\frac12}\right\|_{L^\infty}\right)\left(1, F^m\right)\leq C\Delta t\sum_{m=2}^n\left(\left\|D_t\overline{\mat x}_\Pi^{m+\frac12}\right\|_{L^\infty}+\left\|D_t\overline{\mat E}^{m+\frac12}\right\|_{H^1}\right)\left(1, F^m\right)\nn\\
		&\leq C\Delta t\sum_{m=2}^n
		\left[1+\Delta t^{-1}\left(\Delta t^2+h^2\right)\right]\left(1, F^m\right)\leq C\Delta t\sum_{m=2}^n\left(1, F^m\right),
\end{align}
where we have assumed that $h\leq \gamma_2\sqrt{\Delta t}$ with suitably selected positive constant $\gamma_2$. We then estimate the second term on the right-hand side of \eqref{eq:bdf_con_pf6}. We split 
\begin{align}\label{eq:bdf_con_pf10}
\mathscr T_1\left(\mathbb D_t\mat E^{m+1}\right)& = \left(\overline{\mat X}^{m+1} \cdot \mat e_1\mathbb D_t\mat x_\Pi^{m+1}, \mathbb D_t\mat E^{m+1} \left|\overline{\mat X}^{m+1}_\rho\right|^2\right) -	\left(\mat x^{m+1}\cdot \mat e_1\mat x_t^{m+1},  \mathbb D_t\mat E^{m+1} \left|\mat x_\rho^{m+1}\right|^2\right)\nn\\
& = \left(\left[\overline{\mat X}^{m+1}-\mat x^{m+1}\right] \cdot \mat e_1\mathbb D_t\mat x_\Pi^{m+1}, \mathbb D_t\mat E^{m+1} \left|\overline{\mat X}^{m+1}_\rho\right|^2\right) 
+\left(\mat x^{m+1} \cdot \mat e_1\left[\mathbb D_t\mat x_\Pi^{m+1}-\mat x_t^{m+1}\right], \mathbb D_t\mat E^{m+1} \left|\overline{\mat X}^{m+1}_\rho\right|^2\right) \nn\\
&~~~~+\left(\mat x^{m+1}\cdot \mat e_1\mat x_t^{m+1},  \mathbb D_t\mat E^{m+1} \left[\left|\overline{\mat X}^{m+1}_\rho\right|^2-\left|\mat x_\rho^{m+1}\right|^2\right]\right)\nn\\
&=:\sum_{i=1}^3\mathscr T_{1, i}\left(\mathbb D_t\mat E^{m+1}\right). 
\end{align}
By virtue of \eqref{eq:bdf_con_pf7}, \eqref{eqn:inequ1} and \eqref{eqn:assum2}, we have 
\begin{align}\label{eq:bdf_con_pf11}
	&\mathscr T_{1, 1}\left(\mathbb D_t\mat E^{m+1}\right)
	=\left(\left[\overline{\mat X}^{m+1}-\mat x^{m+1}\right] \cdot \mat e_1\mathbb D_t\mat x_\Pi^{m+1}, \mathbb D_t\mat E^{m+1} \left|\overline{\mat X}^{m+1}_\rho\right|^2\right) \nn\\
	&\leq \left\|\overline{\mat X}^{m+1}-\mat x^{m+1}\right\|_{L^2}\left\|\mathbb D_t\mat x_\Pi^{m+1}\right\|_{L^\infty}
	\left\|\mathbb D_t\mat E^{m+1}\right\|_{L^2}
	\left\|\overline{\mat X}^{m+1}_\rho\right\|_{L^\infty}^2\nn\\
	&\leq \epsilon\left\|\mathbb D_t\mat E^{m+1}\right\|_{L^2}^2+C_\epsilon\left\|\overline{\mat X}^{m+1}-\mat x^{m+1}\right\|_{L^2}^2.
\end{align}
To facilitate subsequent analysis, we define $\mathbb F^m=\left|\mat E^{m}\right|^2+\left|2\mat E^{m}-\mat E^{m-1}\right|^2$. Thanks to 
\begin{align}\label{eq:bdf_con_pf12}
\overline{\mat X}^{m+1}-\mat x^{m+1}
	=\bigg[\overline{\mat X}^{m+1}-\overline{\mat x}^{m+1}\bigg]+\bigg[\overline{\mat x}^{m+1}-\mat x^{m+1}\bigg]
	=-2\left(\mat E^m+\mat d^m\right)+\left(\mat E^{m-1}+\mat d^{m-1}\right)+\bigg[\overline{\mat x}^{m+1}-\mat x^{m+1}\bigg], 
\end{align}
and by using \eqref{eqn:inter_error} and Taylor's formula, we obtain
\begin{align}\label{eq:bdf_con_pf13}
	\Delta t\sum_{m=1}^n \mathscr T_{1, 1}\left(\mathbb D_t\mat E^{m+1}\right)\leq 
	\epsilon\,\Delta t\sum_{m=1}^n\left\|\mathbb D_t\mat E^{m+1}\right\|_{L^2}^2
	+C_\epsilon\Delta t\sum_{m=1}^n\left(1, \mathbb F^m\right)
	+C_\epsilon\left(\Delta t^4+h^4\right). 
\end{align}
Using \eqref{eqn:assum2} and \eqref{eq:bdf_con_pf7}, we have 
\begin{align}\label{eq:bdf_con_pf14}
&\Delta t\sum_{m=1}^n\mathscr T_{1, 2}\left(\mathbb D_t\mat E^{m+1}\right) =	\Delta t\sum_{m=1}^n\left(\mat x^{m+1} \cdot \mat e_1\left[\mathbb D_t\mat x_\Pi^{m+1}-\mat x_t^{m+1}\right], \mathbb D_t\mat E^{m+1} \left|\overline{\mat X}^{m+1}_\rho\right|^2\right) \nn\\
&\leq \Delta t \sum_{m=1}^n\left\|\mat x^{m+1}\right\|_{L^\infty}\left\|\mathbb D_t\mat x_\Pi^{m+1}-\mat x_t^{m+1}\right\|_{L^2}
\left\|\mathbb D_t\mat E^{m+1}\right\|_{L^2}\left\|\overline{\mat X}^{m+1}_\rho\right\|_{L^\infty}^2\nn\\
&\leq \epsilon\Delta t \sum_{m=1}^n\left\|\mathbb D_t\mat E^{m+1}\right\|_{L^2}^2+C_\epsilon\Delta t \sum_{m=1}^n
\left\|\mathbb D_t\mat x_\Pi^{m+1}-\mat x_t^{m+1}\right\|_{L^2}^2. 
\end{align}
By using Taylor's formula, we have 
\begin{align}\label{eq:bdf_con_pf15}
&	\mathbb D_t\mat x_\Pi^{m+1}-\mat x_t^{m+1}=\frac{3}{2\Delta t}\int_{t_m}^{t_{m+1}}\left[\mat x_{\Pi, t}-\mat x_t^{m+1}\right]dt
	-\frac{1}{2\Delta t}\int_{t_{m-1}}^{t_{m}}\left[\mat x_{\Pi, t}-\mat x_t^{m+1}\right]dt\nn\\
	& =\frac{3}{2\Delta t}\int_{t_m}^{t_{m+1}}\left[\mat x_{\Pi, t}-\mat x_t\right]dt
	+\frac{3}{2\Delta t}\int_{t_m}^{t_{m+1}}\left[\mat x_{t}-\mat x_t^{m+1}\right]dt
	-\frac{1}{2\Delta t}\int_{t_{m-1}}^{t_{m}}\left[\mat x_{\Pi, t}-\mat x_t\right]dt-\frac{1}{2\Delta t}\int_{t_{m-1}}^{t_{m}}\left[\mat x_{t}-\mat x_t^{m+1}\right]dt\nn\\
	& = \frac{3}{2\Delta t}\int_{t_m}^{t_{m+1}}\left[\mat x_{\Pi, t}-\mat x_t\right]dt
	+\frac{3}{2\Delta t}\int_{t_m}^{t_{m+1}}\left[\left(t-t_{m+1}\right)\mat x_{tt}^{m+1}+\frac{\left(t-t_{m+1}\right)^2}{2}\mat x_{ttt}\left(\cdot, \mat \xi_1\right)\right]dt\nn\\
	&~~~~-\frac{1}{2\Delta t}\int_{t_{m-1}}^{t_{m}}\left[\mat x_{\Pi, t}-\mat x_t\right]dt-\frac{1}{2\Delta t}\int_{t_{m-1}}^{t_{m}}\left[\left(t-t_{m+1}\right)\mat x_{tt}^{m+1}+\frac{\left(t-t_{m+1}\right)^2}{2}\mat x_{ttt}\left(\cdot, \mat \xi_2\right)\right]dt\nn\\
	& = \frac{3}{2\Delta t}\int_{t_m}^{t_{m+1}}\left[\mat x_{\Pi, t}-\mat x_t\right]dt-\frac{1}{2\Delta t}\int_{t_{m-1}}^{t_{m}}\left[\mat x_{\Pi, t}-\mat x_t\right]dt+\frac{3}{4\Delta t}\int_{t_m}^{t_{m+1}}\left(t-t_{m+1}\right)^2\mat x_{ttt}\left(\cdot, \mat \xi_1\right)dt\nn\\
	&~~~~-\frac{1}{4\Delta t}\int_{t_{m-1}}^{t_{m}}\left(t-t_{m+1}\right)^2\mat x_{ttt}\left(\cdot, \mat \xi_2\right)dt. 
\end{align}
Therefore, from \eqref{eqn:inter_error} and \eqref{eqn:assum2}, we get 
\begin{align}\label{eq:bdf_con_pf16}
	\left\|\mathbb D_t\mat x_\Pi^{m+1}-\mat x_t^{m+1}\right\|_{L^2}\leq  &\frac{3}{2\Delta t}\int_{t_m}^{t_{m+1}}\left\|\mat x_{\Pi, t}-\mat x_t\right\|_{L^2}dt+\frac{1}{2\Delta t}\int_{t_{m-1}}^{t_{m}}\left\|\mat x_{\Pi, t}-\mat x_t\right\|_{L^2}dt+\frac{3}{4\Delta t}\int_{t_m}^{t_{m+1}}\left(t-t_{m+1}\right)^2dt\max_{t\in [t_m, t_{m+1}]}\left\|\mat x_{ttt}\right\|_{L^2}\nn\\
	&+\frac{1}{4\Delta t}\int_{t_{m-1}}^{t_{m}}\left(t-t_{m+1}\right)^2dt\max_{t\in [t_{m-1}, t_{m}]}\left\|\mat x_{ttt}\right\|_{L^2}\nn\\
	\leq & Ch^2\max_{t\in [t_m, t_{m+1}]}\left|\mat x_t\right|_{H^2}+Ch^2\max_{t\in [t_{m-1}, t_{m}]}\left|\mat x_t\right|_{H^2}
	+\frac{\Delta t^2}{4}\max_{t\in [t_m, t_{m+1}]}\left\|\mat x_{ttt}\right\|_{L^2}+\frac{7\Delta t^2}{12}\max_{t\in [t_{m-1}, t_{m}]}\left\|\mat x_{ttt}\right\|_{L^2}\nn\\
	\leq & C\left(\Delta t^2+h^2\right). 
\end{align}
Using \eqref{eq:bdf_con_pf16} in \eqref{eq:bdf_con_pf14} gives that 
\begin{align}\label{eq:bdf_con_pf17}
	\Delta t\sum_{m=1}^n\mathscr T_{1, 2}\left(\mathbb D_t\mat E^{m+1}\right)\leq \epsilon\Delta t \sum_{m=1}^n\left\|\mathbb D_t\mat E^{m+1}\right\|_{L^2}^2+C_\epsilon\left(\Delta t^4+h^4\right).
\end{align}
To estimate the term $\mathscr T_{1, 3}\left(\mathbb D_t\mat E^{m+1}\right)$, we write 
\begin{align}\label{eq:bdf_con_pf18}
	\left|\overline{\mat X}^{m+1}_\rho\right|^2-\left|\mat x_\rho^{m+1}\right|^2
	=-2\overline{\mat d}_\rho^{m+1}\cdot\mat x_\rho^{m+1}+\mathcal U^{m+1}, 
\end{align}
where 
\begin{align*}
	\mathcal U^{m+1}:=\left(\mat x_\rho^{m+1}-\overline{\mat x}^{m+1}_\rho
	+\overline{\mat E}^{m+1}_\rho+\overline{\mat d}^{m+1}_\rho
	\right)^2
	-2\overline{\mat E}^{m+1}_\rho\cdot \mat x_\rho^{m+1}
	+2\left(\overline{\mat x}^{m+1}_\rho-\mat x_\rho^{m+1}\right)\cdot \mat x_\rho^{m+1}.
\end{align*}
Similar as \eqref{eqn:cn_con_pf20}, there holds 
\begin{align}\label{eq:bdf_con_pf19}
		\left\|\mathcal U^{m+1}\right\|_{L^2}\leq C\left(\left\|\overline{\mat E}_\rho^{m+1}\right\|_{L^2}+\Delta t^2+h^2\right). 
\end{align}
Using integration by parts, we have 
\begin{align}\label{eq:bdf_con_pf20}
&\mathscr T_{1, 3}\left(\mathbb D_t\mat E^{m+1}\right) = -2\left(\mat x^{m+1}\cdot \mat e_1\mat x_t^{m+1},  \mathbb D_t\mat E^{m+1} \overline{\mat d}_\rho^{m+1}\cdot\mat x_\rho^{m+1}\right)+\left(\mat x^{m+1}\cdot \mat e_1\mat x_t^{m+1},  \mathbb D_t\mat E^{m+1}\mathcal U^{m+1}\right)\nn\\
& = 2\left(\mat x_\rho^{m+1}\cdot \mat e_1\mat x_t^{m+1},  \mathbb D_t\mat E^{m+1} \overline{\mat d}^{m+1}\cdot\mat x_{\rho}^{m+1}\right)
+
2\left(\mat x^{m+1}\cdot \mat e_1\mat x_{t\rho}^{m+1},  \mathbb D_t\mat E^{m+1} \overline{\mat d}^{m+1}\cdot\mat x_{\rho}^{m+1}\right)\nn\\
&~~~~+2\left(\mat x^{m+1}\cdot \mat e_1\mat x_t^{m+1},  \mathbb D_t\mat E_\rho^{m+1} \overline{\mat d}^{m+1}\cdot\mat x_{\rho}^{m+1}\right)
+2\left(\mat x^{m+1}\cdot \mat e_1\mat x_t^{m+1},  \mathbb D_t\mat E^{m+1} \overline{\mat d}^{m+1}\cdot\mat x_{\rho\rho}^{m+1}\right)\nn\\
&~~~~+\left(\mat x^{m+1}\cdot \mat e_1\mat x_t^{m+1},  \mathbb D_t\mat E^{m+1}\mathcal U^{m+1}\right)\nn\\
&=:\left(\mathbb D_t\mat E^{m+1}, \mat{\mathcal V}^{m+1}\right)+\left(\mathbb D_t\mat E_\rho^{m+1}, \mat{\mathcal W}^{m+1}\right). 
\end{align}
where 
\begin{align*}
	& \mat{\mathcal V}^{m+1}:=2\mat x_\rho^{m+1}\cdot \mat e_1\overline{\mat d}^{m+1}\cdot\mat x_{\rho}^{m+1}\mat x_t^{m+1}+2\mat x^{m+1}\cdot \mat e_1\overline{\mat d}^{m+1}\cdot\mat x_{\rho}^{m+1}\mat x_{t\rho}^{m+1}
	+2\mat x^{m+1}\cdot \mat e_1\overline{\mat d}^{m+1}\cdot\mat x_{\rho\rho}^{m+1}\mat x_t^{m+1}+\mat x^{m+1}\cdot \mat e_1\mathcal U^{m+1}\mat x_t^{m+1},\nn\\
	&  \mat{\mathcal W}^{m+1}:=2\mat x^{m+1}\cdot \mat e_1\overline{\mat d}^{m+1}\cdot\mat x_{\rho}^{m+1}\mat x_t^{m+1}. 
\end{align*}
Obviously, from \eqref{eq:bdf_con_pf19}, and using \eqref{eqn:inter_error} and \eqref{eqn:assum2}, it holds that 
\begin{align}\label{eq:bdf_con_pf21}
	\left\|\mat{\mathcal V}^{m+1}\right\|_{L^2}\leq  C\left(\left\|\overline{\mat E}_\rho^{m+1}\right\|_{L^2}+\Delta t^2+h^2\right)\leq  C\left(\sqrt{\left(1, F^m\right)}+\Delta t^2+h^2\right), 
\end{align}
which further implies that 
\begin{align}\label{eq:bdf_con_pf21x}
\Delta t\sum_{m=1}^n\left(\mathbb D_t\mat E^{m+1}, \mathcal V^{m+1}\right)\leq \epsilon \Delta t\sum_{m=1}^n\left\|\mathbb D_t\mat E^{m+1}\right\|_{L^2}^2+C_\epsilon\Delta t\sum_{m=1}^n\left(1, F^m\right)+C_\epsilon\left(\Delta t^4+h^4\right).
\end{align}
In addition, we have 
\begin{align}\label{eq:bdf_con_pf22}
	&\Delta t\sum_{m=1}^n\left(\mathbb D_t\mat E_\rho^{m+1}, \mathcal W^{m+1}\right)=\frac{3\Delta t}{2}\sum_{m=1}^n\left(D_t\mat E_\rho^{m+\frac12}, \mathcal W^{m+1}\right)
	-\frac{\Delta t}{2}\sum_{m=1}^n\left(D_t\mat E_\rho^{m-\frac12}, \mathcal W^{m+1}\right)\nn\\
	&=\left(\widetilde{\mat E}_\rho^{n+\frac32}, \mat{\mathcal W}^{n+1}\right)
	-\Delta t\sum_{m=2}^n\left(\widetilde{\mat E}_\rho^{m+\frac12}, D_t\mat{\mathcal W}^{m+\frac12}\right)-\frac32\left(\mat E_\rho^1, \mat{\mathcal W}^2\right). 
\end{align}
Similar as \eqref{eqn:cn_con_pf24} and \eqref{eqn:cn_con_pf25}, we also have 
\begin{align}\label{eq:bdf_con_pf23}
	\left\| \mat{\mathcal W}^{n+1}\right\|_{L^2}\leq Ch^2;\qquad 
	\left\| D_t\mat{\mathcal W}^{m+\frac12}\right\|_{L^2}\leq Ch^2,\quad 
	2\leq m\leq n;\qquad \left\| \mat{\mathcal W}^{2}\right\|_{L^2}\leq Ch^2. 
\end{align}
Therefore, from \eqref{eq:bdf_con_pf22} and \eqref{eq:bdf_con_pf23}, and since 
\begin{align}\label{eq:bdf_con_pf23add}
	\left\|\widetilde{\mat E}_\rho^{j+\frac12}\right\|_{L^2}
	=\left\|\frac{\mat E_\rho^{j}}{2}+\frac{2\mat E_\rho^{j}-\mat E_\rho^{j-1}}{2}\right\|_{L^2}
	\leq 
	\frac12\left(\left\|\mat E_\rho^{j}\right\|_{L^2}+\left\|2\mat E_\rho^{j}-\mat E_\rho^{j-1}\right\|_{L^2}\right)\leq 
	\frac{\sqrt{2}}{2}\sqrt{\left(1, F^{j}\right)},\qquad j=2,\ldots,n+1,
\end{align}
we can obtain 
\begin{align}\label{eq:bdf_con_pf24}
	\Delta t\sum_{m=1}^n\left(\mathbb D_t\mat E_\rho^{m+1}, \mathcal W^{m+1}\right)\leq 
	\epsilon\left(1, F^{n+1}\right)+C\Delta t\sum_{m=2}^n\left(1, F^{m}\right)+C\left\|\mat E_\rho^1\right\|_{L^2}^2+
	C_\epsilon\left(\Delta t^4+h^4\right).
\end{align}
By using \eqref{eq:bdf_con_pf21x} and \eqref{eq:bdf_con_pf24}, we derive 
\begin{align}\label{eq:bdf_con_pf25}
\Delta t\sum_{m=1}^n\mathscr T_{1, 3}\left(\mathbb D_t\mat E^{m+1}\right)\leq 	\epsilon\left(1, F^{n+1}\right)+
\epsilon \Delta t\sum_{m=1}^n\left\|\mathbb D_t\mat E^{m+1}\right\|_{L^2}^2+C_\epsilon\Delta t\sum_{m=1}^n\left(1, F^m\right)+C\left\|\mat E_\rho^1\right\|_{L^2}^2+C_\epsilon\left(\Delta t^4+h^4\right).
\end{align}
From \eqref{eq:bdf_con_pf13}, \eqref{eq:bdf_con_pf17} and \eqref{eq:bdf_con_pf25}, we conclude that 
\begin{align}\label{eq:bdf_con_pf26}
	\Delta t\sum_{m=1}^n	\mathscr T_{1}\left(\mathbb D_t\mat E^{m+1}\right)\leq& 
	\epsilon\left(1, F^{n+1}\right)+
	\epsilon \Delta t\sum_{m=1}^n\left\|\mathbb D_t\mat E^{m+1}\right\|_{L^2}^2+C_\epsilon\Delta t\sum_{m=1}^n\left(1, F^m\right)+C_\epsilon\Delta t\sum_{m=1}^n\left(1, \mathbb F^m\right)\nn\\
	&+C\left\|\mat E_\rho^1\right\|_{L^2}^2+C_\epsilon\left(\Delta t^4+h^4\right).
\end{align}
In what follows, we investigate the terms involving $\mathscr T_2\left(\mathbb D_t\mat E^{m+1}\right)$. 
From \eqref{eqn:cn_con_pf29}, we can split 
\begin{align}\label{eq:bdf_con_pf27}
	\mathscr T_2\left(\mathbb D_t\mat E^{m+1}\right)&=\left(\overline{\mat X}^{m+1} \cdot \mat e_1\mat x^{m+1}_{\Pi, \rho}, \mathbb D_t\mat E^{m+1}_\rho\right)-\left(\mat x^{m+1}\cdot\mat e_1\mat x^{m+1}_\rho, \mathbb D_t\mat E_\rho^{m+1}\right)\nn\\
	&=-\left(\left[\overline{\mat X}^{m+1} \cdot \mat e_1-P^h\left(\overline{\mat X}^{m+1} \cdot \mat e_1\right)\right]
	\mat d^{m+1}_{\rho}, \mathbb D_t\mat E^{m+1}_\rho\right)+\left(\left[\overline{\mat X}^{m+1}-\overline{\mat x}^{m+1}\right] \cdot \mat e_1\mat x^{m+1}_{\rho}, \mathbb D_t\mat E^{m+1}_\rho\right)\nn\\
	&~~~~
	+\left(\left[\overline{\mat x}^{m+1}-\mat x^{m+1}\right] \cdot \mat e_1\mat x^{m+1}_{\rho}, \mathbb D_t\mat E^{m+1}_\rho\right)\nn\\
	&=:\sum_{i=1}^3\mathscr T_{2, i}\left(\mathbb D_t\mat E^{m+1}\right). 
\end{align}
For convenience, we denote 
\begin{align}\label{eq:bdf_con_pf28}
	\Delta t\sum_{m=1}^n\mathscr T_{2, 1}\left(\mathbb D_t\mat E^{m+1}\right)=-\Delta t\sum_{m=1}^n\left(\left[\overline{\mat X}^{m+1} \cdot \mat e_1-P^h\left(\overline{\mat X}^{m+1} \cdot \mat e_1\right)\right]
	\mat d^{m+1}_{\rho}, \mathbb D_t\mat E^{m+1}_\rho\right)
	=:-\Delta t\sum_{m=1}^n\left(\mat{\mathcal G}^{m+1}, \mathbb D_t\mat E^{m+1}_\rho\right). 
\end{align}
Using \eqref{eq:bdf_con_pf22}, we further obtain 
\begin{align}\label{eq:bdf_con_pf29}
	\Delta t\sum_{m=1}^n\mathscr T_{2, 1}\left(\mathbb D_t\mat E^{m+1}\right)= -\left(\mat{\mathcal G}^{n+1}, \widetilde{\mat E}_\rho^{n+\frac32}\right)+\Delta t\sum_{m=2}^n\left(D_t\mat {\mathcal G}^{m+\frac12}, \widetilde{\mat E}_\rho^{m+\frac12}\right)
	+\frac32\left(\mat{\mathcal G}^2, \mat E_\rho^1\right). 
\end{align}
Similar as \eqref{eqn:cn_con_pf33}-\eqref{eqn:cn_con_pf36}, we have 
\begin{align}\label{eq:bdf_con_pf30}
	\left\|\mat{\mathcal G}^{n+1}\right\|_{L^2}\leq Ch^2;\quad 
	\left\|D_t\mat {\mathcal G}^{m+\frac12}\right\|_{L^2}\leq Ch^2,\quad 2\leq m\leq n;\quad \left\|\mat{\mathcal G}^2\right\|_{L^2}\leq Ch^2. 
\end{align}
Hence, using \eqref{eq:bdf_con_pf23add} and \eqref{eq:bdf_con_pf30}, we obtain 
\begin{align}\label{eq:bdf_con_pf31}
	&\Delta t\sum_{m=1}^n\mathscr T_{2, 1}\left(\mathbb D_t\mat E^{m+1}\right)\leq \left\|\mat{\mathcal G}^{n+1}\right\|_{L^2}
	\left\|\widetilde{\mat E}_\rho^{n+\frac32}\right\|_{L^2}
	+
	\Delta t\sum_{m=2}^n\left\|D_t\mat {\mathcal G}^{m+\frac12}\right\|_{L^2} \left\|\widetilde{\mat E}_\rho^{m+\frac12}\right\|_{L^2}
	+\frac32\left\|\mat{\mathcal G}^2\right\|_{L^2}\left\|\mat E_\rho^1\right\|_{L^2}\nn\\
	&\leq \frac{\sqrt{2}}{2}\left\|\mat{\mathcal G}^{n+1}\right\|_{L^2}\sqrt{\left(1, F^{n+1}\right)}+
	\frac{\sqrt{2}}{2}\Delta t\sum_{m=2}^n\left\|D_t\mat {\mathcal G}^{m+\frac12}\right\|_{L^2} \sqrt{\left(1, F^{m}\right)}+\frac32\left\|\mat{\mathcal G}^2\right\|_{L^2}\left\|\mat E_\rho^1\right\|_{L^2}\nn\\
	&\leq \epsilon\left(1, F^{n+1}\right)
	+C\Delta t\sum_{m=2}^n\left(1, F^{m}\right)+C\left\|\mat E_\rho^1\right\|_{L^2}^2+C_\epsilon h^4.
\end{align}
Next, for the term involving $\mathscr T_{2, 2}\left(\mathbb D_t\mat E^{m+1}\right)$, using integration by parts, we denote
\begin{align}\label{eq:bdf_con_pf32}
	&\Delta t\sum_{m=1}^n\mathscr T_{2, 2}\left(\mathbb D_t\mat E^{m+1}\right)=\Delta t\sum_{m=1}^n\left(\left[\overline{\mat X}^{m+1}-\overline{\mat x}^{m+1}\right] \cdot \mat e_1\mat x^{m+1}_{\rho}, \mathbb D_t\mat E^{m+1}_\rho\right)\nn\\
	&=-\Delta t\sum_{m=1}^n\left(\overline{\mat d}^{m+1} \cdot \mat e_1\mat x^{m+1}_{\rho}, \mathbb D_t\mat E^{m+1}_\rho\right)+\Delta t\sum_{m=1}^n\left(\overline{\mat E}_\rho^{m+1} \cdot \mat e_1\mat x^{m+1}_{\rho}+\overline{\mat E}^{m+1} \cdot \mat e_1\mat x^{m+1}_{\rho\rho}, \mathbb D_t\mat E^{m+1}\right)\nn\\
	& =: -\Delta t\sum_{m=1}^n\left(\mat{\mathcal L}^{m+1}, \mathbb D_t\mat E^{m+1}_\rho\right)+\Delta t\sum_{m=1}^n\left(\overline{\mat E}_\rho^{m+1} \cdot \mat e_1\mat x^{m+1}_{\rho}+\overline{\mat E}^{m+1} \cdot \mat e_1\mat x^{m+1}_{\rho\rho}, \mathbb D_t\mat E^{m+1}\right). 
\end{align}
For the first term on the right-hand side of \eqref{eq:bdf_con_pf32}, by using \eqref{eq:bdf_con_pf22}, we have 
\begin{align}\label{eq:bdf_con_pf33}
	-\Delta t\sum_{m=1}^n\left(\mat{\mathcal L}^{m+1}, \mathbb D_t\mat E^{m+1}_\rho\right)=-\left(\mat{\mathcal L}^{n+1}, \widetilde{\mat E}_\rho^{n+\frac32}\right)+\Delta t\sum_{m=2}^n\left(D_t\mat {\mathcal L}^{m+\frac12}, \widetilde{\mat E}_\rho^{m+\frac12}\right)
	+\frac32\left(\mat{\mathcal L}^2, \mat E_\rho^1\right). 
\end{align}
Similar as \eqref{eqn:cn_con_pf33}-\eqref{eqn:cn_con_pf36}, there hold  
\begin{align}\label{eq:bdf_con_pf34}
	\left\|\mat{\mathcal L}^{n+1}\right\|_{L^2}\leq Ch^2;\quad 
	\left\|D_t\mat {\mathcal L}^{m+\frac12}\right\|_{L^2}\leq Ch^2,\quad 2\leq m\leq n;\quad \left\|\mat{\mathcal L}^2\right\|_{L^2}\leq Ch^2. 
\end{align}
Therefore, similar as \eqref{eq:bdf_con_pf31}, we have 
\begin{align}\label{eq:bdf_con_pf35}
	-\Delta t\sum_{m=1}^n\left(\mat{\mathcal L}^{m+1}, \mathbb D_t\mat E^{m+1}_\rho\right)\leq \epsilon\left(1, F^{n+1}\right)
	+C\Delta t\sum_{m=2}^n\left(1, F^{m}\right)+C\left\|\mat E_\rho^1\right\|_{L^2}^2+C_\epsilon h^4.
\end{align}
 Thanks to \eqref{eqn:assum2}, the second term on the right-hand side of \eqref{eq:bdf_con_pf32} can be bounded by 
\begin{align}\label{eq:bdf_con_pf36}
	&\Delta t\sum_{m=1}^n\left(\overline{\mat E}_\rho^{m+1} \cdot \mat e_1\mat x^{m+1}_{\rho}+\overline{\mat E}^{m+1} \cdot \mat e_1\mat x^{m+1}_{\rho\rho}, \mathbb D_t\mat E^{m+1}\right)\nn\\
  &	\leq \epsilon \Delta t\sum_{m=1}^n\left\| \mathbb D_t\mat E^{m+1}\right\|_{L^2}^2
	+C_\epsilon\Delta t\sum_{m=1}^n\left\| \overline{\mat E}_\rho^{m+1}\right\|_{L^2}^2+C_\epsilon\Delta t\sum_{m=1}^n\left\| \overline{\mat E}^{m+1}\right\|_{L^2}^2\nn\\
	&\leq \epsilon \Delta t\sum_{m=1}^n\left\| \mathbb D_t\mat E^{m+1}\right\|_{L^2}^2
	+C_\epsilon\Delta t\sum_{m=1}^n\left(1, F^m\right)+C_\epsilon\Delta t\sum_{m=1}^n\left(1, \mathbb F^m\right). 
\end{align}
Using \eqref{eq:bdf_con_pf35} and \eqref{eq:bdf_con_pf36} in \eqref{eq:bdf_con_pf32} gives that 
\begin{align}\label{eq:bdf_con_pf37}
	&\Delta t\sum_{m=1}^n\mathscr T_{2, 2}\left(\mathbb D_t\mat E^{m+1}\right)\leq \epsilon\left(1, F^{n+1}\right)
+ \epsilon \Delta t\sum_{m=1}^n\left\| \mathbb D_t\mat E^{m+1}\right\|_{L^2}^2	+C_\epsilon\Delta t\sum_{m=1}^n\left(1, F^m\right)+C_\epsilon\Delta t\sum_{m=1}^n\left(1, \mathbb F^m\right)+C\left\|\mat E_\rho^1\right\|_{L^2}^2+C_\epsilon h^4.
\end{align}
For the term involving $\mathscr T_{2, 3}\left(\mathbb D_t\mat E^{m+1}\right)$, using integration by parts and Taylor's formula, we have 
\begin{align}\label{eq:bdf_con_pf38}
\Delta t\sum_{m=1}^n\mathscr T_{2, 3}\left(\mathbb D_t\mat E^{m+1}\right)&=	-\Delta t\sum_{m=1}^n\left(\left[\overline{\mat x}_\rho^{m+1}-\mat x_\rho^{m+1}\right] \cdot \mat e_1\mat x^{m+1}_{\rho}, \mathbb D_t\mat E^{m+1}\right)-\Delta t\sum_{m=1}^n\left(\left[\overline{\mat x}^{m+1}-\mat x^{m+1}\right] \cdot \mat e_1\mat x^{m+1}_{\rho\rho}, \mathbb D_t\mat E^{m+1}\right)\nn\\
&\leq \epsilon \Delta t\sum_{m=1}^n\left\| \mathbb D_t\mat E^{m+1}\right\|_{L^2}^2+C_\epsilon\Delta t^4. 
\end{align}
Combining \eqref{eq:bdf_con_pf31}, \eqref{eq:bdf_con_pf37} and \eqref{eq:bdf_con_pf38}, we obtain 
\begin{align}\label{eq:bdf_con_pf39}
	\Delta t\sum_{m=1}^n\mathscr T_{2}\left(\mathbb D_t\mat E^{m+1}\right)\leq &\epsilon\left(1, F^{n+1}\right)
	+ \epsilon \Delta t\sum_{m=1}^n\left\| \mathbb D_t\mat E^{m+1}\right\|_{L^2}^2	+C_\epsilon\Delta t\sum_{m=1}^n\left(1, \mathbb F^m\right)+C_\epsilon\Delta t\sum_{m=1}^n\left(1, F^m\right)\nn\\
&	+C\left\|\mat E_\rho^1\right\|_{L^2}^2+C_\epsilon\left(\Delta t^4+ h^4\right). 
\end{align}
For the term with respect to $\mathscr T_{3}\left(\mathbb D_t\mat E^{m+1}\right)$, it follows from similar process as $\mathscr T_{1, 3}\left(\mathbb D_t\mat E^{m+1}\right)$ that 
\begin{align}\label{eq:bdf_con_pf40}
\Delta t\sum_{m=1}^n\mathscr T_{3}\left(\mathbb D_t\mat E^{m+1}\right)\leq 	\epsilon\left(1, F^{n+1}\right)+
\epsilon \Delta t\sum_{m=1}^n\left\|\mathbb D_t\mat E^{m+1}\right\|_{L^2}^2+C_\epsilon\Delta t\sum_{m=1}^n\left(1, F^m\right)+C\left\|\mat E_\rho^1\right\|_{L^2}^2+C_\epsilon\left(\Delta t^4+h^4\right).
\end{align}
For the last term on the right-hand side of \eqref{eq:bdf_con_pf6}, thanks to \eqref{eq:bdf_con_pf7}, we have 
\begin{align}\label{eq:bdf_con_pf41}
	\frac{1}{4}\left(\overline{\mat X}^2\cdot\mat e_1, F^1\right)\leq 
	\frac{1}{4}\left\|\overline{\mat X}^2\right\|_{L^\infty}\left(1, F^1\right)\leq \frac{5C_0}{2}\left\|\mat E_\rho^1\right\|_{L^2}^2. 
\end{align}

Using \eqref{eq:bdf_con_pf8}, \eqref{eq:bdf_con_pf26}, \eqref{eq:bdf_con_pf39}, \eqref{eq:bdf_con_pf40} and \eqref{eq:bdf_con_pf41} in \eqref{eq:bdf_con_pf7}, thanks to \eqref{eqn:ini}, by selecting sufficiently small $\epsilon$, we conclude that 
\begin{align}\label{eq:bdf_con_pf42}
	\Delta t\sum_{m=1}^n\left\|\mathbb D_t\mat E^{m+1}\right\|_{L^2}^2+
	\left(1, F^{n+1}\right)
	\leq C\Delta t\sum_{m=1}^n\left(1, \mathbb F^m\right)+C_\epsilon\Delta t\sum_{m=1}^n\left(1, F^m\right)
	+C\left(\Delta t^4+h^4\right).
\end{align}
Thanks to 
\begin{align}\label{eq:bdf_con_pf43}
	\left(\mat E^{m+1}, \mathbb D_t\mat E^{m+1}\right)
	&=\frac{1}{4\Delta t}\left[\left(\left\|\mat E^{m+1}\right\|_{L^2}^2-\left\|\mat E^{m}\right\|_{L^2}^2\right)+\left(\left\|2\mat E^{m+1}-\mat E^m\right\|_{L^2}^2-\left\|2\mat E^{m}-\mat E^{m-1}\right\|_{L^2}^2\right)+\left\|\mat E^{m+1}-2\mat E^m+\mat E^{m-1}\right\|_{L^2}^2\right]\nn\\
	&\geq \frac{1}{4\Delta t}\left[\left(\left\|\mat E^{m+1}\right\|_{L^2}^2-\left\|\mat E^{m}\right\|_{L^2}^2\right)+\left(\left\|2\mat E^{m+1}-\mat E^m\right\|_{L^2}^2-\left\|2\mat E^{m}-\mat E^{m-1}\right\|_{L^2}^2\right)\right]\nn\\
	&=\frac{\left(1, \mathbb F^{m+1}\right)-\left(1, \mathbb F^m\right)}{4\Delta t}, 
\end{align}
we have 
\begin{align}\label{eq:bdf_con_pf44}
	\left(1, \mathbb F^{n+1}\right)&=\sum_{m=1}^n\bigg[\left(1, \mathbb F^{m+1}\right)-\left(1, \mathbb F^m\right)\bigg]+\left(1, \mathbb F^1\right)\leq 4\Delta t\sum_{m=1}^n\left(\mat E^{m+1}, \mathbb D_t\mat E^{m+1}\right)+\left(1, \mathbb F^1\right)\nn\\
	&\leq \epsilon \Delta t\sum_{m=1}^n\left\|\mathbb D_t\mat E^{m+1}\right\|_{L^2}^2
	+C_\epsilon \Delta t\sum_{m=1}^n\left\|\mat E^{m+1}\right\|_{L^2}^2+\left(1, \mathbb F^1\right)\nn\\
	&\leq \epsilon \Delta t\sum_{m=1}^n\left\|\mathbb D_t\mat E^{m+1}\right\|_{L^2}^2
	+C_\epsilon \Delta t\sum_{m=1}^n\left(1, \mathbb F^{m+1}\right)+C\left\|\mat E^1\right\|_{L^2}^2. 
\end{align}
By choosing sufficient small $\Delta t$ in \eqref{eq:bdf_con_pf44}, we have
\begin{align}\label{eq:bdf_con_pf45}
	\left(1, \mathbb F^{n+1}\right)\leq \epsilon \Delta t\sum_{m=1}^n\left\|\mathbb D_t\mat E^{m+1}\right\|_{L^2}^2
	+C_\epsilon \Delta t\sum_{m=1}^n\left(1, \mathbb F^{m}\right)+C\left\|\mat E^1\right\|_{L^2}^2. 
\end{align}
Taking the sum of \eqref{eq:bdf_con_pf43} and \eqref{eq:bdf_con_pf45}, and selecting a sufficient small $\epsilon$, we obtain 
\begin{align}\label{eq:bdf_con_pf46}
	\left(1, \mathbb F^{n+1}\right)+\left(1, F^{n+1}\right)
	\leq C\Delta t\sum_{m=1}^n\bigg[\left(1, \mathbb F^m\right)+\left(1, F^m\right)\bigg]+C\left(\Delta t^4+h^4\right).
\end{align}
By using the discrete Gronwall inequality in \eqref{eq:bdf_con_pf46}, if $\Delta t$ is selected sufficiently small, we can obtain  
\begin{align}\label{eq:bdf_con_pf47}
	\left(1, \mathbb F^{n+1}\right)+\left(1, F^{n+1}\right)
	\leq C\left(\Delta t^4+h^4\right),
\end{align}
which immediately implies that 
\begin{align}\label{eq:bdf_con_pf48}
		\left\|\mat E^{n+1}\right\|_{H^1}\leq C\left(\Delta t^2+h^2\right). 
\end{align}
Therefore, we have completed the proof. 

\section{Numerical results}\label{sec5}
In this section we present several numerical experiments to test the CN method and the BDF2 method.

\begin{exa}
We in this example test the convergence order for the CN method and BDF2 method for an evolving torus. We add a right-hand source term $\mat f$ of the system \eqref{eqn:model} by selecting the exact solution 
\begin{equation}\label{eq:xx}
	\mat x(\rho, t) =
	\left(
	\begin{matrix}
		g(t)+\cos(2\pi\rho)\\
		\sin(2\pi\rho)
	\end{matrix}
	\right),\qquad g(t)=2+\sin(\pi t). 
\end{equation}
In the following tests, to check the spatial convergence order, we use a temporally refined discretization with $M=10000$; conversely, for testing temporal convergence order, a spatially refined grid with $J=50000$ is employed.
In this example, we set $T=1$. The
results in Tables \ref{tab1}-\ref{tab4} confirm the optimal convergence rates given in Theorem \ref{thm:main}.

\begin{table}[!h]
	\renewcommand{\arraystretch}{1.2}
	\caption{The errors and spatial convergence order of the CN method.}\label{tab1}
	\begin{center}
		\setlength{\tabcolsep}{5mm}{
			\begin{tabular}{|c|cccc|}\hline
				$h$ & $\max\limits_{m=0,\ldots,M}\left\|\mat x^m-\mat X^m\right\|_{L^2}$ & order & $\max\limits_{m=0,\ldots,M}\left|\mat x^m-\mat X^m\right|_{H^1}$& order \\ \hline
			 32&2.9849e-03&-- &6.1671e-01&--
				
				\\
				64&7.4381e-04 &2.0047 &3.0841e-01&0.9998
				\\
				128& 1.8582e-04& 2.0010&1.5421e-01& 0.9999 \\
				256& 4.6461e-05& 1.9998&7.7106e-02&1.0000\\
			512& 1.1631e-05& 1.9980&3.8553e-02&1.0000
				\\\hline
		\end{tabular}}
	\end{center}
\end{table}

\begin{table}[!h]
	\renewcommand{\arraystretch}{1.2}
	\caption{The errors and temporal convergence order of the CN method.}\label{tab2}
	\begin{center}
		\setlength{\tabcolsep}{5mm}{
			\begin{tabular}{|c|cccc|}\hline
				$h$ & $\max\limits_{m=0,\ldots,M}\left\|\mat x^m-\mat X^m\right\|_{L^2}$ & order & $\max\limits_{m=0,\ldots,M}\left|\mat x^m-\mat X^m\right|_{H^1}$& order \\ \hline
				8&4.8655e-02&-- &1.7852e-01&--
				
				\\
				16&1.3066e-02 &1.8967 &3.7061e-02&2.2681
				\\
				32& 3.2908e-03& 1.9893&9.1971e-03&2.0106 \\
				64& 8.2149e-04& 2.0021&2.2903e-03&2.0056\\
				128& 2.0481e-04& 2.0039&6.8269e-04&1.7463
				\\\hline
		\end{tabular}}
	\end{center}
\end{table}

\begin{table}[!h]
	\renewcommand{\arraystretch}{1.2}
	\caption{The errors and spatial convergence order of the BDF2 method.}\label{tab3}
	\begin{center}
		\setlength{\tabcolsep}{5mm}{
			\begin{tabular}{|c|cccc|}\hline
				$h$ & $\max\limits_{m=0,\ldots,M}\left\|\mat x^m-\mat X^m\right\|_{L^2}$ & order & $\max\limits_{m=0,\ldots,M}\left|\mat x^m-\mat X^m\right|_{H^1}$& order \\ \hline
				32&2.9852e-03&-- &6.6171e-01&--
				
				\\
				64&7.4389e-04 &2.0046 &3.0841e-01&0.9998
				\\
				128& 1.8585e-04& 2.0010&1.5421e-01& 0.9999 \\
				256& 4.6476e-05& 1.9995&7.7106e-02&1.0000\\
				512& 1.1643e-05& 1.9997&3.8553e-02&1.0000
				\\\hline
		\end{tabular}}
	\end{center}
\end{table} 

\begin{table}[!h]
	\renewcommand{\arraystretch}{1.2}
	\caption{The errors and temporal convergence order of the BDF2 method.}\label{tab4}
	\begin{center}
		\setlength{\tabcolsep}{5mm}{
			\begin{tabular}{|c|cccc|}\hline
				$h$ & $\max\limits_{m=0,\ldots,M}\left\|\mat x^m-\mat X^m\right\|_{L^2}$ & order & $\max\limits_{m=0,\ldots,M}\left|\mat x^m-\mat X^m\right|_{H^1}$& order \\ \hline
				8&9.6879e-02&-- &3.5874e-01&--
				
				\\
				16&2.4075e-02 &2.0086 &7.4390e-02&2.2698
				\\
				32& 5.4847e-03&2.1341&1.6369e-02& 2.1842 \\
				64& 1.2957e-03& 2.0817&3.7843e-03&2.1128\\
				128& 3.1887e-04& 2.0454&9.7916e-04&1.9504
				\\\hline
		\end{tabular}}
	\end{center}
\end{table} 
\end{exa}

\begin{exa}
	Motivated by \cite{barrett2021finite}, we revisit the evolution of the initial surface defined by the set  
	\[
	\mS(0) := \left\{ \mat z \in \mathbb{R}^3: \left( 1 - \left| \mat z - (\mat z \cdot \mat e_1) \mat e_2 \right| \right)^2 + (\mat z \cdot \mat e_2)^2 = r^2, \quad 0 < r < 1 \right\}.
	\]  
	We denote \( T_r \) as the time at which the surface \( \mS(t) \) becomes singular. 
	As observed by Soner \& Souganidis \cite{soner1993singularities}, there exists a critical value \( r_0 \in (0, 1) \) such that for \( r \in (0, r_0) \), the solution contracts to a circle at time \( T_r \), whereas for \( r \in (r_0, 1) \), the solution closes the hole at time \( T_r \). We mainly do the following tests:
	\begin{itemize}
		\item Firstly, by setting \( r = 0.7 \), we observe from Fig. \ref{figure:Exa2_1} that the surface gradually closes up, eventually forming a genus-0 surface at \( t = 0.081 \), indicating the disappearance of the hole.  
		Additionally, we conduct the same numerical experiment with a smaller radius of \( r = 0.5 \). Unlike the case of \( r = 0.7 \), Fig. \ref{figure:Exa2_2} shows that the surface evolves by shrinking towards a circular shape, reaching this form at \( t = 0.136 \).
		\begin{figure}[!h]
			\centering
			\includegraphics[width=0.8\textwidth]{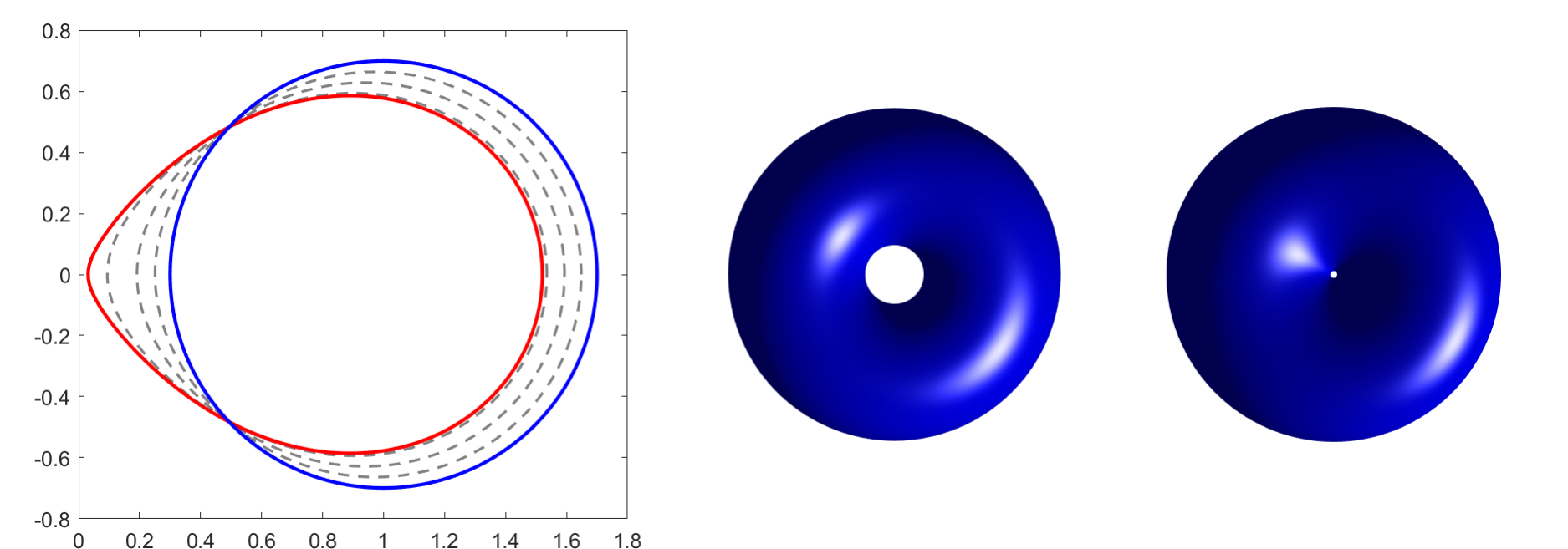}\\
			
			\includegraphics[width=0.8\textwidth]{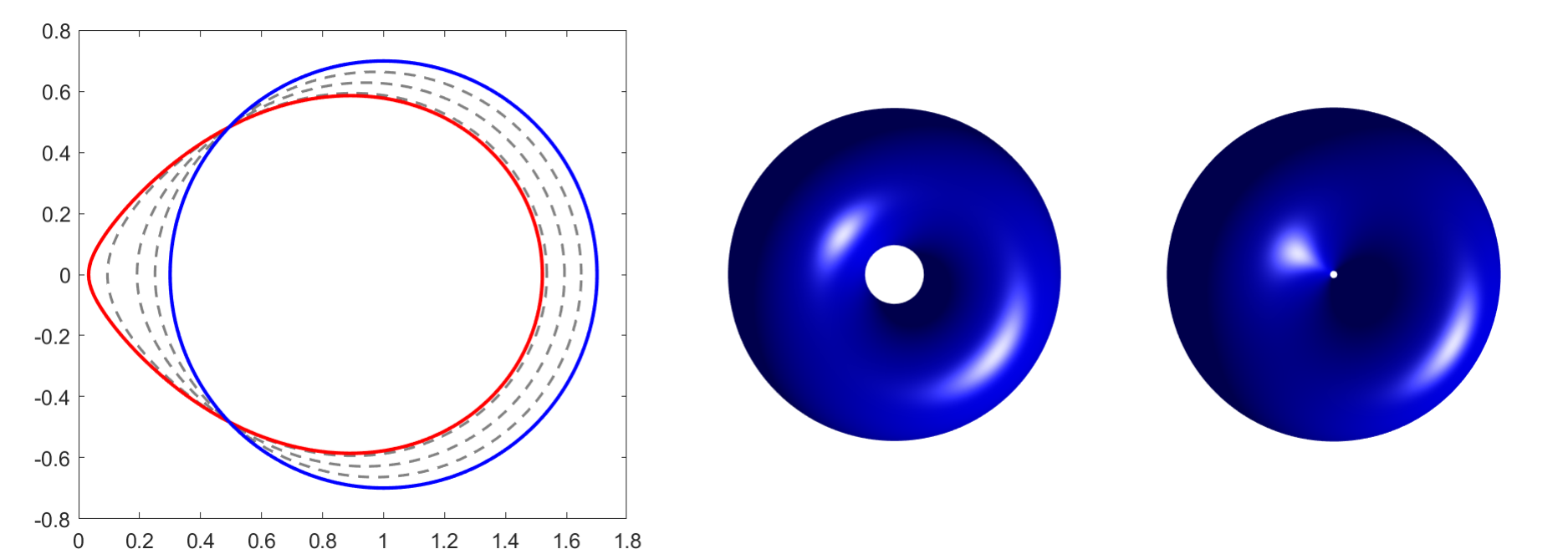}
			\caption{Evolution for a torus using the CN method (first row) and the BDF2 method (second row) with $r=0.7$. Plots at the times $t=0, 0.025, 0.05, 0.075, 0.081$, and visualizations of the axisymmetric surfaces generated by $t=0$ and $t=0.081$. Here, $J=512$, $\Delta t=10^{-4}$.}
			\label{figure:Exa2_1}
		\end{figure}
		
		\begin{figure}[!h]
			\centering
			\includegraphics[width=0.8\textwidth]{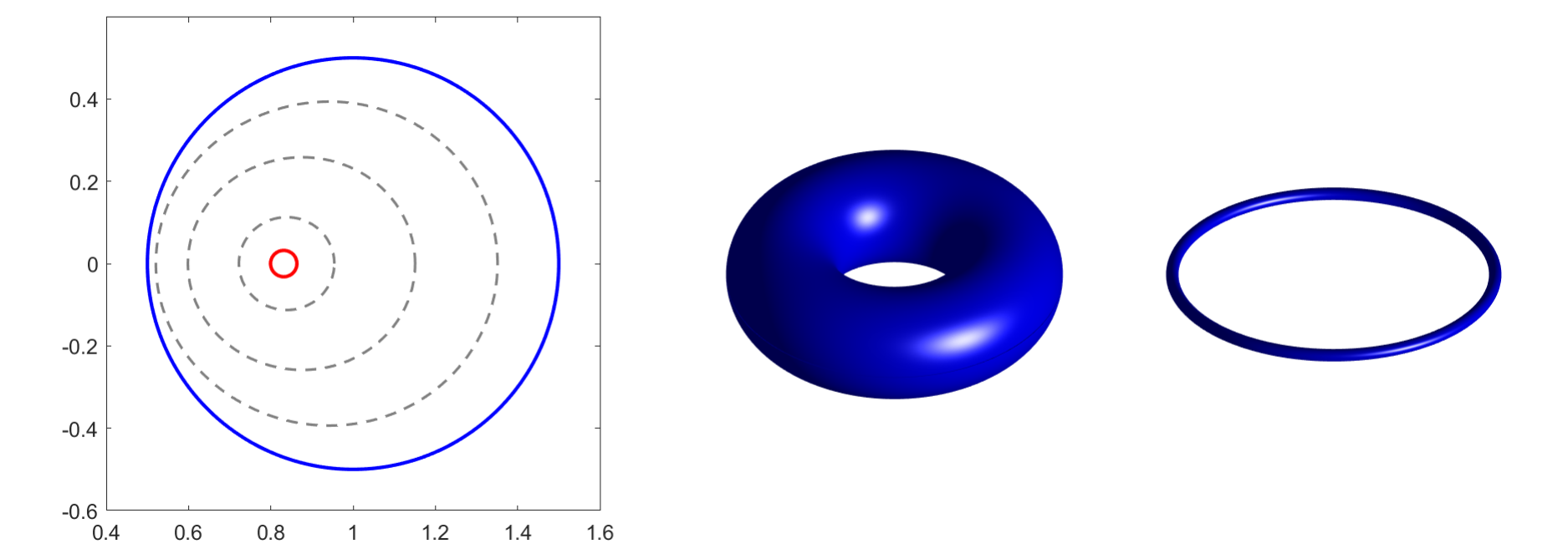}\\
			\includegraphics[width=0.8\textwidth]{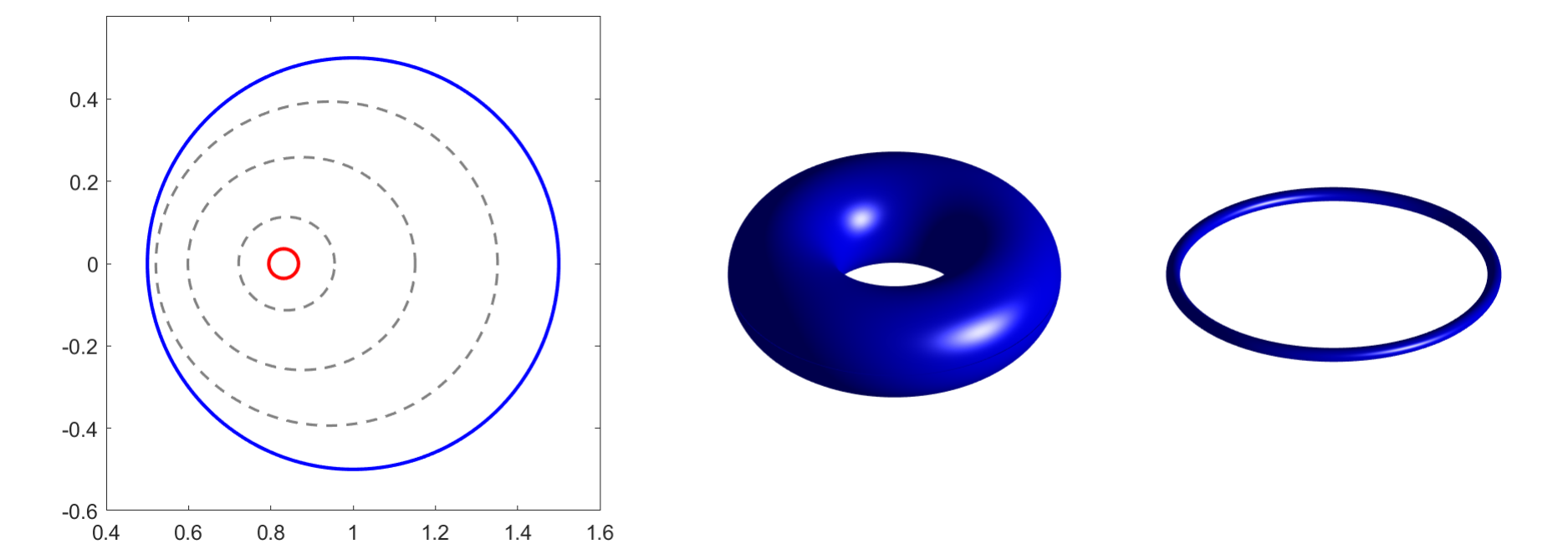}
			\caption{Evolution for a torus using the CN method (first row) and the BDF2 method (second row) with $r=0.7$. Plots at the times $t=0, 0.05, 0.1, 0.13, 0.136$, and visualizations of the axisymmetric surfaces generated by $t=0$ and $t=0.136$. Here, $J=512$, $\Delta t=10^{-4}$.}
			\label{figure:Exa2_2}
		\end{figure}

	    \item 
	   Secondly, we evaluate the mesh quality for the two types of second-order temporal methods. To this end, we define the mesh ratio as  
	   \[
	   \mathcal{R}^m := \frac{\max\limits_{j=1,\ldots,J} \left| \mathbf{X}^m(q_j) - \mathbf{X}^m(q_{j-1}) \right|}{\min\limits_{j=1,\ldots,J} \left| \mathbf{X}^m(q_j) - \mathbf{X}^m(q_{j-1}) \right|}.
	   \]  
	  As shown in Fig. \ref{figure:Exa2_3} for \( r = 0.7 \) and \( r = 0.5 \), the mesh quality of both methods remains relatively good and is consistent with that of the first-order temporal method \cite{barrett2021finite}. However, we emphasize that if a second-order time-stepping method is constructed based on the BGN approach using an extrapolation technique, the mesh quality would deteriorate significantly. Our final example further illustrates this issue.  
	  Among these figures, in the case of \( r = 0.5 \), we also observe a rapid increase in the mesh ratio, which may be attributed to the surface gradually approaching the \( z \)-axis over time.
	   
	    \begin{figure}[!h]
	    	\centering
	    	\includegraphics[width=0.45\textwidth]{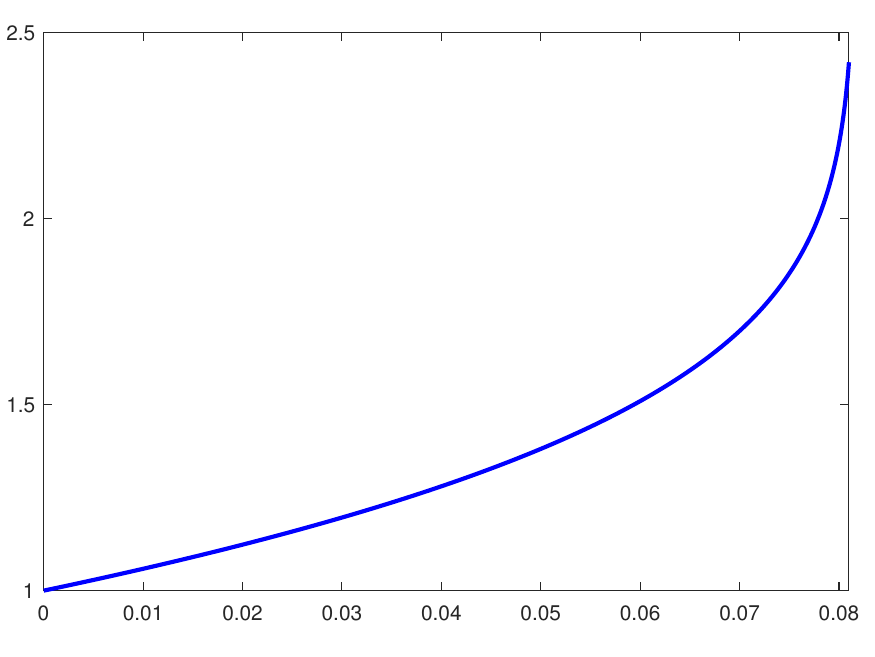}
	    	\includegraphics[width=0.45\textwidth]{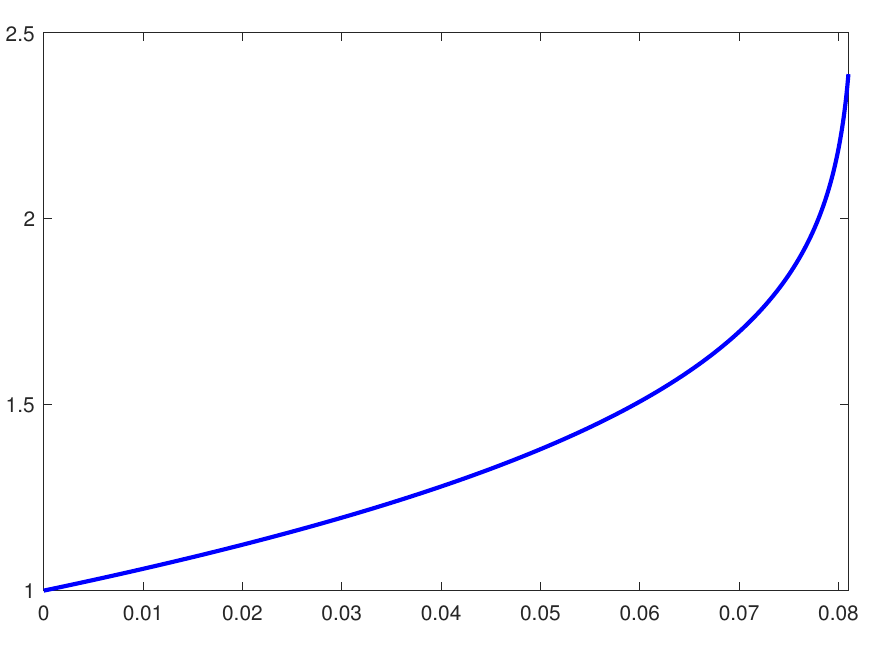}\\
	    	\includegraphics[width=0.45\textwidth]{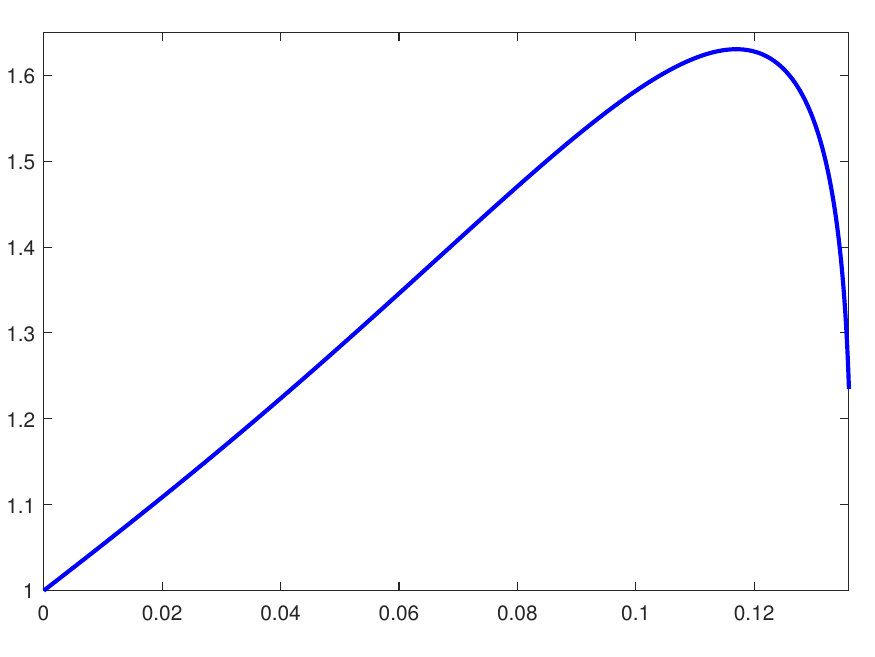}
	    	\includegraphics[width=0.45\textwidth]{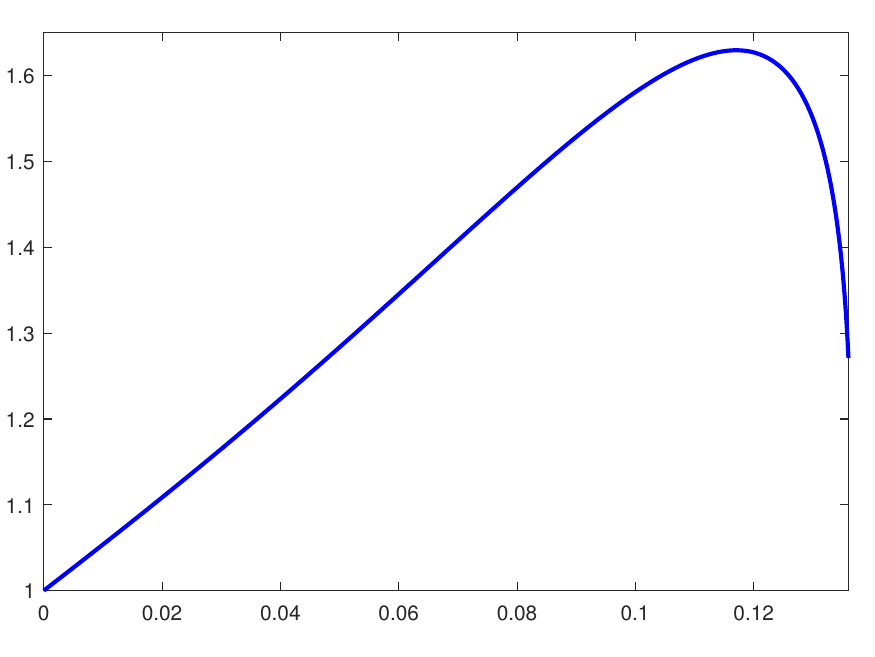}
	    	\caption{The mesh ratio $\mathcal R^m$ over time for the CN method (left column) and the BDF2 method (right column): $r=0.7$ (first row), $r=0.5$ (second row).}
	    	\label{figure:Exa2_3}
	    \end{figure}
	    \item Lastly, we numerically approximate the value of \( r_0 \) using the CN and BDF2 methods. To the best of our knowledge, the exact value of the critical radius \( r_0 \) remains unknown; however, Ishimura \cite{ishimura1993limit} and Ahara \& Ishimura \cite{ahara1991mean} have rigorously proven that \( r_0 \geq \frac{2}{3+\sqrt{5}} \approx 0.38 \).  
	    Subsequently, many researchers have numerically estimated the approximate value of \( r_0 \). For instance, Paolini \& Verdi \cite{paolini1992asymptotic} reported \( r_0 \approx 0.65 \), while Barrett et al. \cite{barrett2021finite} further refined the estimate to \( r_0 \in [0.64151, 0.64152] \). In this test, by setting \( J = 4096 \) and \( \Delta t = 5 \times 10^{-6} \), we aim to numerically compute \( r_0 \) using our proposed CN and BDF2 methods.  
	    Based on the interval \( r_0 \in [0.64151, 0.64152] \) given in \cite{barrett2021finite}, we plot the evolution of a torus with \( r=0.64151 \), \( r=0.6415125 \), \( r=0.641515 \), and \( r=0.64152 \) in Figs. \ref{figure:Exa2_4}–\ref{figure:Exa2_5}. From these results, we conclude that \( r_0 \in [0.64151, 0.6415125] \) for the CN method and \( r_0 \in [0.6415125, 0.641515] \) for the BDF2 method.

\begin{figure}[!htp]
	\centering
	\includegraphics[width=0.45\textwidth]{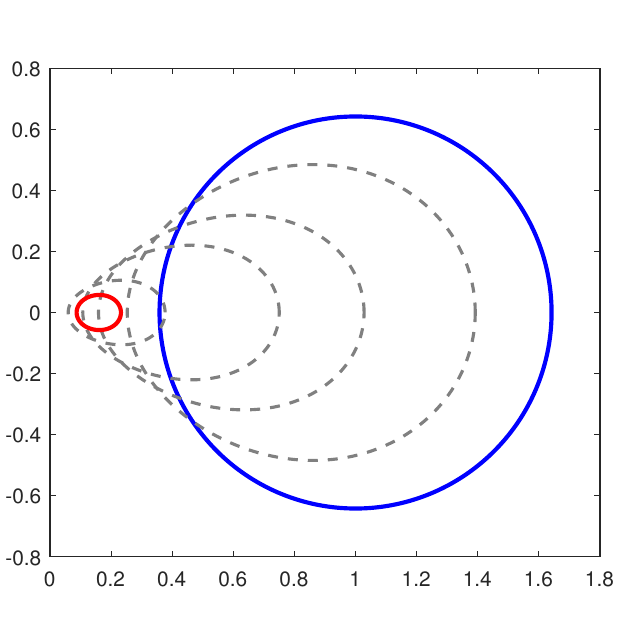}
	\includegraphics[width=0.45\textwidth]{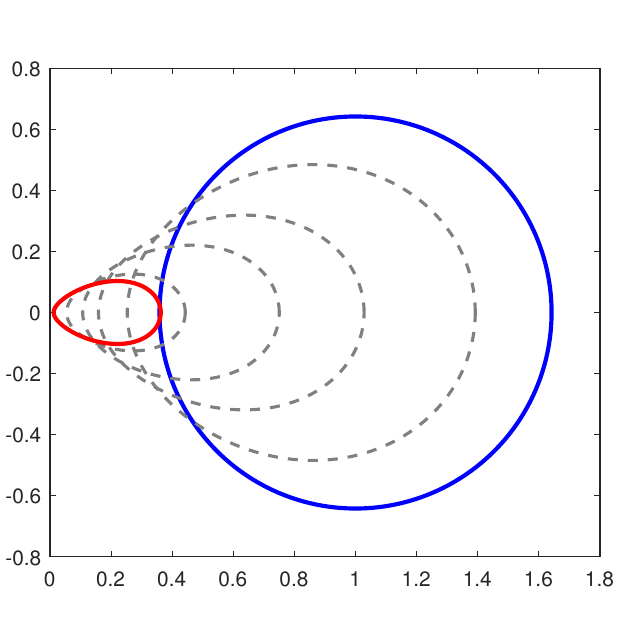}\\
	\includegraphics[width=0.45\textwidth]{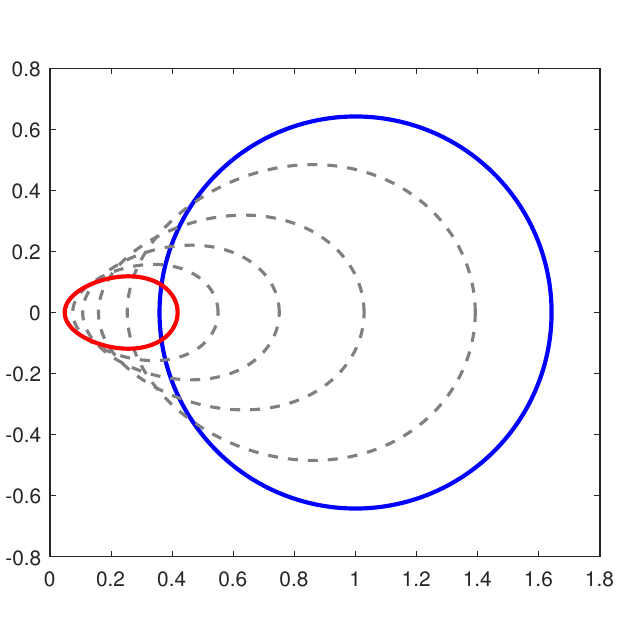}
	\includegraphics[width=0.45\textwidth]{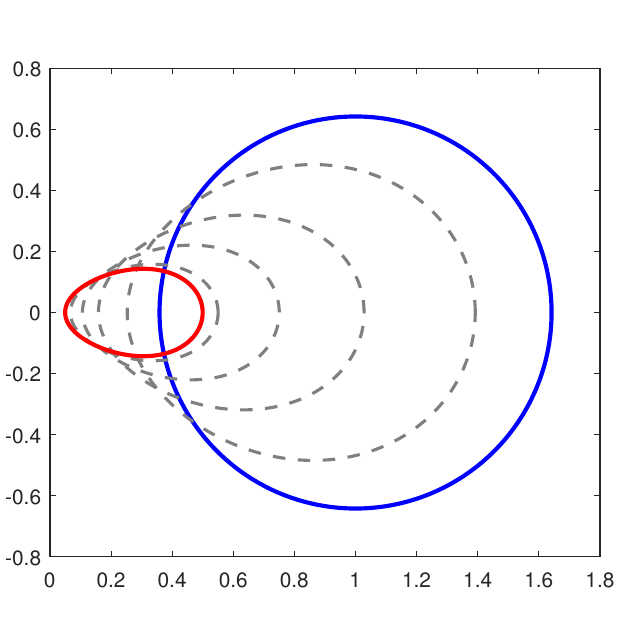}
	\caption{Evolution of a torus using the CN method with \( r=0.64151 \) (top left), \( r=0.6415125 \) (top right), \( r=0.641515 \) (bottom left), and \( r=0.64152 \) (bottom right). Plots are at times:  
	$t=0, 0.1, 0.2, 0.25, 0.285, 0.291$ (top left);
	$t=0, 0.1, 0.2, 0.25, 0.275, 0.287$ (top right);
	$t=0, 0.1, 0.2, 0.25, 0.29, 0.289 $ (bottom left);
	$t=0, 0.1, 0.2, 0.25, 0.275, 0.28$ (bottom right). } 
	\label{figure:Exa2_4}
\end{figure}

\begin{figure}[!htp]
	\centering
	\includegraphics[width=0.45\textwidth]{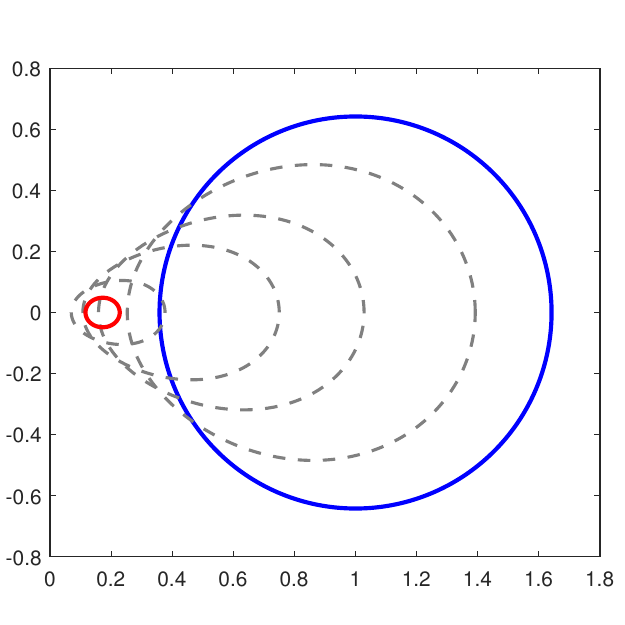}
	\includegraphics[width=0.45\textwidth]{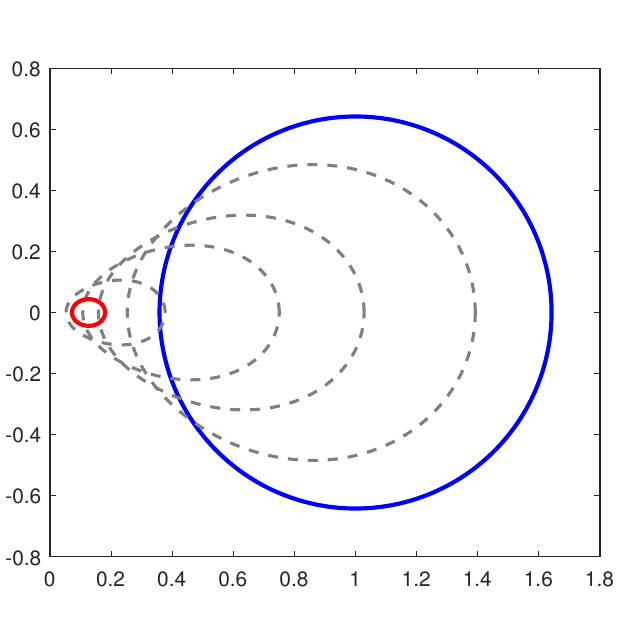}\\
	\includegraphics[width=0.45\textwidth]{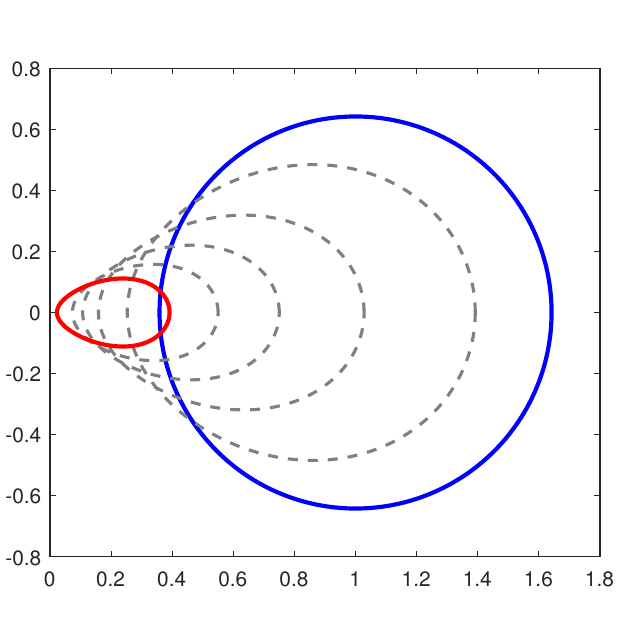}
	\includegraphics[width=0.45\textwidth]{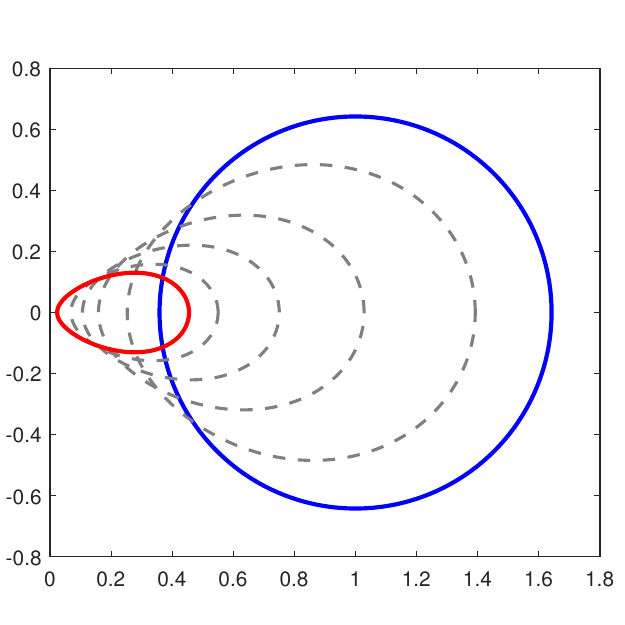}
	\caption{Evolution of a torus using the BDF2 method with \( r=0.64151 \) (top left), \( r=0.6415125 \) (top right), \( r=0.641515 \) (bottom left), and \( r=0.64152 \) (bottom right). Plots are at times:  
		$t=0, 0.1, 0.2, 0.25, 0.29, 0.3$ (top left);
		$t=0, 0.1, 0.2, 0.25, 0.275, 0.289$ (top right);
		$t=0, 0.1, 0.2, 0.25, 0.29, 0.298 $ (bottom left);
		$t=0, 0.1, 0.2, 0.25, 0.275, 0.284$ (bottom right).}
	\label{figure:Exa2_5}
\end{figure}

\end{itemize}
\end{exa}

\begin{exa}
	In this example, we examine the mean curvature flow of a genus-1 surface, generated from the initial data parameterizing a closed spiral.
We initially employ the CN method for computation. As illustrated in Fig. \ref{figure:Exa3_1}, the spiral gradually untangles, leading the surface to contract into a torus before eventually shrinking into a circle. To further explore this behavior, we increase the number of spiral layers and apply the BDF2 method, observing the same phenomenon (see Fig. \ref{figure:Exa3_2}). For this experiment we use the discretization parameters $J=512$ and $\Delta t = 10^{-4}$. 

\begin{figure}[!htp]
	\centering
	\includegraphics[width=0.85\textwidth]{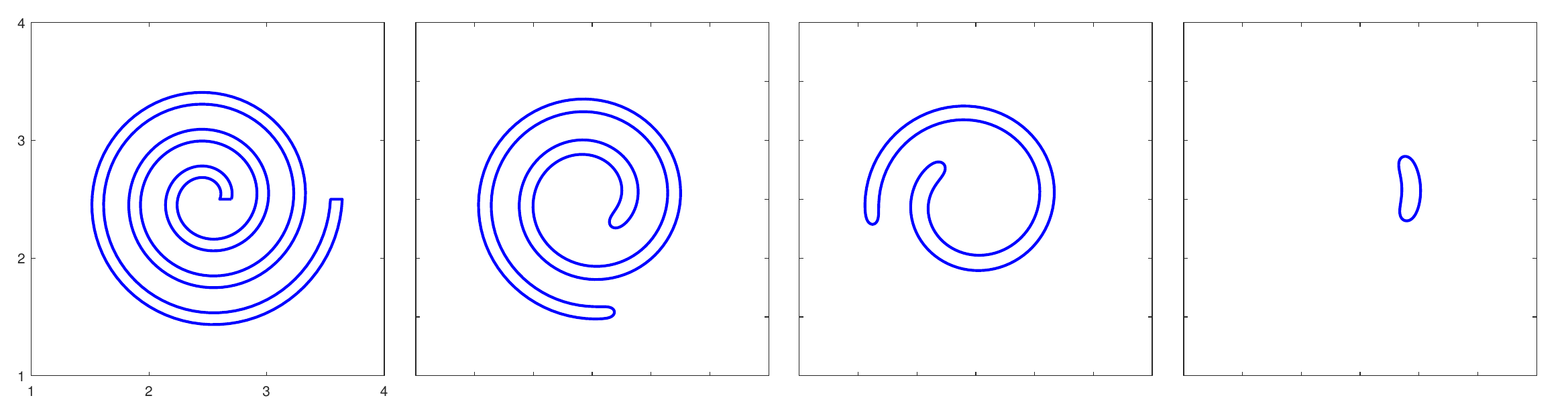}\\
	\includegraphics[width=0.85\textwidth]{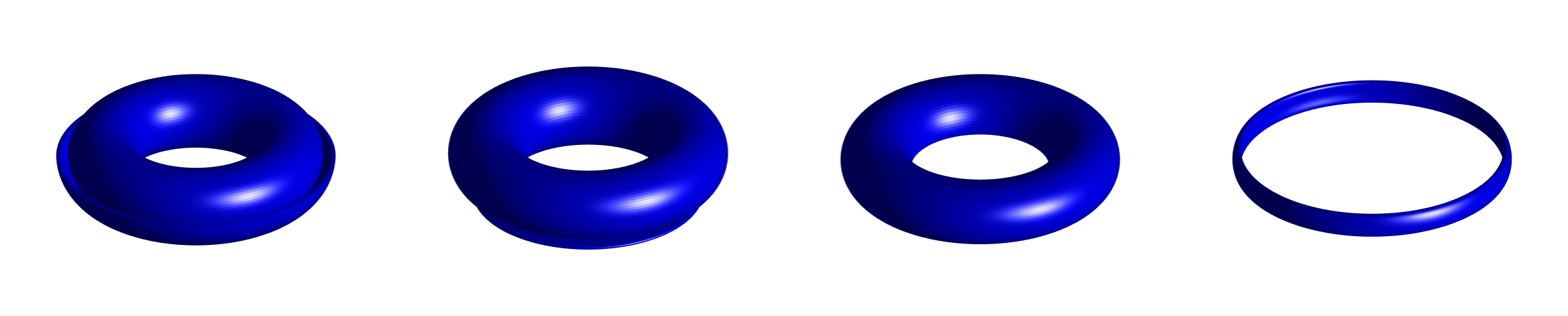}
	\caption{Evolution for a genus-1 surface generated by a spiral, with the use of the CN method. Plots are at times t = 0, 0.05, 0.1, 0.18. Below we visualize the
		axisymmetric surfaces generated by the curves.}
	\label{figure:Exa3_1}

\end{figure}

\begin{figure}[!htp]
	\centering
	\includegraphics[width=0.85\textwidth]{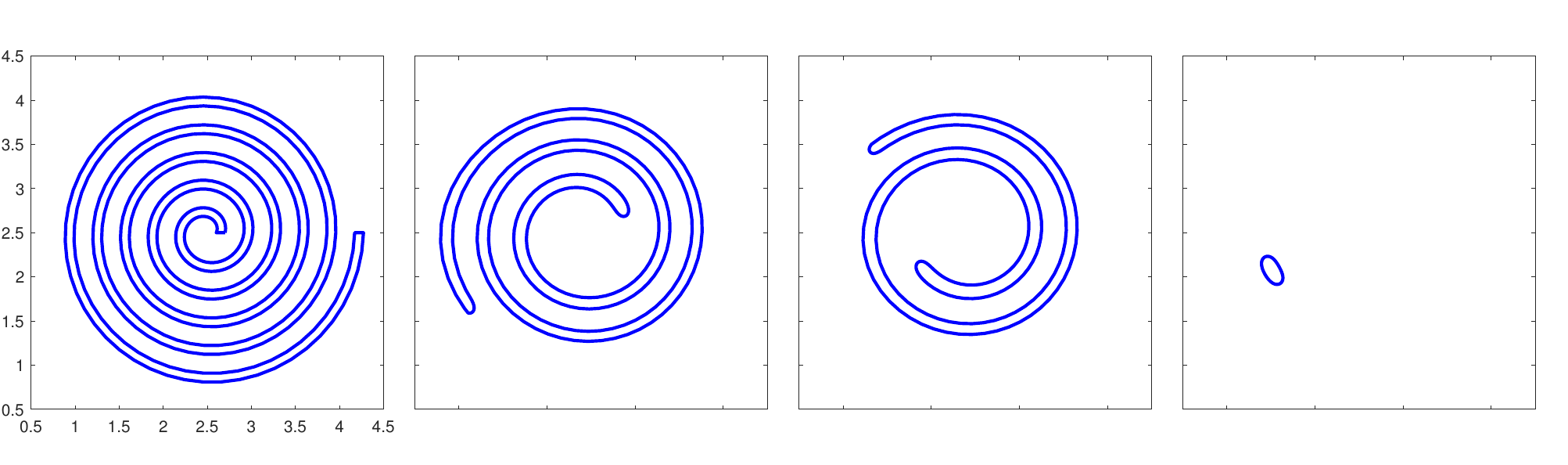}\\
	\includegraphics[width=0.85\textwidth]{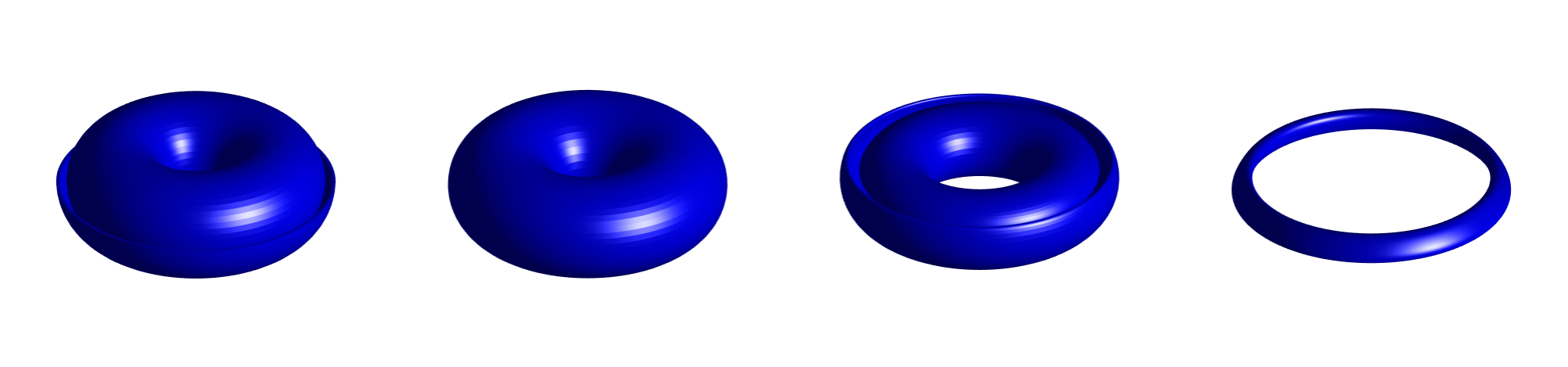}
	\caption{Evolution for a genus-1 surface generated by a spiral, with the use of the BDF2 method. Plots are at times t = 0, 0.2, 0.3, 0.54. Below we visualize the
		axisymmetric surfaces generated by the curves.}
	\label{figure:Exa3_2}
\end{figure}
\end{exa}

\begin{exa}
We conclude our example by comparing the method presented in this paper with second-order approaches based on the BGN-type method. Specifically, we compare our methods with the CN-BGN method and the BDF2-BGN method, both of which are based on the variational formulation presented in \cite[(2.19)]{Barrett2019variational}. The initial mesh will be chosen as an ellipse: \( x = 5 + \cos (2\pi\rho) \), \( y = \sin (2\pi\rho) \), and \( J = 128 \), \( \Delta t = 10^{-2} \). The comparison results are plotted in Figs. \ref{figure:Exa4_1}-\ref{figure:Exa4_2}. It is evident that the CN method and the BDF2 method proposed in this paper exhibit certain mesh advantages and ensure long-term evolution stability.  
To further validate our findings, we select a more complex initial mesh (the Rose curve): \( x = 10 + (2 + \cos(12\pi\rho))\cos(2\pi\rho) \), \( y = (2 + \cos(12\pi\rho))\sin(2\pi\rho) \). Its evolution also demonstrates the significant mesh advantages of our methods, as shown in Figs. \ref{figure:Exa4_3}-\ref{figure:Exa4_4}.

\begin{figure}[!htp]
	\centering
	\includegraphics[width=0.85\textwidth]{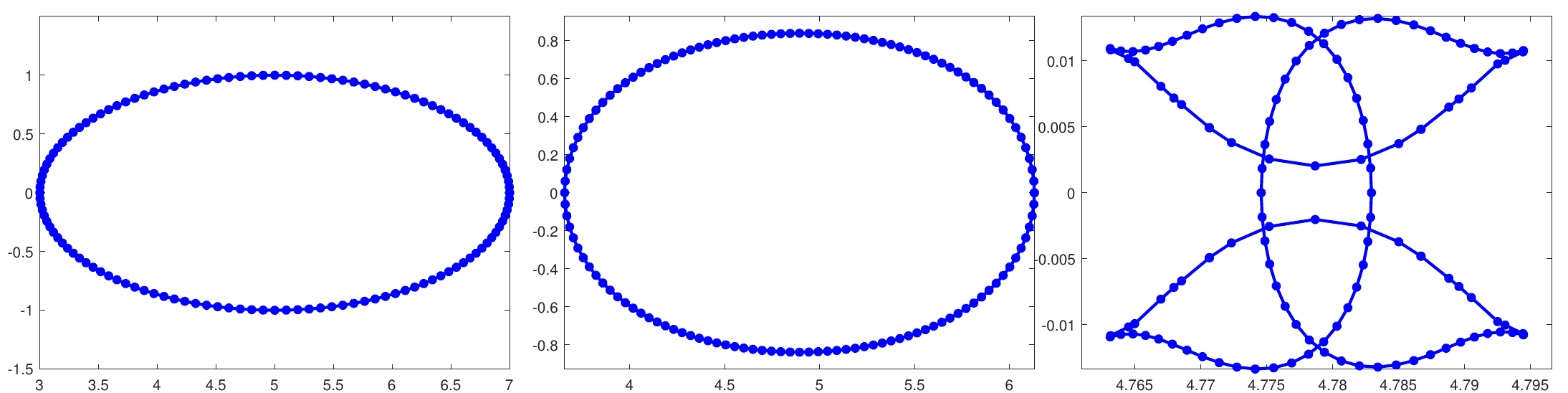}\\
	\includegraphics[width=0.85\textwidth]{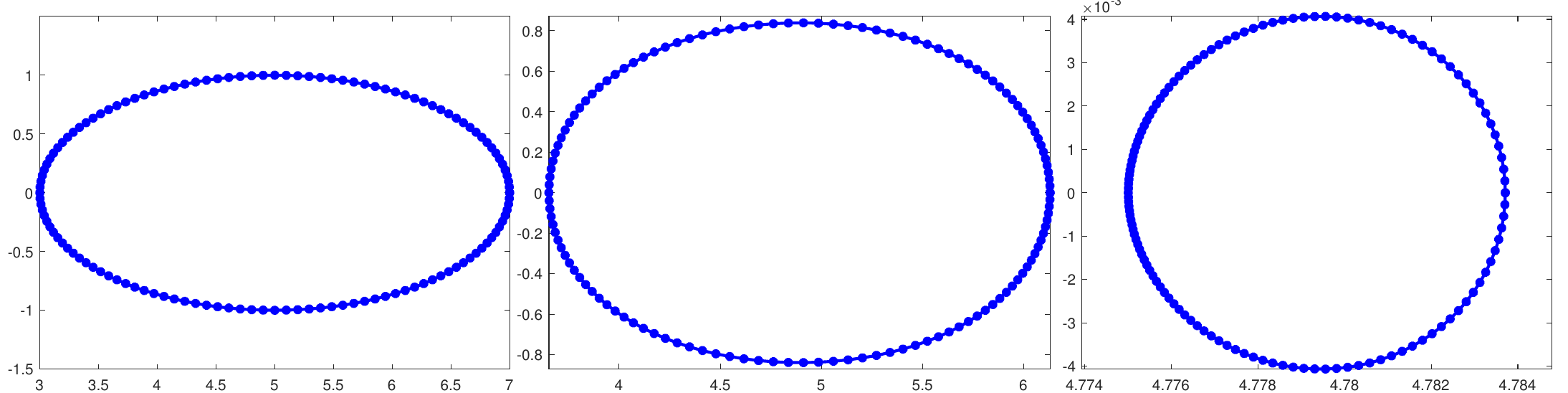}
	\caption{Evolution for a genus-1 surface generated by an ellipse, with the use of the CN-BGN method and the CN method. Plots are at times t = 0, 0.8, 1.03.}
	\label{figure:Exa4_1}
\end{figure}

\begin{figure}[!htp]
	\centering
	\includegraphics[width=0.85\textwidth]{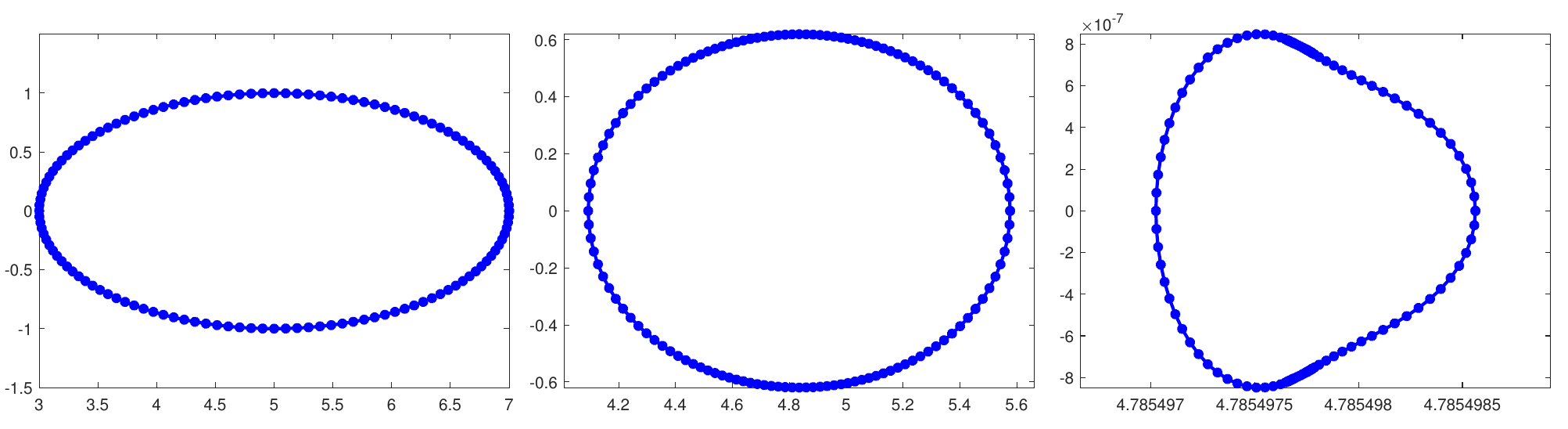}\\
	\includegraphics[width=0.85\textwidth]{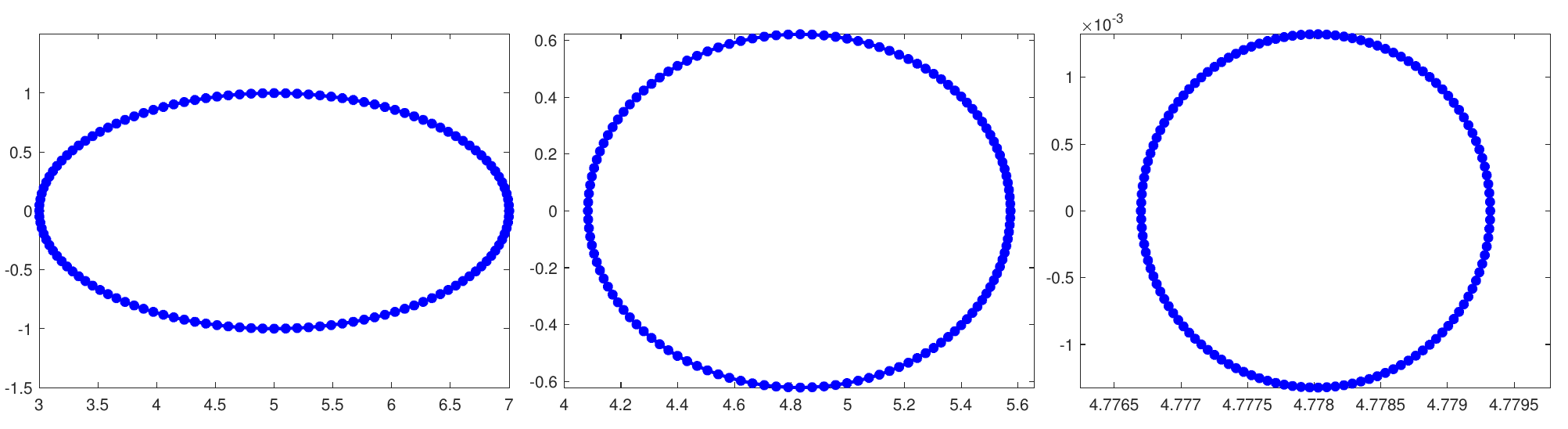}
	\caption{Evolution for a genus-1 surface generated by an ellipse, with the use of the BDF2-BGN method and the BDF2 method. Plots are at times t = 0, 0.8, 1.05.}
	\label{figure:Exa4_2}
\end{figure}

\begin{figure}[!htp]
	\centering
	\includegraphics[width=0.85\textwidth]{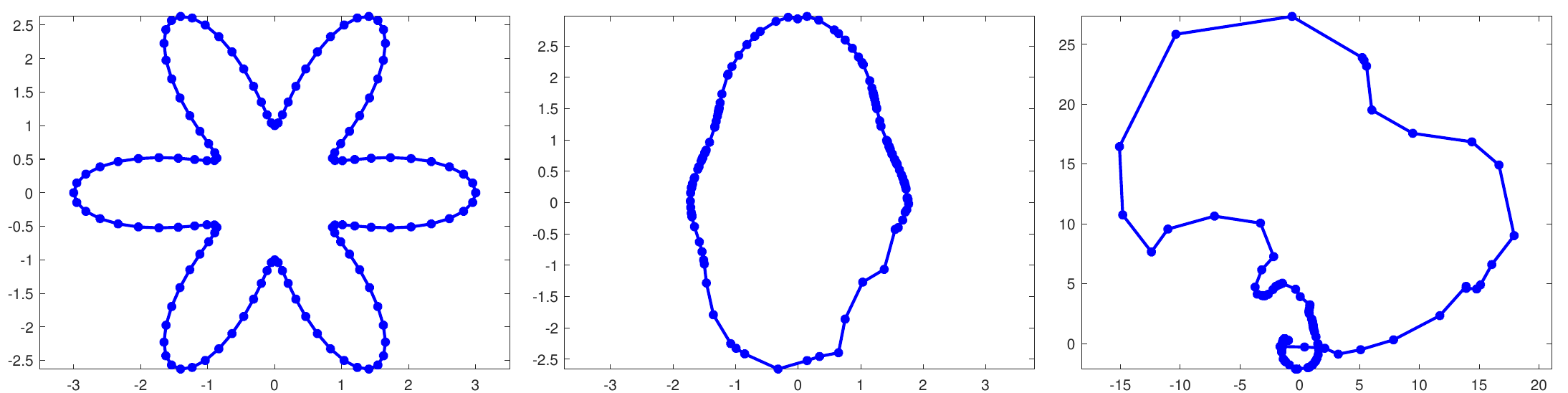}\\
	\includegraphics[width=0.85\textwidth]{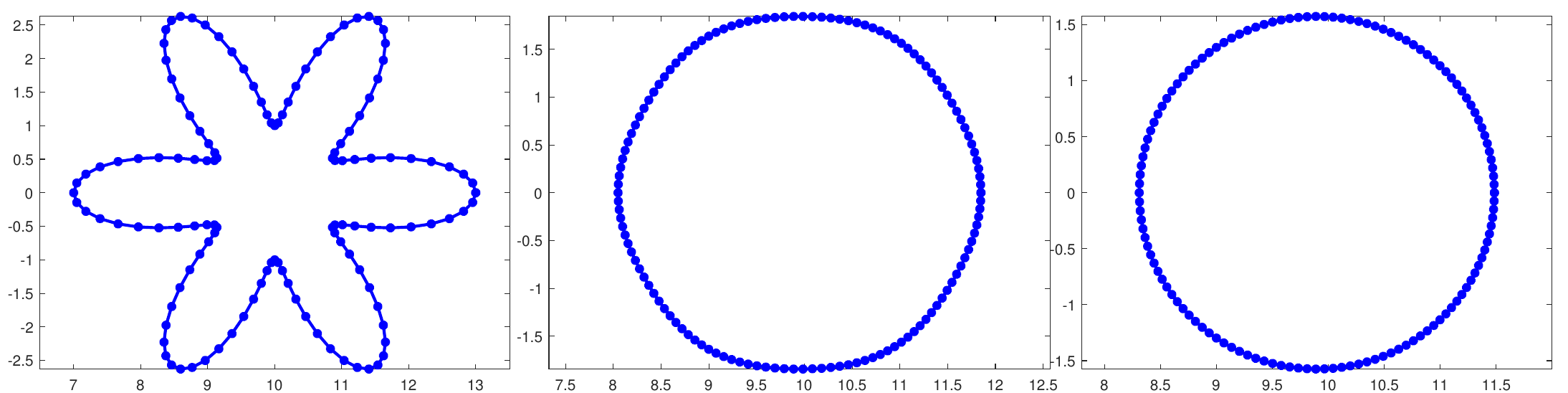}
	\caption{Evolution for a genus-1 surface generated by the Rose curve, with the use of the CN-BGN method and the CN method. Plots are at times t = 0, 0.5, 1.}
	\label{figure:Exa4_3}
\end{figure}

\begin{figure}[!htp]
	\centering
	\includegraphics[width=0.85\textwidth]{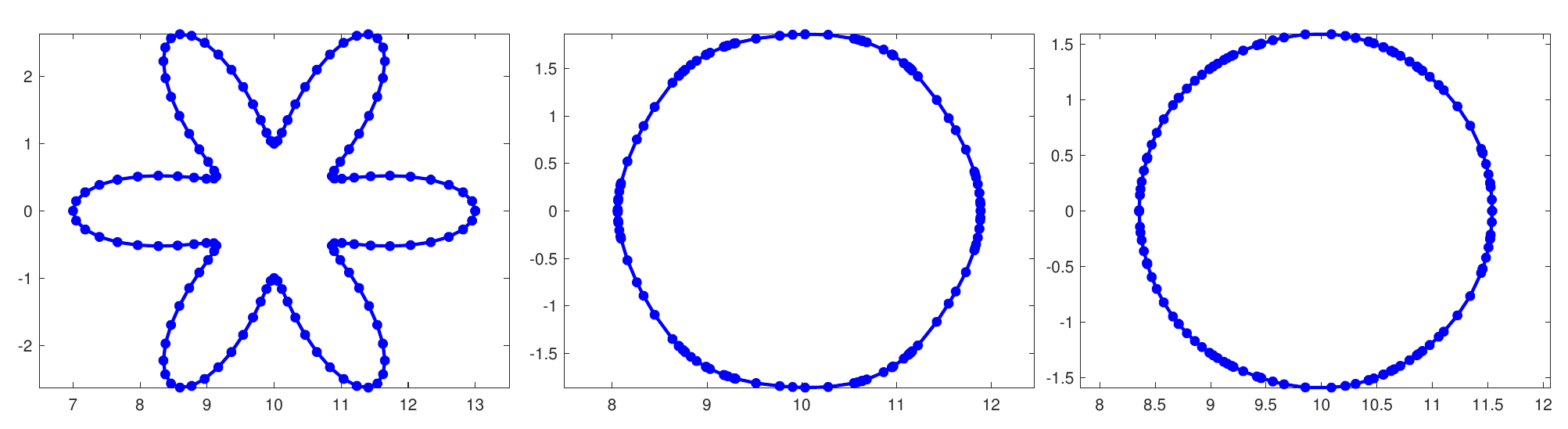}\\
	\includegraphics[width=0.85\textwidth]{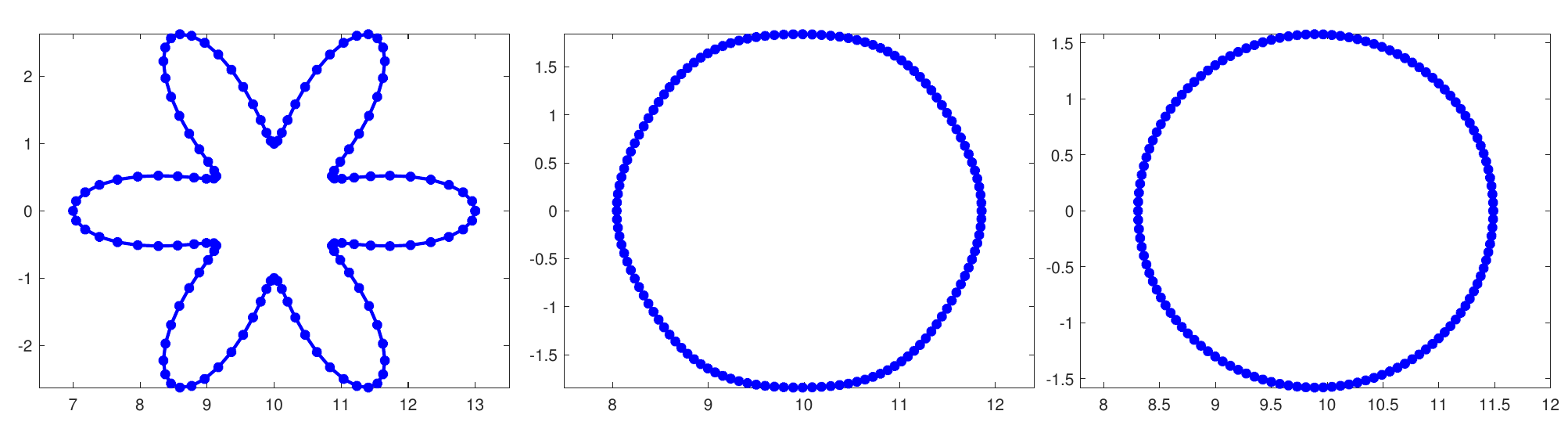}
	\caption{Evolution for a genus-1 surface generated by the Rose curve, with the use of the BDF2-BGN method and the BDF2 method. Plots are at times t = 0, 0.5, 1.}
	\label{figure:Exa4_4}
\end{figure}

\end{exa}

\section{Conclusions}\label{sec6}
In this work, we conduct an error analysis of the parametric finite element approximations for genus-1 axisymmetric mean curvature flow using two classical second-order temporal methods: the Crank-Nicolson method and the BDF2 method.  
Our results establish optimal error bounds in both the \(L^2\)-norm and \(H^1\)-norm, as well as a superconvergence result in the \(H^1\)-norm for fully discrete approximations.  
To validate our theoretical findings, we conduct convergence experiments for both the CN and BDF2 methods. Additionally, we present numerical simulations on various genus-1 surfaces, demonstrating the practical applicability of our approach. Comparisons further reveal that the second-order time-stepping schemes employed in this study offer significant advantages in mesh quality.
Our study highlights the advantages of using higher-order temporal schemes in simulating the mean curvature flow with axisymmetric structure, and provides a foundation for future research on efficient and accurate numerical methods for complex geometric evolution equations. 
In future work, we plan to extend our approach to the curvature flows with more general boundary conditions, enhance computational efficiency and robustness, and develop high-order structure-preserving temporal algorithms.

\bibliographystyle{model1b-num-names}
\bibliography{thebibx}
\end{document}